\documentclass{article}

\usepackage[margin=1in]{geometry}
\usepackage{amsmath,amsthm,amssymb,amsfonts}
\usepackage{todonotes}
\usepackage{subcaption}
\usepackage[export]{adjustbox}
\usepackage{comment}
\usepackage[percent]{overpic}
\usepackage{tikz}
\usepackage{url}
\usepackage{hyperref}

\theoremstyle{definition}
\newtheorem{theorem}{Theorem}
\newtheorem{lemma}[theorem]{Lemma}
\newtheorem{claim}[theorem]{Claim}

\newtheorem{corollary}[theorem]{Corollary}

\newtheorem{definition}[theorem]{Definition}
\newtheorem{remark}[theorem]{Remark}

\newcommand{\N}{\mathbb{N}}
\newcommand{\Z}{\mathbb{Z}}

\newcommand{\R}{\mathbb{R}}

\newcommand{\pr}{\mathbb{P}}
\newcommand{\E}{\mathbb{E}}

\newcommand{\eps}{\varepsilon}
\newcommand{\cost}[1]{\mathcal{C}(#1)}

\newcommand{\calC}{\mathcal{C}}

\newcommand{\calP}{\mathcal{P}}

\newcommand{\calV}{\mathcal{V}}

\author{Zylan Benjert \and J\'{u}lia Komj\'{a}thy \and Johannes Lengler \and John Lapinskas \and Ulysse Schaller\thanks{U.S.\ gratefully acknowledges support by the Swiss National Science Foundation [grant number
200021 192079].}}
\title{Degree-dependent and distance-dependent contact rates interpolate between explosive, exponential and polynomial epidemic growth}

\begin{document}

\maketitle

\begin{abstract}
It is a fundamental question in epidemiology to estimate, model and predict the growth rate of a pandemic. Analogously, analysing the diffusion of innovation, (fake) news, memes, and rumours is of key importance in the social sciences. The resulting epidemic growth curves can be classified according to their growth rates. These have been found to range from exponential to both faster super-exponential curves and slower subexponential or polynomial curves. 
Previous research has lacked a unified explanatory framework capable of accommodating super-exponential, (stretched) exponential, and polynomial growth patterns within the same contact network. In this paper we propose a simple agent-based network model that can capture all these phases. We provide such a framework by modelling how transmission rates depend on spatial distance and on individuals' numbers of contacts. By comparing the growth rate of spreading processes with or without degree-dependent and/or distance-dependent contact rates through data-driven and synthetic simulations on real and modelled networks with underlying geometry, we find evidence that even a `sublinear presence' of these causes may cause a significant slow down of the growth rate on the same underlying network.  
We find that the growth rate is governed by a combination of three factors: geometry, the prevalence of weak ties, and superspreaders. We confirm our results with rigorous proofs in a theoretical model, using a spatial multiscale-argument in long-range heterogeneous first passage percolation. Our results give a plausible explanation of why the consecutive waves of a single pandemic can differ in their growth even if their spreading mechanisms are similar.

\end{abstract}

Biological~\cite{pastor2015epidemic}, environmental~\cite{stone2007seasonal} and behavioural~\cite{funk2009spread} factors together determine the early growth of an epidemic. Within the framework of information propagation - encompassing phenomena such as misinformation, online memes, and viruses - information carriers serve a role equivalent to infectious agents in epidemic modelling \cite{rogers2014diffusion}. The early-stage growth dynamics of such spreading or diffusion processes are governed by individual willingness and transmission capacity, which vary significantly across populations and contexts. 
This variability in transmission characteristics provides a natural explanation for the wide spectrum of growth curves observed in diffusion phenomena.

\emph{A variety of epidemic growth curves.} It is more remarkable that different waves of the \emph{same} diffusion process may show different growth dynamics even on essentially the same underlying contact network.
To illustrate this, we take a look at the recent Covid-19 pandemic. Pavithran \emph{et.\ al.\ }\cite{pavithran2022extreme} perform statistically significant tests to show that the 3rd Covid wave in the USA, the 2nd waves in India and Italy, and the 5th wave in Japan all show super-exponential (i.e.\ faster than exponential) growth, while the 1st waves in the UK, Brazil, India and Mexico are all polynomial growth in nature. Concurrently, the 1st waves in the USA and South Africa, the 3rd wave in the UK and the 4th in Germany show pure exponential growth. It is notable here that consecutive waves in the same country show different growth classes in the USA, the UK and India. Similar observations were made by Kohanczyk \emph{et.\ al.\ }\cite{kochanczyk2020super}, with an initially shrinking doubling time of the Covid-19 pandemic in China, Italy, Spain, France, UK, Germany,
Switzerland, and New York State, indicating faster than exponential growth, and by Baruah \cite{baruah2020hyper}, where super-exponential growth of the 2nd Covid-19 wave in Russia, but exponential growth for the 1st wave, was reported. Alongside super-exponential growth curves, \emph{sub}-exponential growth curves were also observed during the Covid-19 pandemic \cite{liu2022temporal}, where the authors fitted \emph{stretched exponential} functions $\exp(c t^\gamma)$ to the epidemic curves.

Variation in the growth class of epidemic curves is not restricted to Covid-19. Viboud \emph{et.\ al.\ }\cite{viboud2016generalized} carry out a detailed statistical analysis of early epidemic growth for a range of diseases including influenza, Ebola, foot-and-mouth disease, HIV/AIDS, plague, measles, and smallpox, and find polynomial growth curves of the form $\sim (\tfrac{r}{\eta}t + A)^\eta$ for growth exponent $\eta>0$ and scaling parameters $A, r$. A major plague epidemic in Bombay in 1905, the 1918 influenza pandemic in San Francisco, and a smallpox outbreak in Khulna, Bangladesh in 1972 are shown to grow polynomially with estimated exponent $\eta\sim 0.85$, while the early growth dynamics of a Zika epidemic in Colombia exhibits much slower polynomial growth with  $\eta \in [0.4,0.7]$. Similar to Covid, the growth curves of the Ebola virus also show high variation: in districts of Margibi in Liberia, and in Bombali and Bo in Sierra Leone, the pandemic displayed near exponential growth, while the districts of Bomi in Liberia and Kenema in Sierra Leone displayed particularly slow polynomial growth with estimated exponent $\eta\sim 0.1$. An intermediate pattern of growth has been observed for the district of Western Area Urban in Sierra Leone. 
Chowell \emph{et.\ al.\ }\cite{chowell2015western, polyepidemicsurvey} draw similar conclusions,
with both slow local polynomial growth followed by an acceleration to global exponential curves for the Ebola pandemic in Western Africa. 

Intermediate growth of epidemic is often tied with fractal-like spreading patterns, as observed for the Covid-19 pandemic in mainland China  \cite{long2023multifractal}  and in \cite{PACURAR2020110073, natalia2023fractal} the authors argue that the multi-fractal nature of daily infection curves can be explained by geographic fractal patterns of alternating local and global infection dynamics. Our model below naturally recovers these fractal-like dynamics in its polynomial phase.  
To summarise, we can classify the early growth of epidemic curves of spreading processes into four universality classes: 
\begin{enumerate}
\setlength\itemsep{0em}
\item[(i)]\label{item:hyperexp} faster than exponential, i.e.\ super-exponential;
\item[(ii)]\label{item:exp} exponential: $\sim e^{C t}$
\item[(iii)]\label{item:subexp} quasi-exponential $\sim e^{Ct^\psi}$ for some $\psi<1$; 
\item[(iv)]\label{item:poly} polynomial; $\sim C t^{2\varphi}$ for some $\varphi>0$;
\end{enumerate}
where $C$ is some suitably chosen positive constant, and the factor $2$ in (iv) reflects the local $2$-dimensional nature of the Earth's surface.

\emph{Techniques for modelling spreading processes.} Modelling epidemic/diffusion processes via \emph{compartmental models} based on differential equations dates back to Bernoulli's work on small pox in 1760 \cite{bernoulli1760reflexions}, with classical references \cite{kermack1927contribution, kendall1956deterministic, ross1916application}; see also the textbook \cite{brauer2008compartmental}. While compartmental models serve as versatile tools for fitting empirical data, they fail to explain the emergence of polynomial growth patterns as opposed to exponential trajectories: in general, incorporation of node-degree heterogeneity can only be done numerically and at the cost of an enormous increase in the system's complexity leading to analytically less tractable simulation approaches (e.g.\ Individual Based Mean-Field Theory, Quenched Mean Field Approximation, NIMFA), see the review \cite{pastor2015epidemic}.

The use of random infection trees -- called branching processes in mathematical contexts, see \cite{athreya2004branching, athreya2012classical, jagers1995branching} for classical textbooks -- allows us to model even the earliest phases of diffusion processes, and provides powerful modelling tools in modern epidemiology not just to estimate the growth \cite{levesque2021model} but also to recover mutation-clusters of the virus-RNA \cite{tran2024estimating} or design effective  contact tracing methods \cite{fyles2021using, muller2023contact}. Nevertheless, random trees generally either grow super-exponentially (in case $R_0$, the average number of secondary cases, is infinite) or exponentially (if $R_0$ is finite and larger than $1$), or not at all, i.e.\ the spreading stops or the disease disappears from the population (in case $R_0\le 1$) \cite{athreya2004branching}. Modelling polynomial growth with tree-based models is difficult and only possible with population-dependent reproduction dynamics. 

Agent-based models, with agents either situated on a network or moving in continuum space, are effective simulation tools that allow for behavioural heterogeneity. When properly tuned they may produce both polynomial and exponential growth curves, see \cite{namatame2016agent}. 
 Often, these models are analytically less or entirely not tractable, and their mean-field approximations often fail to capture emergent phenomena \cite{chatterjee2009contact}.

In the network-based approach --- a special type of agent-based modelling -- individuals correspond to nodes in a network while their contacts correspond to links. The power of this approach is the  \emph{universality} phenomenon: often, the exact microscopic rules of how the network is formed or how the spreading occurs are only relevant via a few key parameters, and can lead to only a few possible macroscopic behaviours of the diffusion process \cite{bhamidi2017universality, adriaans2018weighted}. Universality thus allows us to test whether a phenomenon is robustly present. This approach was successfully applied in \cite{odor2021switchover}, where it was illustrated that geometry causes the ratio of the final size of a pandemic started from two types of initial seed sets (dispersed vs core-seeding) to undergo a `switchover', a robust phenomenon generally present in networks with a core-structure. 

\emph{Structural and behavioural factors in diffusion processes.}
There are systematic studies on the role of the underlying social network for diffusion processes \cite{keeling2005implications, newman2002spread, stegehuis2016epidemic}, yet there is no unanimous 
agreement on the effect of topological substructures on the growth of the diffusion process.  With the Covid-19 pandemic behind us, the role of individuals (nodes) with a high number of contacts (degree) in the network -- called superspreaders in epidemiological contexts and  influencers in social media contexts -- in speeding up the diffusion is now widely acknowledged \cite{illingworth2021superspreaders, rambo2021impact, sneppen2020impact, kochanczyk2020super}. The role of communities is less clear, as they may both facilitate or slow down the spread \cite{lieberthal2023epidemic, stegehuis2016epidemic}. 
Beyond such network effects, the dynamics of a diffusion process can be also immensely influenced by behavioural factors \cite{tiwari2021dynamics}, i.e., individuals adjusting their own transmission rates, either consciously or unconsciously. 

There is now significant evidence that in social networks actual contacts and transmissions do not scale linearly with the degree~\cite{feldman2017high,kroy2023superspreading,wang2022effects, ke2021vivo}, which leads to a reduction of the effective degree of superspreaders. Analogously, while the effect of weak ties (less-often used, often long-distance contacts) is well established in our social relationships, mental well-being and even in cognitive function \cite{easley2010strong,granovetter1973strength, pan2020power, sandstrom2014social}, there is no scientific agreement on exactly how the presence of these weak ties influence the growth rate of a pandemic \cite{fraser2021dual, larson2017weakness, shi2007networks}. This knowledge-gap can be explained by the variability of `how often' these weak ties are effectively being used to transmit. 
   For instance, in online social media contexts both weak and strong links of influential nodes may be used or omitted simultaneously, affecting the popularity growth of viral spreading such as memes, videos, and re-tweets: 
polynomial, exponential, and even super-exponential growth curves all have been observed \cite{wang2019model, bauckhage2015viral}.

Returning to the role of the underlying contact network and its heterogeneity, we may conclude that while the above modelling approaches are excellent for data fitting, they do not provide a plausible explanation of why a given growth curve occurs, or give guidance on which model to choose, often leading to data-driven, ad-hoc solutions. These observations illustrate the need for a simple theoretical model with a limited parameter set that can coherently account for the diverse growth regimes observed within a single network structure, encompassing polynomial, exponential, and super-exponential phases, and which can also be tuned to fit parametrised growth curves, for instance to adjust $\varphi$ in the polynomial regime $(B+At)^\varphi$ or $\psi$ in the quasi-exponential regime $\exp(C t^\psi)$.  

\section{Network-dependent transmission dynamics.} 
We propose a stochastic model of only a few parameters following the idea that the effective usage of links is decaying with spatial distance and with node degree. 
We are guided by the universality principle and we use the simplest scenario where the phenomenon appears by using an \textbf{SI} diffusion model; adaptation to more complicated network-based compartmental models should be straightforward. In the SI model, initially a single seed  node is aware of an information or called infected (has state $I$), and the rest of the nodes are susceptible to receiving the information (all have state $S$). Once a node is aware of the information, it stays aware forever. Upon receiving the information the first time at some time $t_u$, the node $u$ starts to transmit the information to each of its neighbours, using the node-to-node-dependent transmission rate $r(u,v)$ in \eqref{eq:transition-rate}. Unless a neighbouring node $v$ receives the information via another route earlier, it becomes aware of the information at time $t_u+ E_{uv}/r(u,v)$, where $E_{uv}$ encodes the randomness; an exponential random variable of mean $1$. 

For the rate $r(u,v)$ we consider the underlying network embedded in a geometric space: this could be the earth's surface in case of social networks with people's approximate location as node location; or a more hidden form of geometry as, for instance, hyperbolic geometry was observed for the internet-network~\cite{krioukov2010hyperbolic}. Each node $v$ has a corresponding spatial coordinate $x_v$, and we denote the total number of contacts (degree) of this node by $\deg(v)$. We set the transmission rate from a given node $u$ to any neighbouring node $v$ to
\begin{equation}\label{eq:transition-rate}
r(u,v) = \beta \deg(u)^{-\mu}  \deg(v)^{-\nu} \|x_u -x_v \|^{-\zeta}, 
\end{equation}
where $\|x_u-x_v\|$ denotes the spatial distance between the two nodes $u,v$. This choice
results in a mean passage time $\beta^{-1} \deg(u)^\mu \deg(v)^\nu \|x_u-x_v\|^{\zeta}$ on the link $(u,v)$ when the information is sent from $u$ to $v$.
Heuristically, $\mu>0$ causes a slow-down in the mean transmission time out of high degree vertices,  while for $\mu\in (0,1)$ the typical number of infection within unit time still grows with the degree. It scales as $\propto \deg(u)^{1-\mu}$, assuming a scale-free network model where most neighbours of $u$ possess low degree and are situated nearby. Analogously, $\nu>0$ corresponds to a slow-down in the mean infection time towards high degree nodes on any particular link, while the effective number of incoming infection within unit time again scales as $\propto \deg(u)^{1-\nu}$ with the scale-free network assumption. Whenever $\zeta>0$, long links experience a slow down in the transmission rate, intuitively corresponding to the observed effective less frequent usage of weak links \cite{kramer2014let, maness2017theory, urena2020estimating}. These choices are supported by empirical observations: very high degree nodes are often unable to effectively realise their degree when there is a `cost' associated to communication or where  time-constraints play a role~\cite{feldman2017high}, and similar penalisation has been used to model the sublinear impact of superspreaders as a function of contacts~\cite{giuraniuc2006criticality, karsai2006nonequilibrium, miritello2013time, pu2015epidemic,yang2008optimal, baxter2021degree}, also in other contexts~\cite{bonaventura2014characteristic, ding2018centrality, lee2009centrality, zlatic2010topologically, hooyberghs2010biased}. Consistent with our model, all these applications assume a polynomial dependence with exponent(s) in the range $\mu, \nu \in (0,1)$, where a high-degree vertex may cause more new infections per time than a low-degree vertex, but this effect is sublinear in the degree. Spatial penalisation  is analysed in long-range population dispersion models~\cite{chatterjee2016multiple}, modelling 
aerial and marine pathogen dispersal, plant dispersal, and human-helped dispersion
\cite{brown2002aerial, cuenca2003long, filipe2004effects, mccallum2003rates, ruiz2000global, suarez2001patterns}.
We call the SI model with rates in \eqref{eq:transition-rate} \emph{contact-dependent spreading process}. Note that the dependence is local, as it only considers the ego-networks of the two individual's involved in the transmission.
Subsampling ego-networks and their effect on diffusion processes was explored also in \cite{smith2020using}, where the focus lies on the inference of the true network from ego networks rather than how ego-networks affect the transmission dynamics. 
Since the transmission rate function in \eqref{eq:transition-rate} is decreasing with degree and distance, switching to dependent transmission times generally slows down the early growth of the pandemic. Our focus of interest is to answer the question:
\begin{align*}
\textit{When do degree- and distance-dependent transmission rates}\\ 
\textit{change the universality class of the early epidemic growth?}
\end{align*}
Or, whether and when such dependencies do not cause a qualitative change, but only result in a different scaling within the same growth class. 
In what follows we give an essentially full theoretical answer to this question. To reduce the parameter space, we consider only the case $\mu=\nu$ in \eqref{eq:transition-rate}. While this seems like a restrictive assumption, the same analysis as described below can be carried out to study the case $\mu\neq \nu$, and we do not expect to see new universality classes. 
To illustrate the validity of our theoretical results, we also carry out data driven and synthetic simulations. We start with these first.

\subsection{Gowalla geographic network dataset} Our first scenario is the publicly available Gowalla network dataset \cite{gowalla}.
The Gowalla dataset is a location based social networking website where users can opt-in to share their geographic locations upon login. The obtained friendship network is simple, undirected, consists of $n=196,591$ nodes and $950,327$ links, each node corresponding to a user and each link to a friendship between two users. There are $n = 107,092$ users with location data and $456,830$ links between them. The average degree (number of neighbours per node) $\bar d = 8.53$, the minimal degree $\delta=0$, and the maximal degree $\Delta = 9967$. 
The degree distribution is highly heterogeneous and the proportion of degree $k$ nodes can be approximated by a power function $n_k/n\propto k^{-\tau_{\mathrm{Gow}}}$, see Figure \ref{fig:gowalla-hills}. We estimate  $\tau_\mathrm{Gow} \approx 2.78$ using the Hill's estimator, see Methods and SI \cite{voitalov2019scale}. 

The spatial location of nodes spans across the globe, with most nodes located either in Europe ($\sim42.5k$) or the US ($\sim52.5k$), and some coverage of Japan and Australia with resp. $\sim 1.3k$ and $2.2k$ users.  
The link-length distribution follows a truncated power law with mean $885$ and estimated decay exponent $\tau_{\mathrm{Gow,link}}\approx 0.387$. See Figure \ref{fig:gowalla-illustrations} for visualisations.

To decide which effects are relevant in the network and which are not, we also use a random reference network model -- we shuffle the links of the dataset while keep the nodes in place to obtain a configuration model version of the network. In this reference model, link probabilities are independent of the spatial distance between nodes.

\begin{figure*}[h]
\centering 

\begin{minipage}[b]{0.60\linewidth} 
    \centering
    \begin{subfigure}[b]{0.49\linewidth}
        \includegraphics[height=3.1cm]{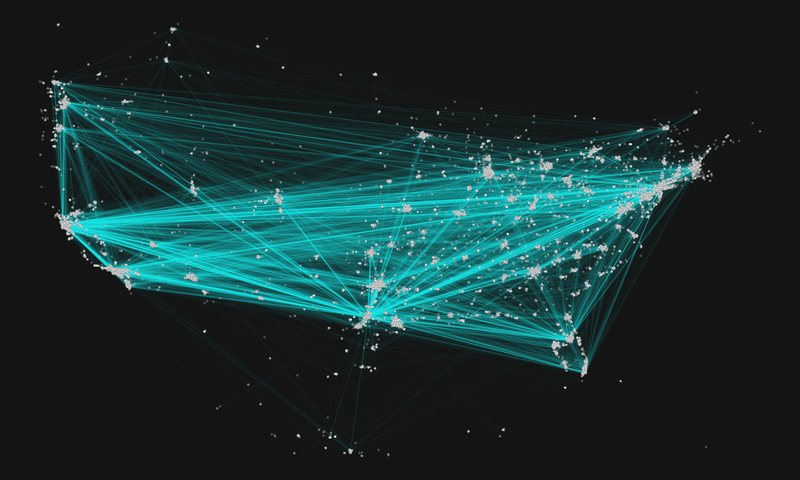}
        \caption{}
        \label{fig:gowalla-US-graph}
    \end{subfigure}
    \hfill 
    \begin{subfigure}[b]{0.49\linewidth} 
        \centering
        \includegraphics[height=3.1cm, right]{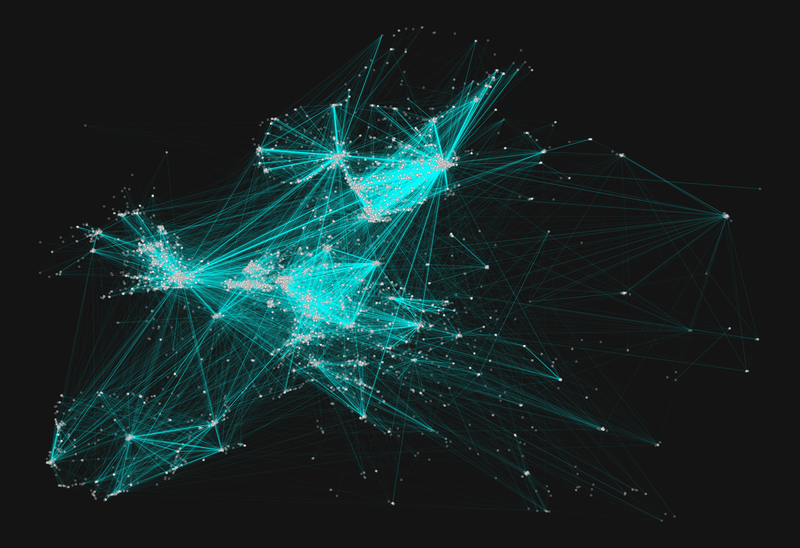}
        \caption{}
        \label{fig:gowalla-europe-graph}
    \end{subfigure}
    
    \vspace{0.3cm} 
    
    \begin{subfigure}[b]{\linewidth}
        \centering
        \includegraphics[width=\linewidth]{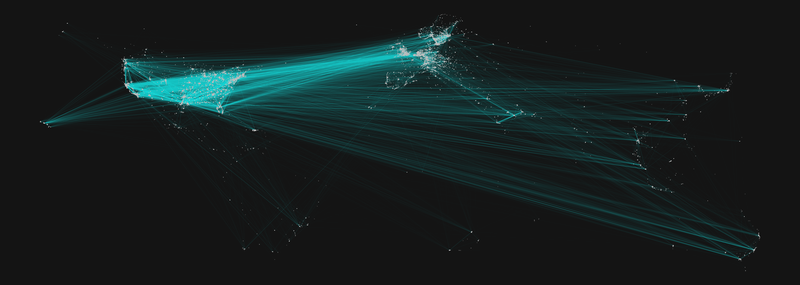}
        \caption{}
        \label{fig:gowalla-world-graph}
    \end{subfigure}
\end{minipage}
\begin{minipage}[b]{0.38\linewidth} 
    \centering
    \begin{subfigure}[b]{\linewidth} 
        \centering
        \includegraphics[width=\linewidth, trim={1cm 2.5cm 2cm 0cm}, clip]{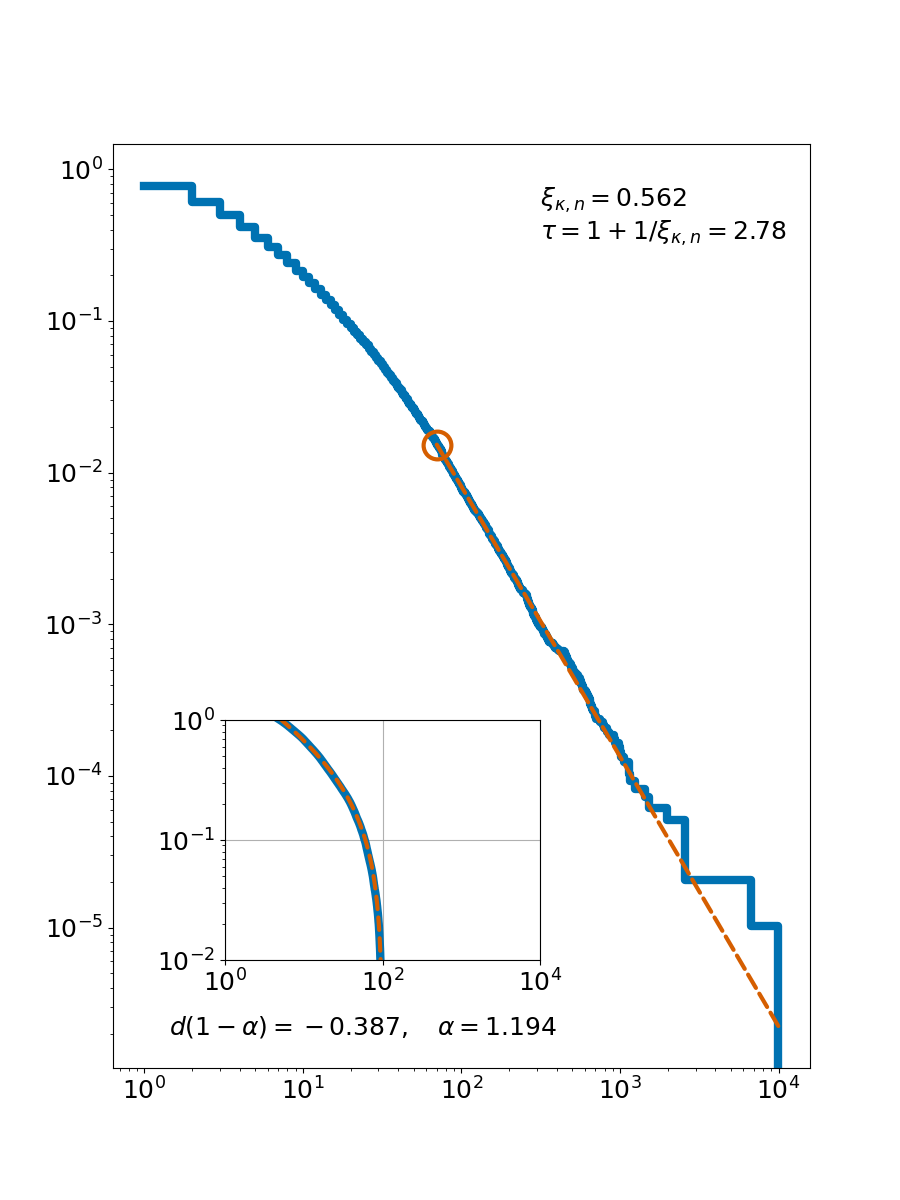}
        \caption{}
        \label{fig:gowalla-hills}
    \end{subfigure}
\end{minipage}

\caption{The Gowalla network. Visualizations of the network restricted to (a) the US and (b) Europe, as well as (c) the entire network. Figure (d) shows its cumulative degree distribution function, revealing a power-law tail. We estimated the power-law exponent using three different estimators (dashed), with consistent estimates around $1.78$, yielding $\tau_{\mathrm{Gow}} = 2.78$. The inset shows the link-length cumulative distribution, which follows a truncated power-law. We estimate the exponent using non-linear regression (that recovers the true parameter for synthetic GIRGs) on the well-behaved region of the plot from $10$km to $100$km, giving $\tau_{\mathrm{Gow,link}}=1.4 \pm 0.02$, see SI.}
\label{fig:gowalla-illustrations}

\end{figure*}

\begin{figure*}
\centering
\begin{subfigure}{0.32\textwidth}
    \includegraphics[width=\textwidth, trim = 90 90 90 90, clip]{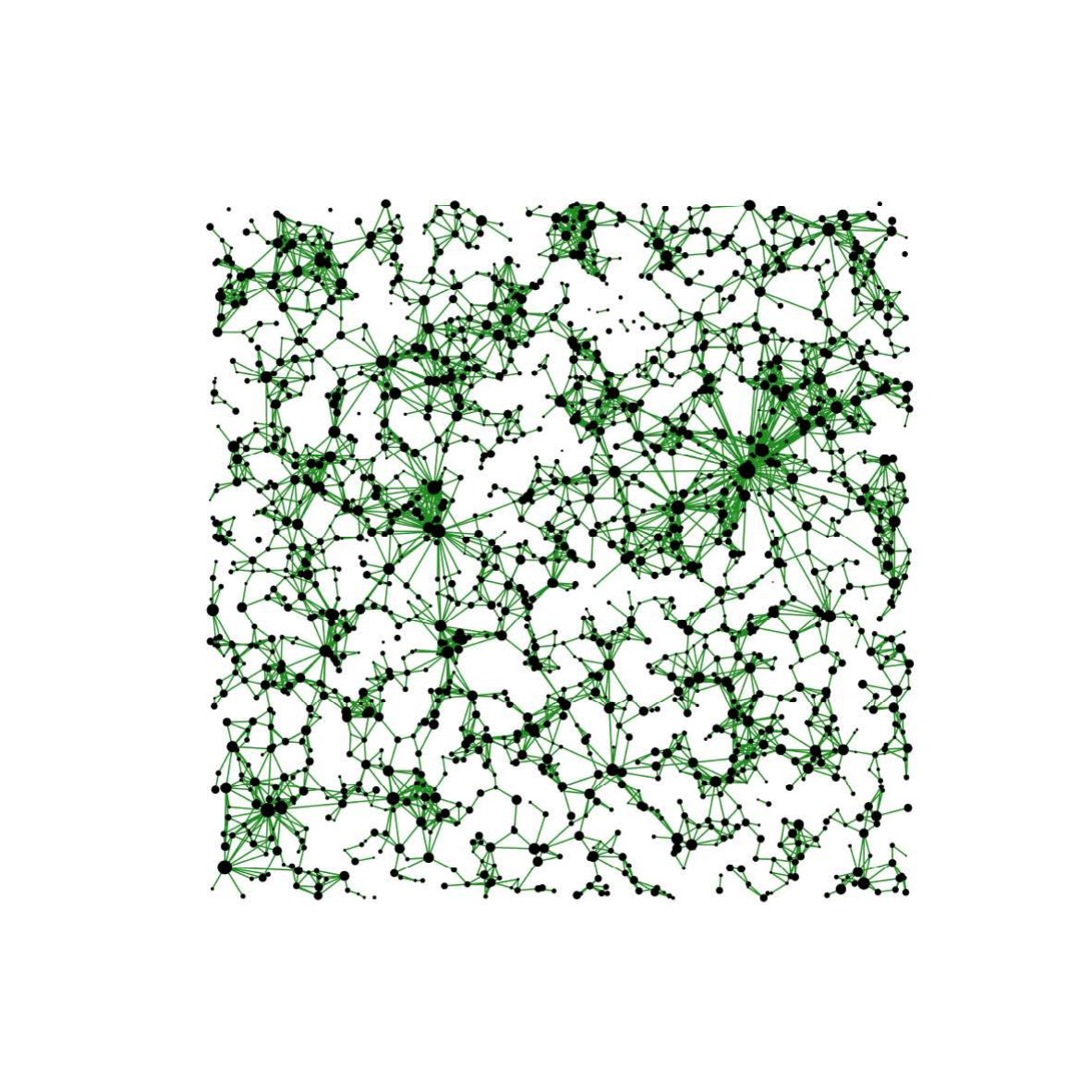}
    \caption{$\tau=3.7$, $\alpha=6$, $c_1 = 0.7$.}
    \label{fig:girg-high-tau-high-alpha}
\end{subfigure}
\begin{subfigure}{0.32\textwidth}
    \includegraphics[width=\textwidth, trim = 90 90 90 90, clip]{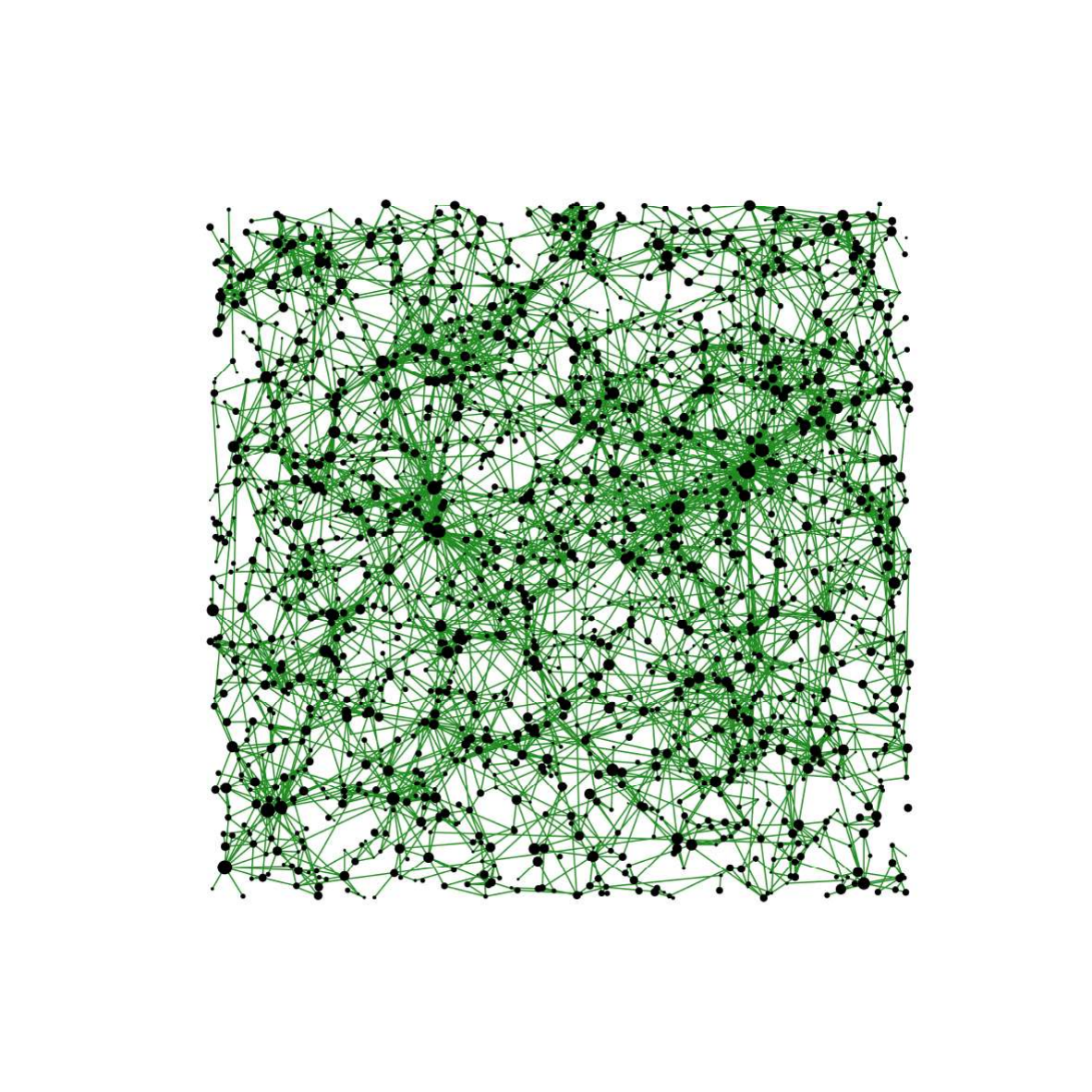}
    \caption{$\tau=3.7$, $\alpha=1.7$, $c_1 = 0.43$.}
    \label{fig:girg-high-tau-low-alpha}
\end{subfigure}
\begin{subfigure}{0.32\textwidth}
    \includegraphics[width=\textwidth, trim = 90 90 90 90, clip]{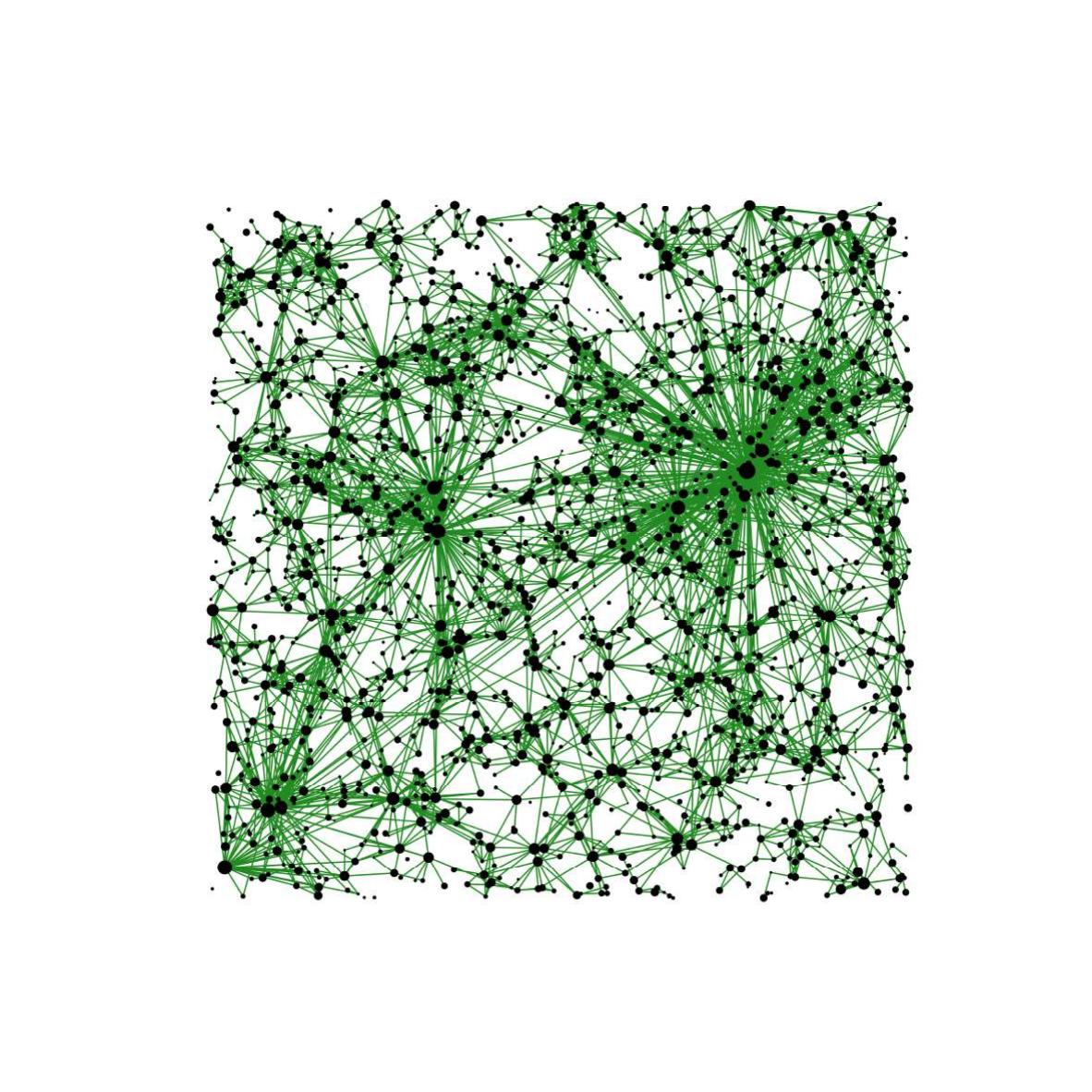}
    \caption{$\tau=2.8$, $\alpha=6$, $c_1 = 0.4$.}
    \label{fig:girg-low-tau-high-alpha}
\end{subfigure}

\caption{Illustration of the Geometric Inhomogeneous Random Graph model for three different settings of the parameters $\tau$ and $\alpha$ that each fall into a different universality class with respect to graph distance (or hop-count). On Fig.\ \ref{fig:girg-high-tau-high-alpha}, one needs to use $\propto \|u-v\|$ many links to go from node $u$ to node $v$. On Fig.\ \ref{fig:girg-high-tau-low-alpha} one needs $\propto (\log \|u-v\|)^\Delta$ many links for some $\Delta>1$: this graph has the small-world property. On Fig.\ \ref{fig:girg-low-tau-high-alpha} one needs only $\propto \log \log \|u-v\|$ links by making use of the hubs: the network is an ultra-small world. The graphs were generated by assigning a position (chosen uniformly at random in the unit square $[0,1]^2$) and a random weight $w$ (following a power-law with exponent $\tau$) to $2000$ vertices. Each pair of vertices $\{u,v\}$ draws an i.i.d.\ uniform variable $E_{uv}\in[0,1]$, and is connected by a link if $E_{uv} \le c_1 \cdot (1 \wedge c_2\cdot(w_u w_v / (\mathbb{E}[W]\|u-v\|^2))^{\alpha})$. The constant $c_2$ was set to $c_2 = 0.7$ here, while the constant $c_1$ was varied so that the three graphs all have $\sim5800$ links. The position, weights, and link-existence random variables are coupled on the three graph instances. The size of a node is proportional to its weight.}
\label{fig:girg-illustrations}
\end{figure*}

\subsection{Geometric inhomogeneous random graphs (GIRG)}
Complex networks and their dynamics can be studied using synthetic models. To study the influence of superspreaders and long-distance connections, we employ a model that couples a scale-free degree distribution with an underlying geometric space. Geometric Inhomogeneous Random Graphs (GIRGs) \cite{Sampling_GIRG} provide an analytically tractable framework for this purpose, see Methods.
GIRGs reproduce key features of many real systems: (i) highly skewed degrees, with a few hubs and many low-degree nodes; (ii) predominantly short-range connections supplemented by occasional long-range links \cite{granovetter1973strength}; and (iii) strong clustering, reflecting the tendency that “a friend of a friend is likely to be a friend’’ (Fig.~\ref{fig:girg-illustrations}).
Two parameters govern GIRG behavior. The \emph{power-law exponent} $\tau \in (2,3)$ controls the degree distribution, with $\Pr(k)\sim k^{-\tau}$; smaller $\tau$ yields heavier tails and more hubs while keeping the mean degree finite. The \emph{long-range parameter} $\alpha \in (1,\infty]$ modulates the role of geometry in link formation; the limit $\alpha=\infty$ recovers a \emph{threshold model} in which nodes connect whenever their distance is below a cutoff $r$. When $\alpha$ is close to $1$, the network contains many long links. The network exhibits a more geometric appearance, with most links being short-range as $\alpha$ increases. The parameter $\tau$ governing the degree-exponent can be directly fitted to data, while $\alpha$ can be estimated using the computed link-length distribution. In the large network limit, the edge-length distribution of GIRGs follows a power law. However, on finite GIRGs (even for $\sim 10^6$ nodes) finite size effects are non-negligible. To overcome this technical difficulty we compute the corrected finite-size edge-length distribution and use more advanced statistical methods (see Supplementary material) to estimate the long-range parameter which reliably recovers the true $\alpha$ for synthetic GIRGs.
The estimated best-fit long-range parameter is $\hat \alpha_{\mathrm{Gow}}\sim 1.2\pm 0.02$ for the Gowalla dataset.

\section{Results}
\subsection{Simulation results}
Our simulations on the Gowalla dataset confirm both our theoretical findings below in Section \ref{sec:theoretical-results} and also the empirical observations of real disease spreading.  We start the SI-epidemic process from a user in the center of Europe near the city of Nuremberg, with degree $52$.   
We set $\nu = \mu$ in \eqref{eq:transition-rate} and vary $\mu$ between $0$ and $1$, and $\zeta$ between $0$ and $3$.  We write $I(t)$ for the number of nodes reached by the SI-process by time $t$:
\begin{equation}\label{eq:It}
I(t):=\#\{ v \text{ node: $v$ infected before $t$}\}.
\end{equation}
Our first scenario is when the link-transmission rates have no degree- and spatial dependence, corresponding to $\mu = \zeta =0$. In this case the epidemic grows explosively (extremely fast, super-exponentially) according to the theoretical predictions below in Theorem \ref{theorem:summary} and shows barely any spatial correlations. This is well-reflected in the visualisation of the spread on Figure \ref{fig:heatmaps-explosive} and also the steep increase of $I(t)$ on Figure \ref{fig:epidemic-curves-explosive}.
Spatial mixing is almost immediate as the proportion of infected nodes in the US and Europe both increase rapidly after starting the process (inlay Figure \ref{fig:epidemic-curves-explosive}).  

As we increase both the degree- and spatial penalisation to $\mu\!=\!\zeta \!=\! 1$, the growth of $I(t)$ is time-scales slower, and becomes quasi-exponential. We observe more spatial correlation on the visualisations on Figure \ref{fig:heatmaps-exponential} and less spatial mixing as users from Europe initially dominate the infection process (inlay Figure \ref{fig:epidemic-curves-exponential}). We plot $I(t)$ on a log-linear scale and observe a concave curve that matches our theoretical predictions of $\log I(t)\sim t^\varphi$ for $\varphi<1$ in Figure \ref{fig:epidemic-curves-exponential}.   

Our next scenario is $\mu=1, \zeta=2$. The increase of the spatial penalization $\zeta$ causes $I(t)$ to grow polynomially. The surface of the earth can be approximated by a two-dimensional surface, so we set $d=2$ in our theoretical predictions below in Theorem \ref{theorem:summary}, and for $I(t)$ we predict  polynomial growth $I(t) \sim t^{2\psi}$ with $\psi>1$ for this value of $\mu$ and $\zeta$. This is confirmed on our regression for $\psi$ on Figure  \ref{fig:epidemic-curves-polynomial}
where we obtain $2\hat\psi=2.75$ for the estimate of polynomial growth exponent of $I(t)$. We comment here that the polynomial and pure geometric phases on Figure \ref{fig:epidemic-curves-polynomial}-\ref{fig:epidemic-curves-geometric} reflect the local geometry of the Gowalla dataset around the initial source of the infection as the curve $I(t)$ experiences a plateau around leaving the town that has $\sim 60\sim 10^{1.8}$ users. 
The visualization in Figure \ref{fig:heatmaps-polynomial} show increased spatial correlations compared to the quasi-exponential phase on Figure \ref{fig:heatmaps-exponential} and the pandemic exhibits less spatial mixing as users from Europe dominate $I(t)$ for a much longer time (inlay Figure~\ref{fig:epidemic-curves-polynomial}). 

The same effects are present in an exaggerated form as the process enters what we call the pure geometric phase when we set $\mu=1, \zeta=3$. Here our theoretical results predict polynomial growth with $I(t)\sim t^{2\psi}$  with $\psi=1$ and purely spatial growth of the infected region. Our visualizations confirm this in Figure \ref{fig:heatmaps-geometric} and in the inlay of Figure \ref{fig:epidemic-curves-geometric}. 
A linear regression gives $2\hat \psi=2.37$ after excluding the initial start-up phase and the saturating phase of the process. As  the node-set of the Gowalla network is highly inhomogeneous in space, this is within error range of our theoretical predictions below, valid for networks with homogeneously scattered node-sets.

\begin{figure*}
    \centering
\begin{subfigure}[b]{0.24\textwidth}
         \centering
         \includegraphics[width=\textwidth, clip=true]{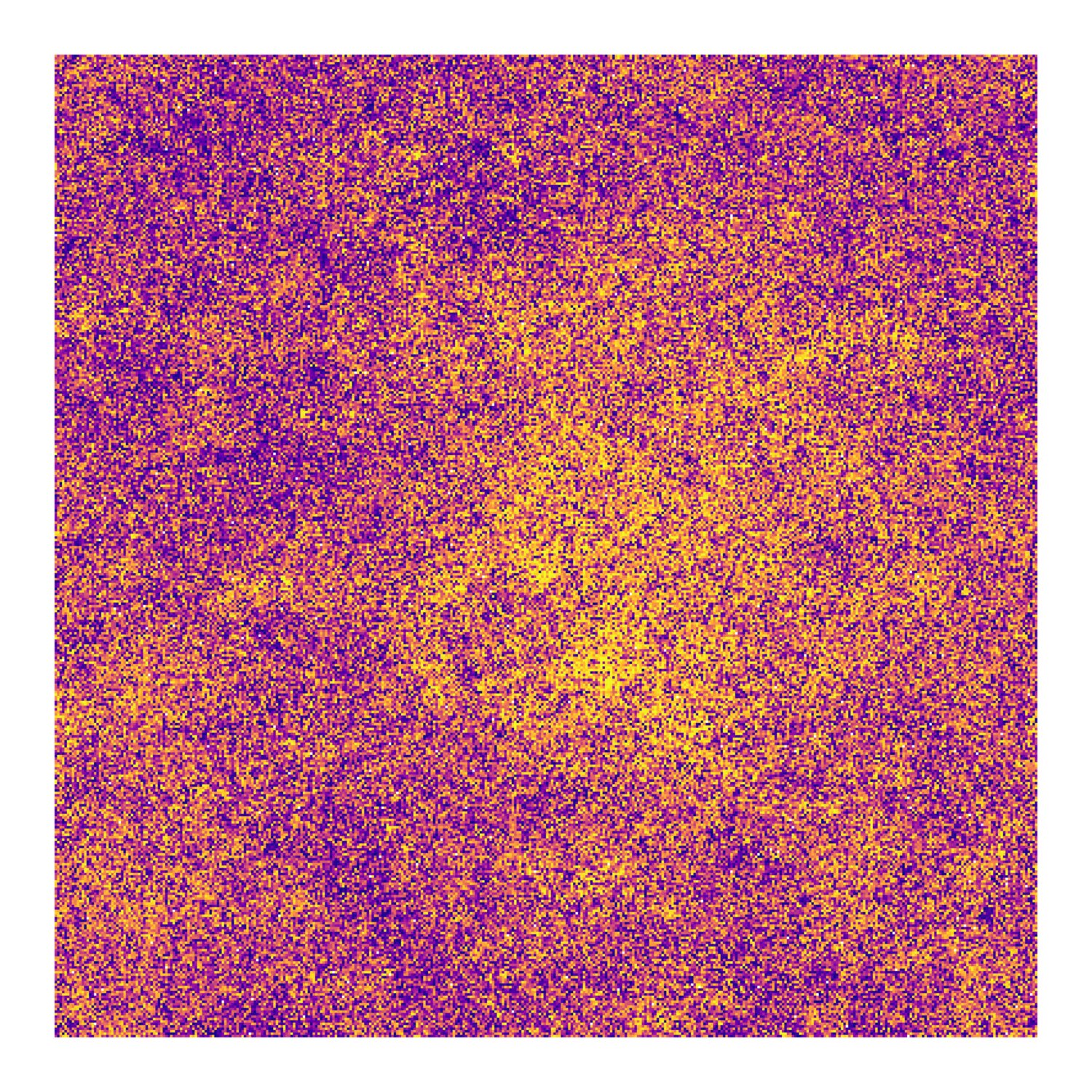}
     \end{subfigure}
     \hfill
     \begin{subfigure}[b]{0.24\textwidth}
         \centering
         \includegraphics[width=\textwidth, clip=true]{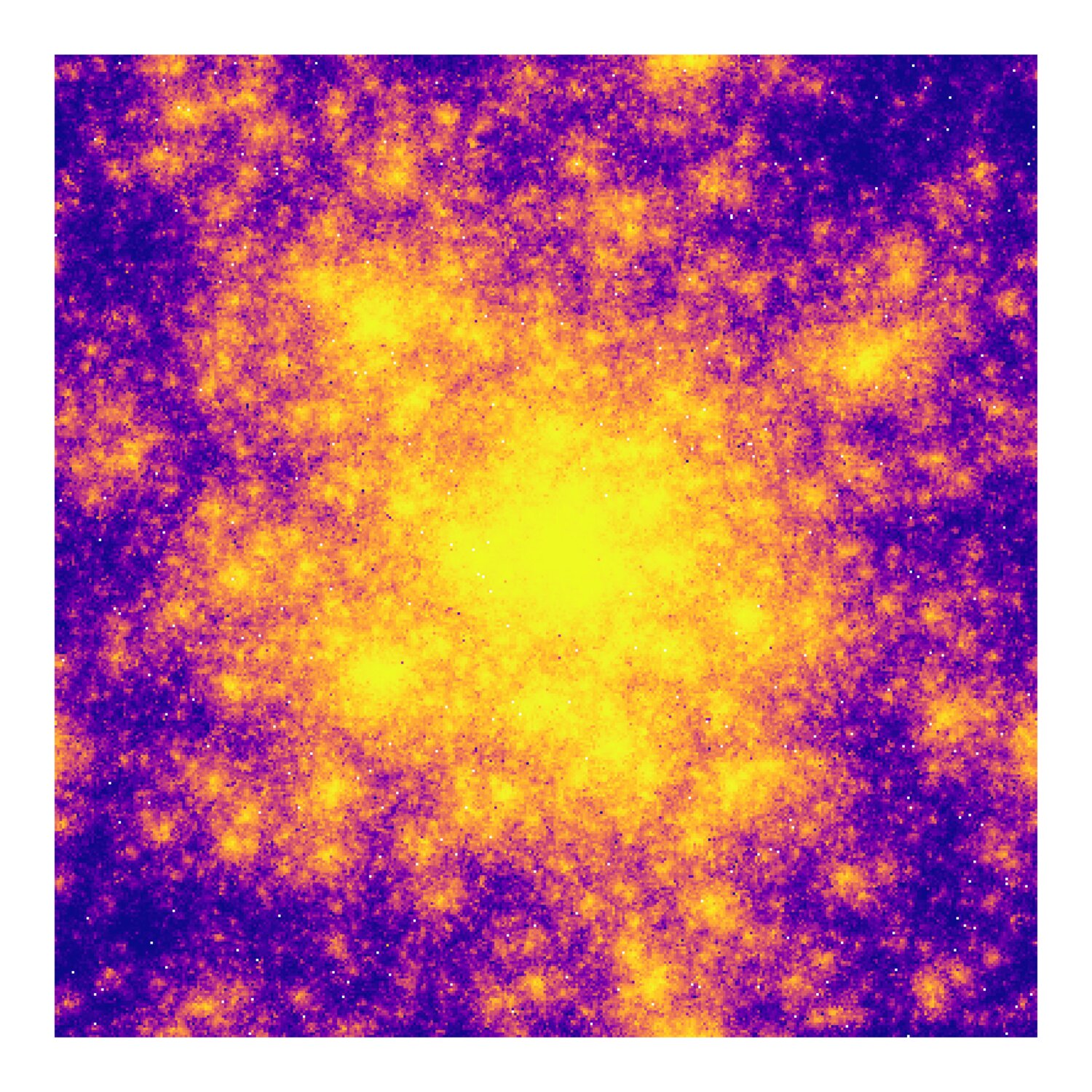}
     \end{subfigure}
     \hfill
     \begin{subfigure}[b]{0.24\textwidth}
         \centering
         \includegraphics[width=\textwidth, clip=true]{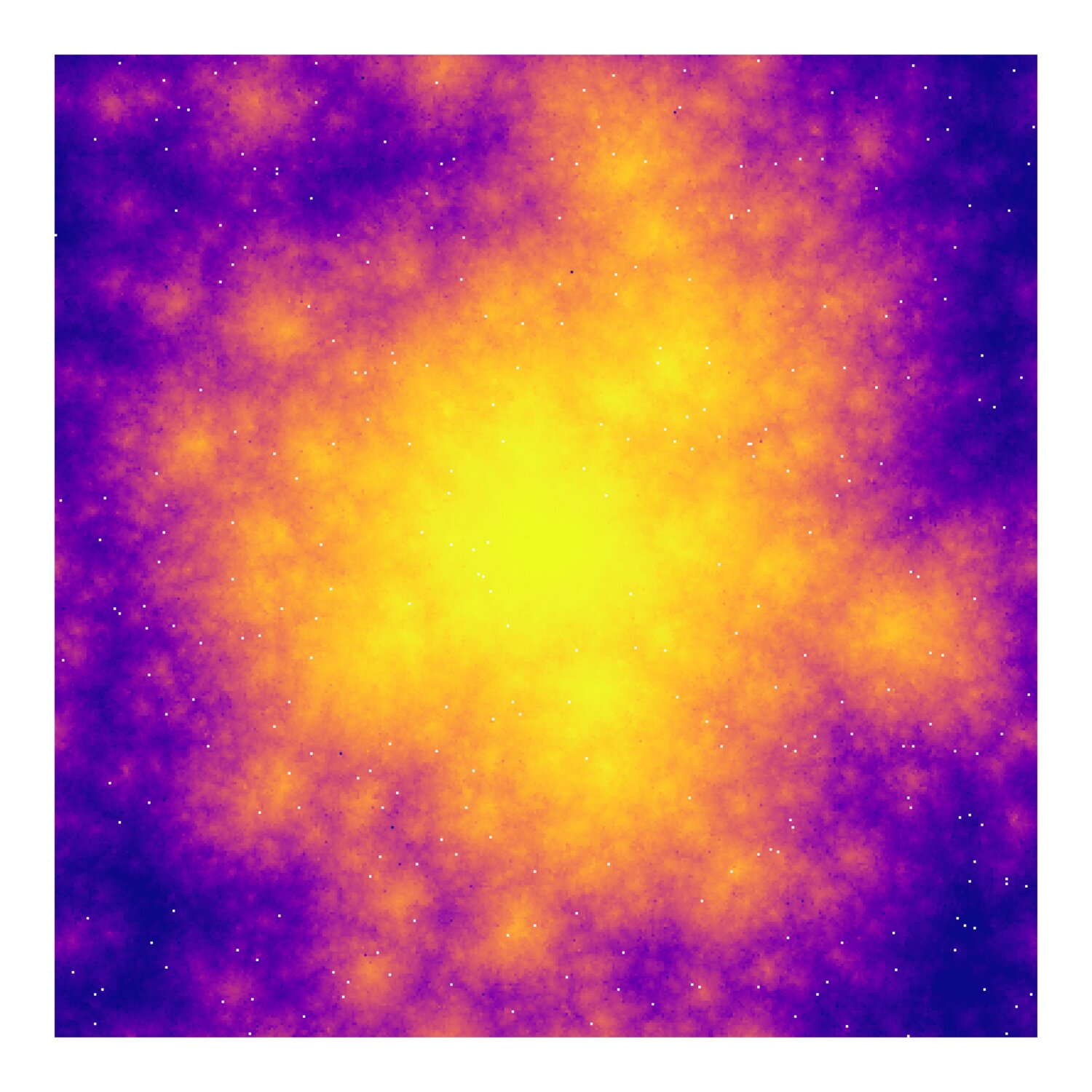}
     \end{subfigure}
     \hfill
     \begin{subfigure}[b]{0.24\textwidth}
         \centering
         \includegraphics[width=\textwidth, clip=true]{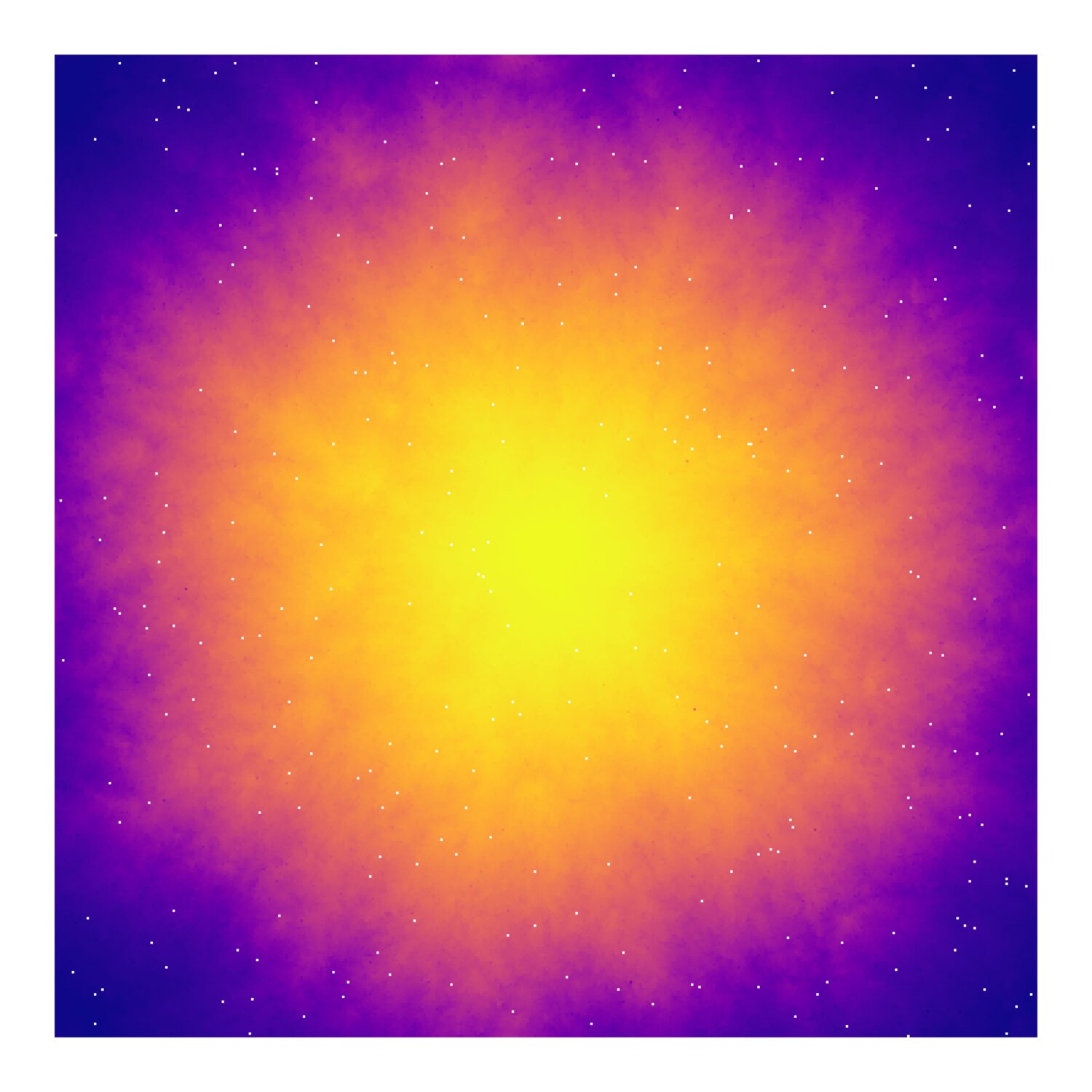}
     \end{subfigure}

     \begin{subfigure}[b]{0.24\textwidth}
         \centering
         \includegraphics[width=\textwidth, clip=true]{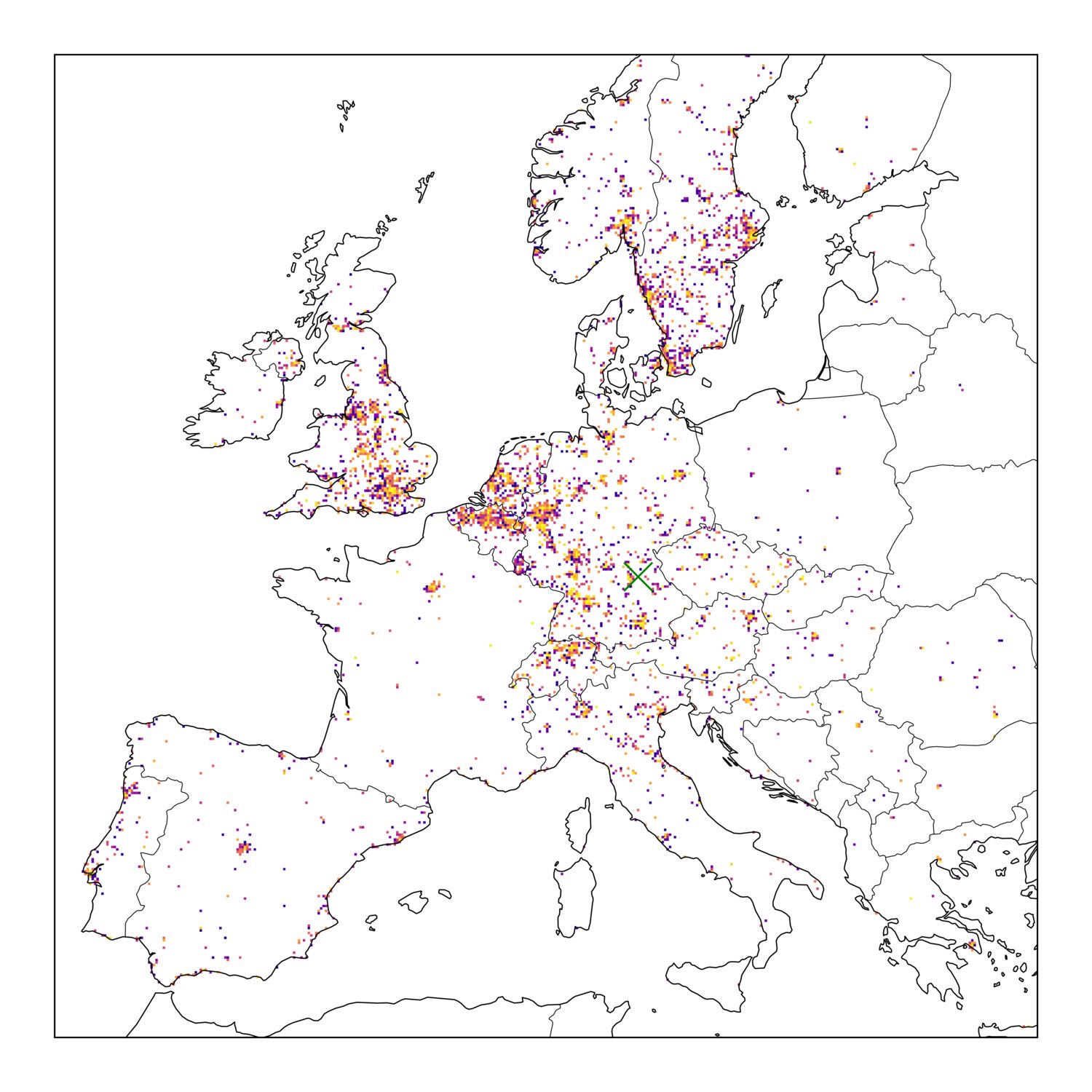}
         \subcaption{$\mu=\zeta=0$}
         \label{fig:heatmaps-explosive}
     \end{subfigure}
     \hfill
     \begin{subfigure}[b]{0.24\textwidth}
         \centering
         \includegraphics[width=\textwidth, clip=true]{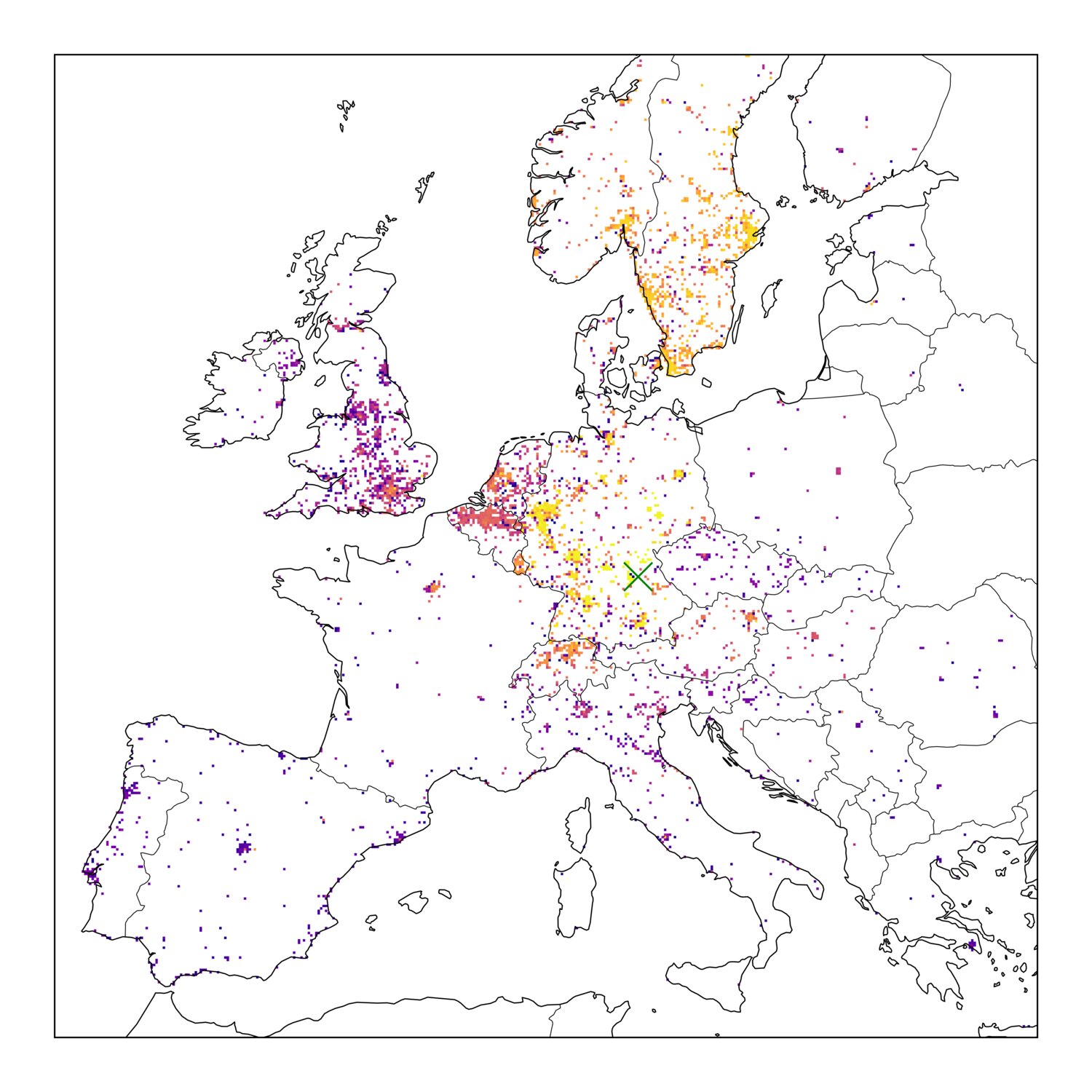}
         \subcaption{$\mu=\zeta=1$ }
         \label{fig:heatmaps-exponential}
     \end{subfigure}
     \hfill
     \begin{subfigure}[b]{0.24\textwidth}
         \centering
         \includegraphics[width=\textwidth, clip=true]{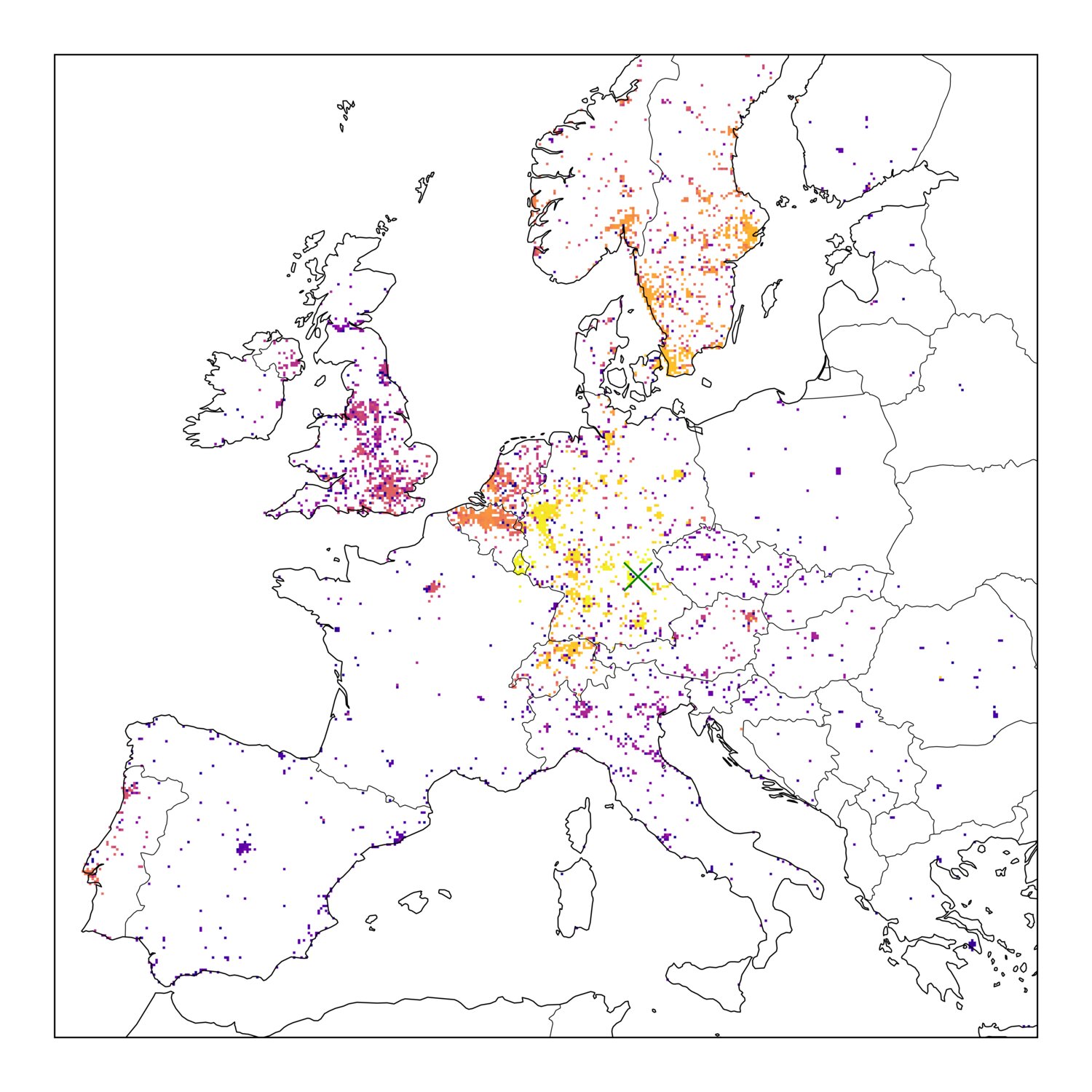}
         \subcaption{$\mu=1, \zeta=2$}
         \label{fig:heatmaps-polynomial}
     \end{subfigure}
     \hfill
     \begin{subfigure}[b]{0.24\textwidth}
         \centering
         \includegraphics[width=\textwidth, clip=true]{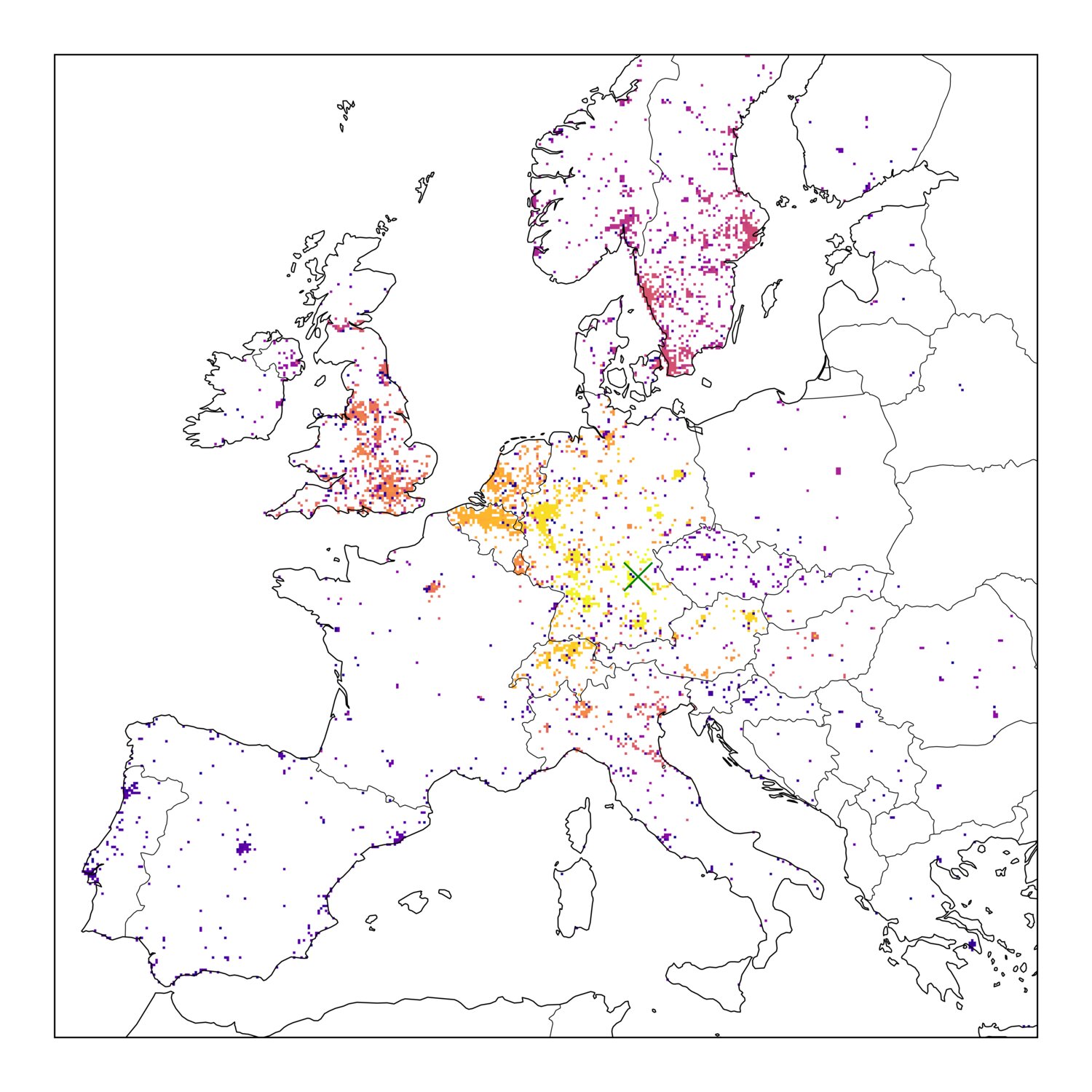}
         \subcaption{$\mu = 1, \zeta = 3$}
         \label{fig:heatmaps-geometric}
     \end{subfigure}
  \caption{
    Heatmaps of the epidemic spread: four different universality classes. Bottom row: Gowalla network with source node (indicated by `x', close to Nuremberg, Germany). Top row: synthetic GIRG network on $1m$ nodes with the torus geometry with $\tau=\tau_{\mathrm{Gow}}=2.78$ and $\alpha=\alpha_{\mathrm{Gow}}=1.2$. Source node  at the origin. Nodes are colored by their infection times: yellow infected first, then orange, then purple (see more details in SI). The heatmaps on GIRG networks use the same underlying graph. The random factors $E_{uv}\sim \mathrm{Exponential}(1)$ associated to each link are identical across all four plots (both for Gowalla and GIRG). We vary only the `penalty parameters' $\mu$ and $\zeta$ in \eqref{eq:transition-rate}: (a) $\mu\!=\!\zeta\!=\!0$ explosive growth (b) $\mu\!=\!\zeta\!=\!1$ quasi-exponential growth (c) $\mu\!=\!1, \zeta\!=\!2$ polynomial growth (d) $\mu\!=\!1, \zeta\!=\!3$ geometric growth.
    }
    \label{fig:heatmaps}
\end{figure*}
\emph{Epidemic curves on random reference networks.} To illustrate the importance of the underlying geometry on $I(t)$, we tested the growth also on random reference networks obtained from the Gowalla dataset, where we kept node-degrees and spatial locations of nodes but reshuffled the links using a configuration model network on the node set. After this transformation, almost all links become weak ties that span large distances, see SI. Due to this effect, increasing $\zeta$ in the transmission rates \eqref{eq:transition-rate} causes a general slow-down in the growth of the pandemic and a shift of $I(t)$ while maintaining the same shape. Increasing $\mu$ in \eqref{eq:transition-rate} still has an effect on $I(t)$. 
When $\mu=0$ the pandemic curve $I(t)$ grows explosively (extremely fast), corresponding to 
theoretical predictions \cite{komjathy2021penalising}. For higher values of $\mu=1$, (and $\zeta=1,2,3$) the growth becomes exponential, i.e., $I(t) \sim e^{C t}$ with the coefficient $C$ depending on $\mu$, see Figure S6, SI for the epidemic curves. The polynomial and pure geometric phases disappear as these were caused by geometry which is destroyed by the re-shuffling of links. This is in line with theoretical results of pandemic growth on configuration model networks, which can be either explosive or show exponential early growth \cite{adriaans2018weighted, fransson2023stochastic}.

\emph{Simulation results on synthetic GIRG networks.} 
Compared to the Gowalla dataset, the nodes of GIRG networks are more homogeneously scattered in space (Figure \ref{fig:girg-illustrations}). For our experiments we generated GIRGs with parameters that we estimated from the Gowalla dataset, namely degree power-law exponent $\tau=2.78$ and long-range parameter $\alpha=1.2$. As GIRG networks mimic the properties of spatially embedded complex networks, we expect similar behavior for $I(t)$ as for the Gowalla dataset. Indeed, as we vary $\mu$ and $\zeta$ analogously to our experiments on the Gowalla dataset, we see the same phases appearing on these synthetic networks. When $\mu=\zeta=0$, $I(t)$ is explosive, shows high variation in its kick-off phase, and there is almost no dependence between distance from source and infection time, see Figure \ref{fig:heatmaps-explosive}. As the penalisation increases to $\mu=\zeta=1$, the quasi-exponential phase has a pronounced concave log-linear plot of $I(t)$ on Figure \ref{fig:epidemic-curves-exponential} following theoretical predictions below in Theorem \ref{theorem:summary}. The polynomial and pure geometric phases with $\mu=1$, and respectively $\zeta=2$ and $\zeta=3$ in Figure \ref{fig:heatmaps-polynomial}-\ref{fig:heatmaps-geometric} show pronounced dependence of infection times on spatial distance from the source and linear curves of $I(t)$ on log-log plots. For the polynomial phase we estimate $2\psi=3.46$ which satisfies $\psi>1$ as predicted. Due to finite-size effects, this value is lower than what Theorem \ref{theorem:EXACT} predicts, (see SI). 

The pure geometric phase shows almost circular growth on Figure \ref{fig:heatmaps-geometric} and a corresponding estimated $I(t)\sim t^{2.47}$ which is close to the quadratic theoretical predictions. GIRG networks show a high degree of \emph{universality}: changing the parameters $\tau, \alpha$ changes the boundaries of the four growth phases with respect to $\mu, \zeta$, while for each scale-free GIRG network with $\tau\in (2,3)$ all four phases may occur for appropriately chosen values of $\mu,\zeta$ in \eqref{eq:transition-rate}. The visualizations and the epidemic curves in Figures \ref{fig:heatmaps}-\ref{fig:epidemic_curves} remain qualitatively similar within each growth phase. This leads us to theoretical investigations.

\subsection{Theoretical results} 
\label{sec:theoretical-results}

\begin{figure*}
    \centering

    \begin{subfigure}[b]{0.24\textwidth}
         \centering
         \includegraphics[width=\textwidth, clip=true]{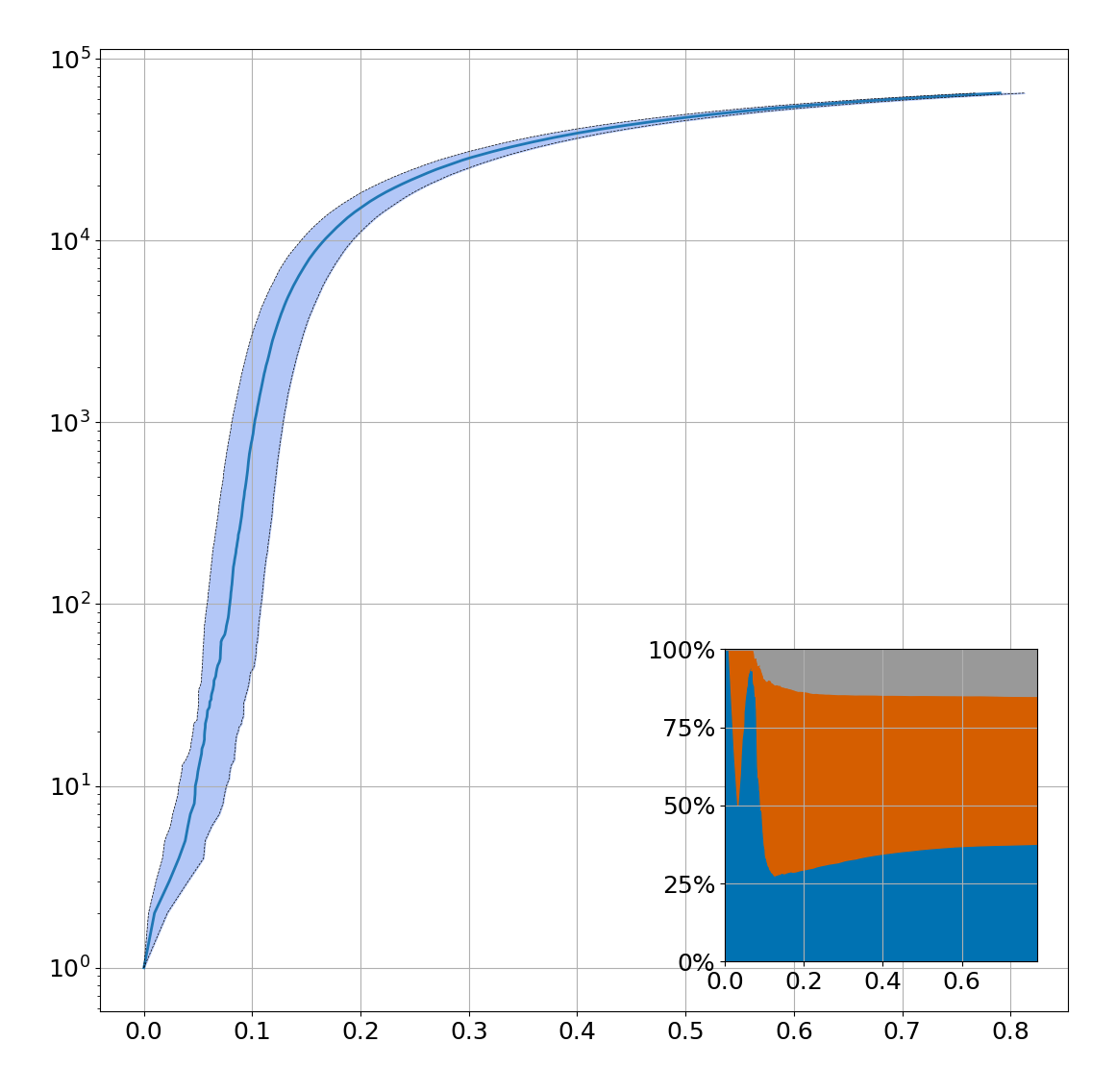}
     \end{subfigure}
     \hfill
     \begin{subfigure}[b]{0.24\textwidth}
         \centering
         \includegraphics[width=\textwidth, clip=true]{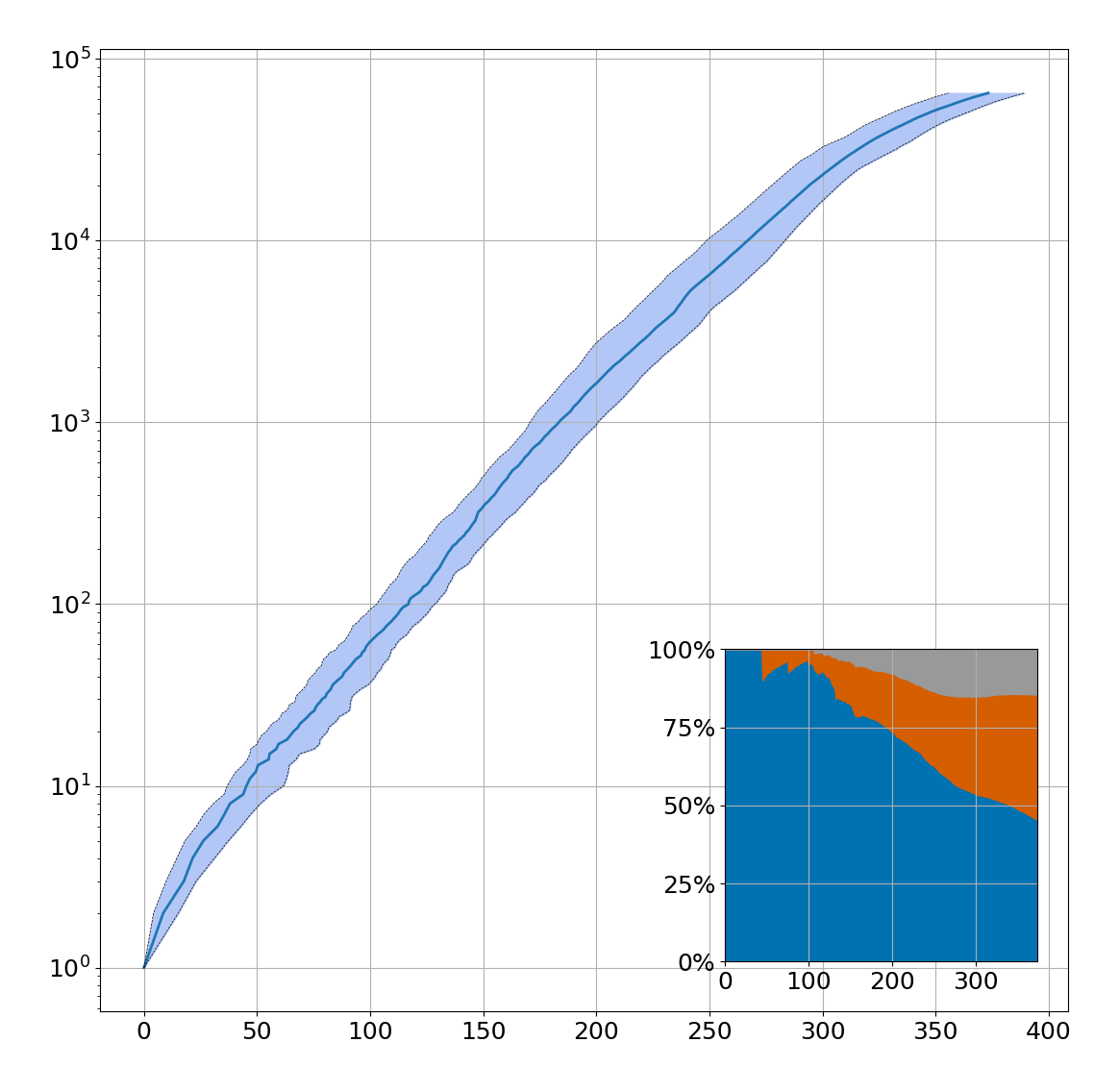}
     \end{subfigure}
     \hfill
     \begin{subfigure}[b]{0.24\textwidth}
         \centering
         \includegraphics[width=\textwidth, clip=true]{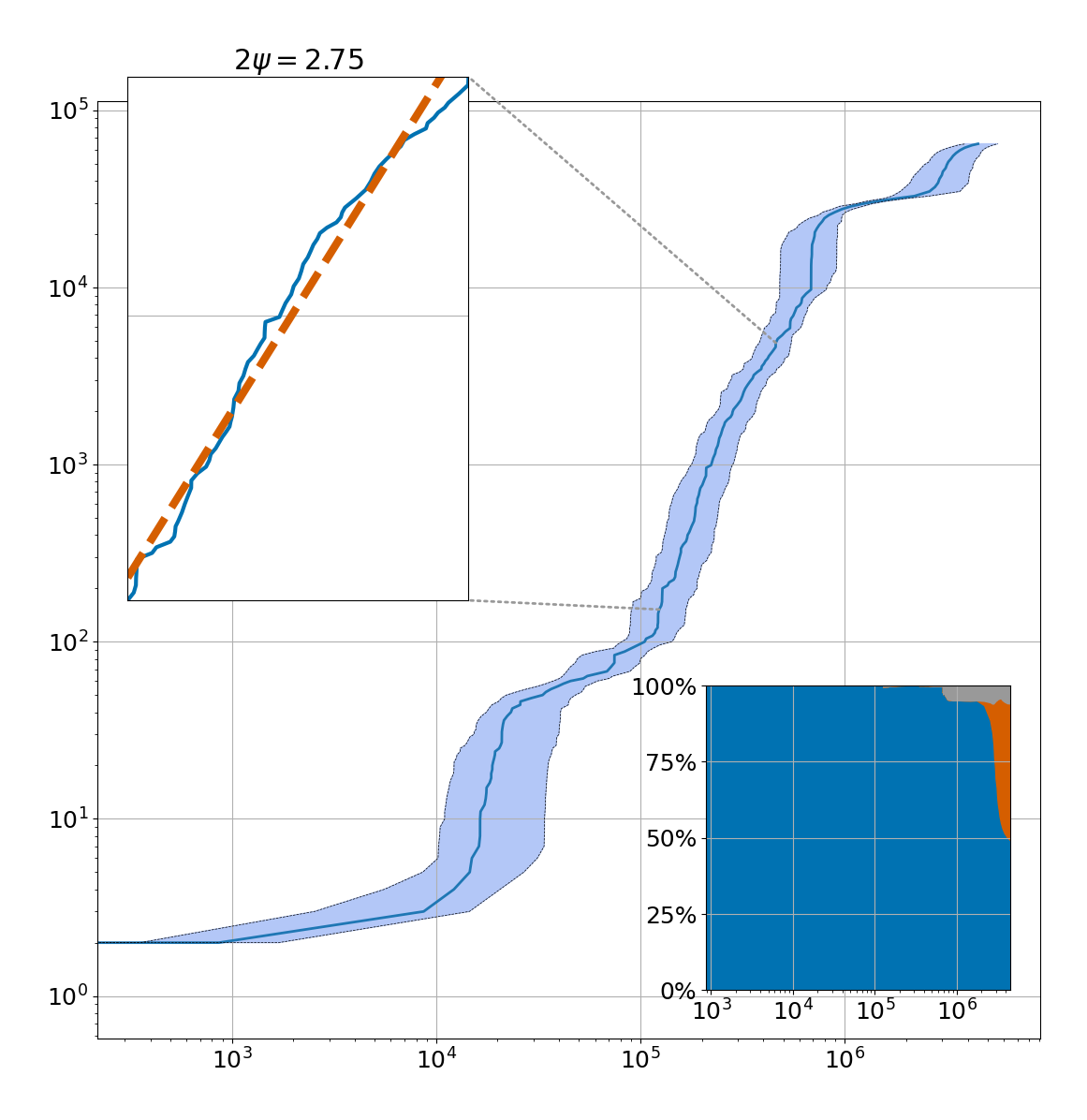}
     \end{subfigure}
     \hfill
     \begin{subfigure}[b]{0.24\textwidth}
         \centering
         \includegraphics[width=\textwidth, clip=true]{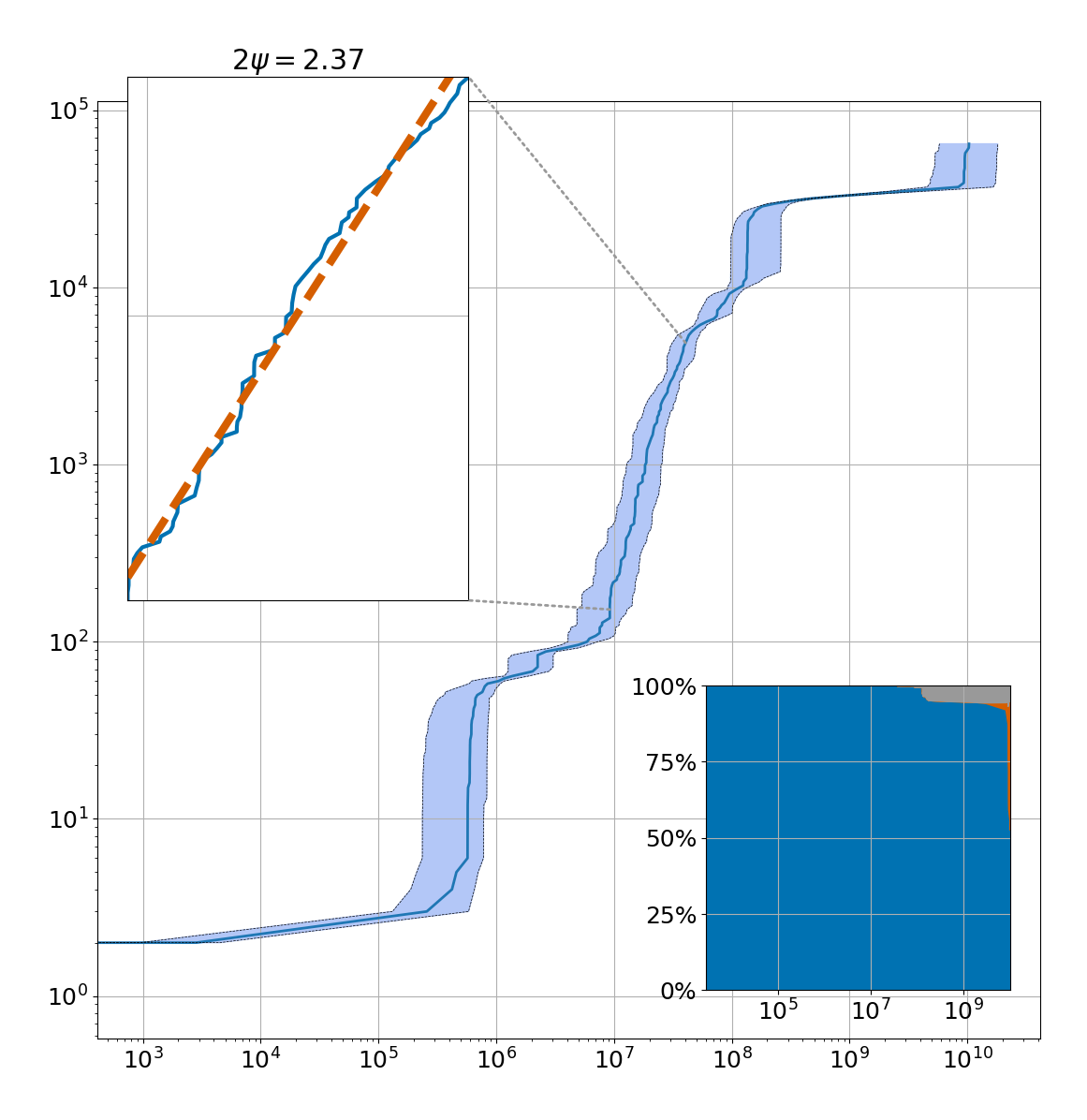}
     \end{subfigure}

     \begin{subfigure}[b]{0.24\textwidth}
         \centering
         \includegraphics[width=\textwidth, clip=true]{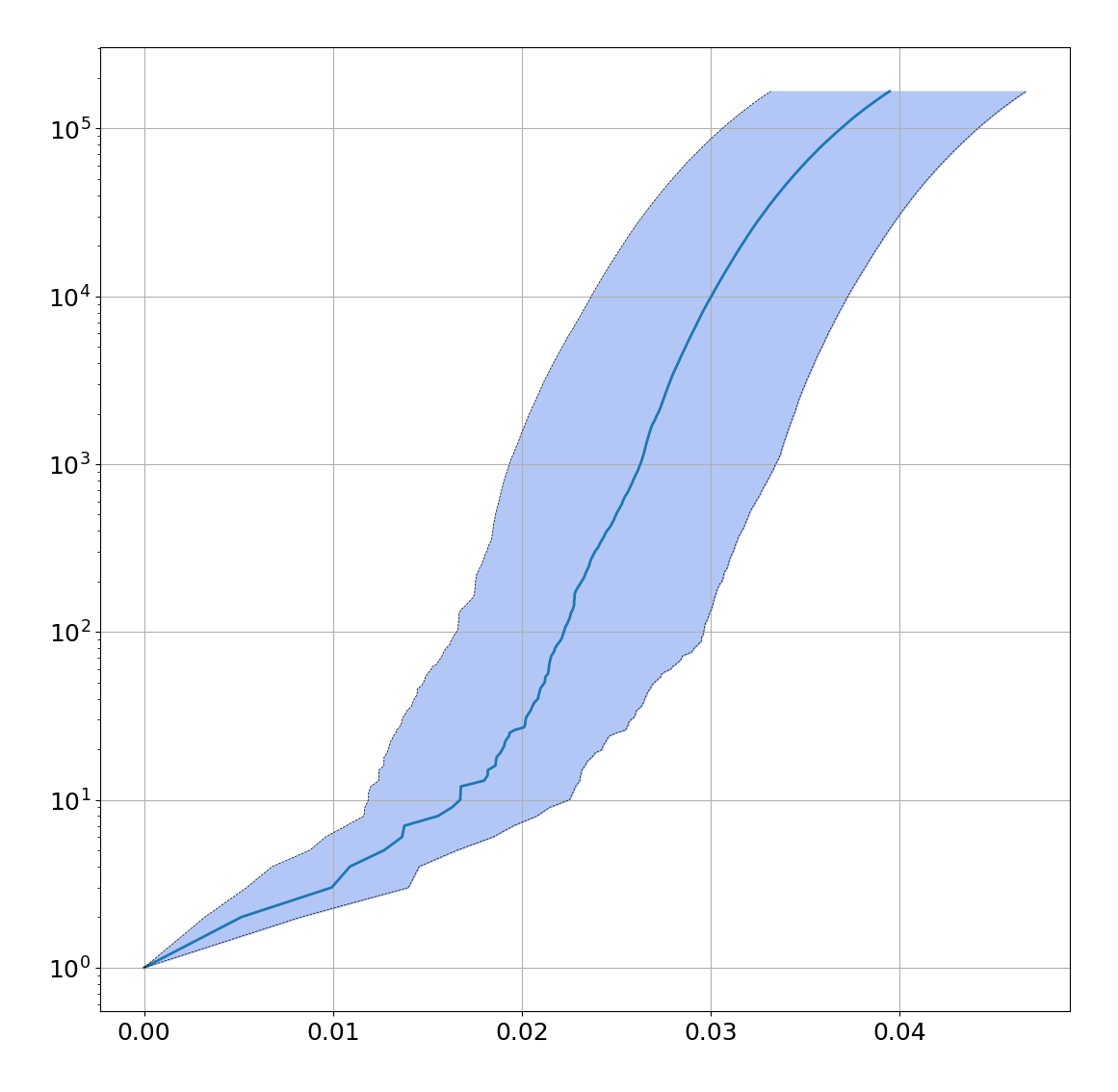}
         \subcaption{$\mu = \zeta = 0$ }
         \label{fig:epidemic-curves-explosive}
     \end{subfigure}
     \hfill
     \begin{subfigure}[b]{0.24\textwidth}
         \centering
         \includegraphics[width=\textwidth, clip=true]{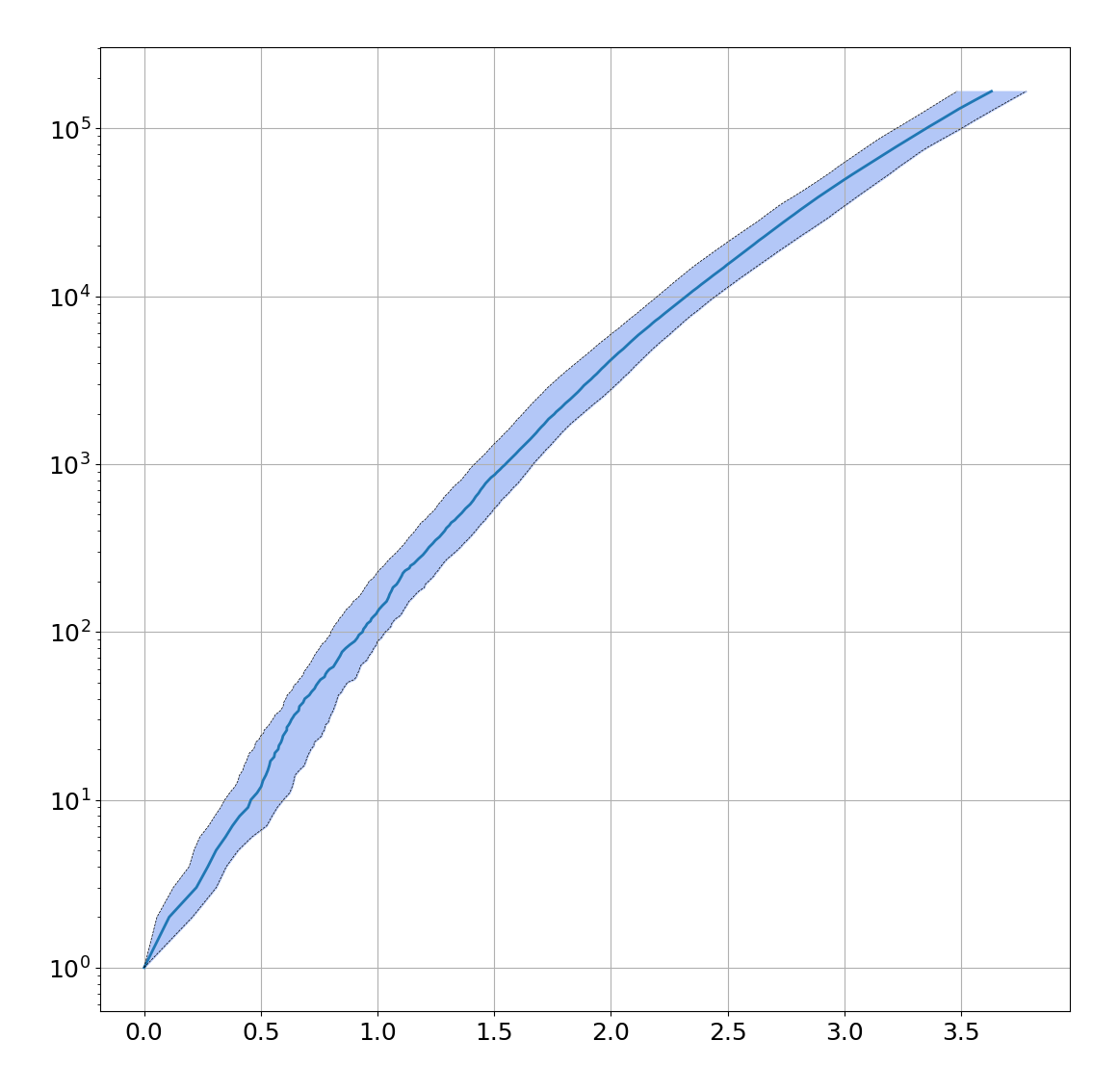}
         \subcaption{$\mu = \zeta = 1$}
         \label{fig:epidemic-curves-exponential}
     \end{subfigure}
     \hfill
     \begin{subfigure}[b]{0.24\textwidth}
         \centering
         \includegraphics[width=\textwidth, clip=true]{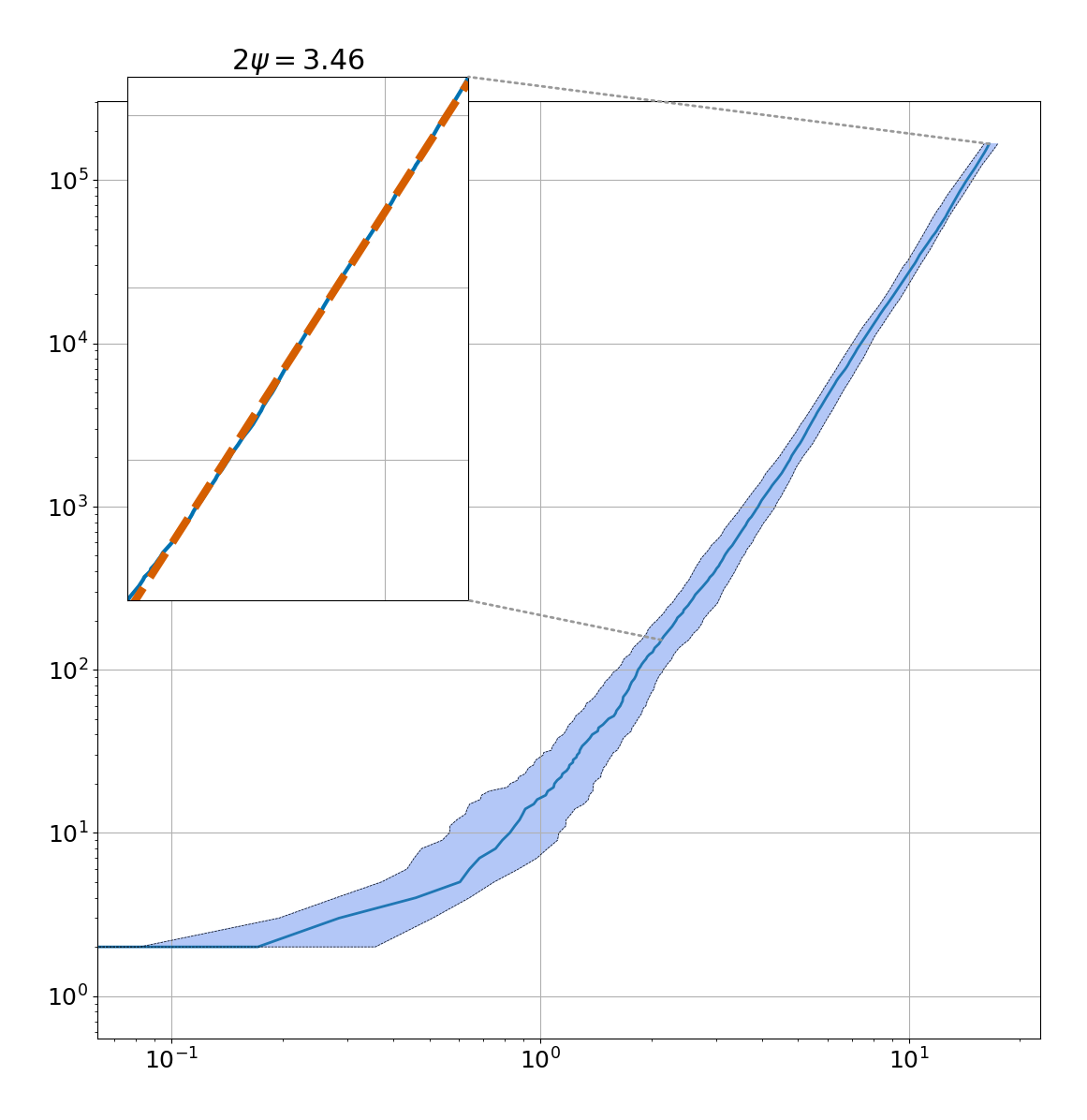}
         \subcaption{$\mu = 1, \zeta = 2$}
         \label{fig:epidemic-curves-polynomial}
     \end{subfigure}
     \hfill
     \begin{subfigure}[b]{0.24\textwidth}
         \centering
         \includegraphics[width=\textwidth, clip=true]{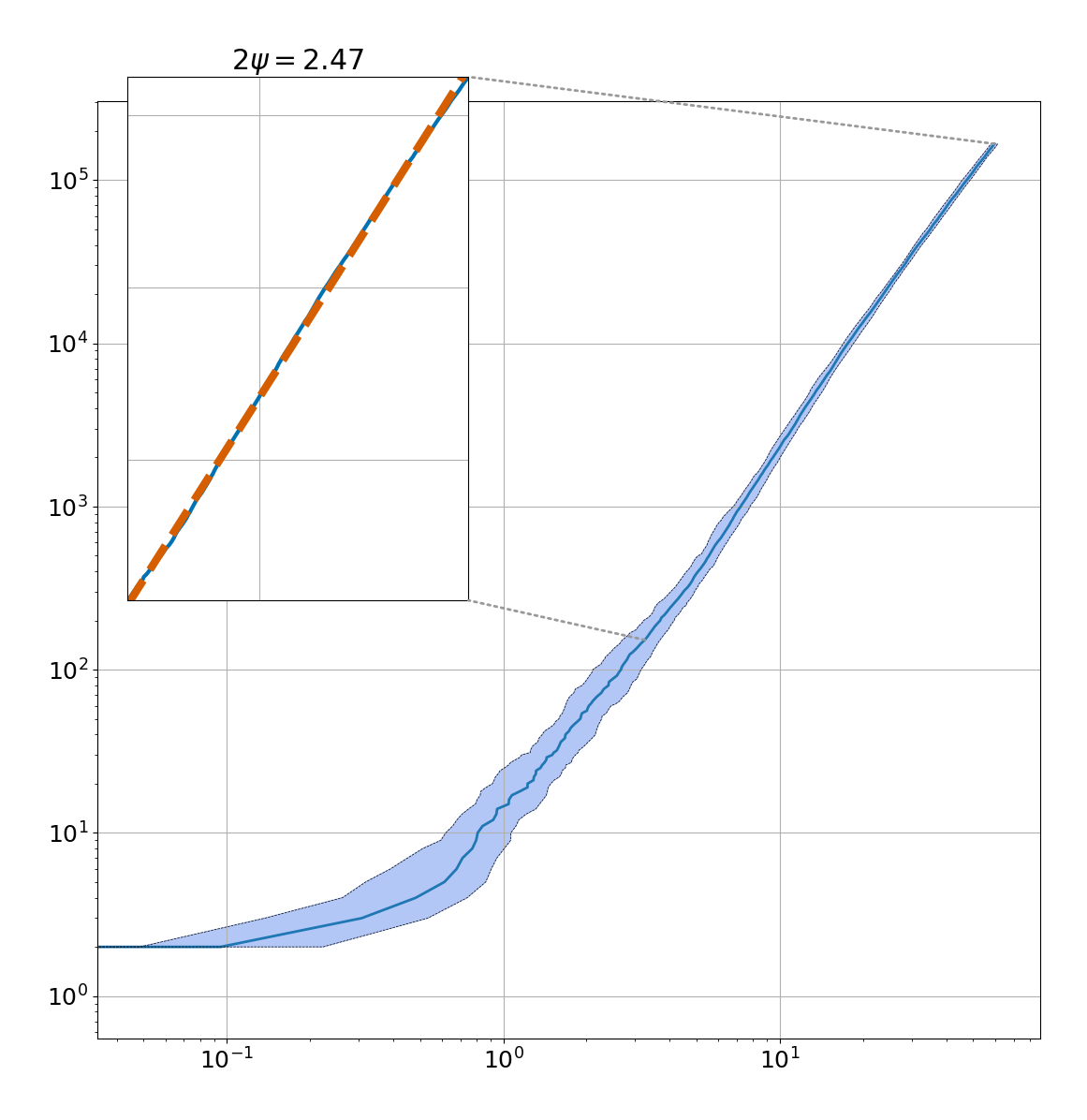}
         \subcaption{$\mu = 1, \zeta = 3$}
         \label{fig:epidemic-curves-geometric}
     \end{subfigure}
     
\caption{The four different phases for the epidemic curve $I(t)$ as a function of $t$. Top row: on the Gowalla dataset $N\sim95k$ nodes. Bottom row: on a synthetic GIRG network with $10^6$ nodes. We visualise $I(t)$ on the Gowalla dataset up to $10^{4.77}$ infected nodes, after which network saturation occurs. The middle curve is the median value of $55$ epidemic runs while the shaded area shows the $25\%-75\%$ quantiles of the runs. In each run only the random link factors $E_{uv}$ are re-sampled while the underlying graphs (i.e.\ the Gowalla network on the top row and the GIRG on the bottom row) stay the same. The $y$-axis is on a log-scale with base $10$. The $x$-axis is on a linear scale on figures (a)-(b), and on a log-scale with base $10$ on figures (c)-(d). Theorem \ref{theorem:summary} predicts that the epidemic on fig.~(a) grows explosively (i.e.\ saturates independently of the network size), on fig.~(b) it grows quasi-exponentially, on fig.~(c) it grows polynomially $\sim t^{2\psi}$ with $\psi>1$ and on fig.~(d) it grows geometrically $\sim t^{2\psi}$ with $\psi=1$. On fig.~(b), $\log I(t)\sim t^\varphi$ for $\varphi <1$, visible as a concave curve on a log-linear scale.   
For the polynomial and geometric phases on figures (c)-(d), the top left inlays show a linear regression (in red) from $I(t) = 10^{2.17}$ to $I(t) = 10^{3.70}$ with estimated slopes $2 \psi = 2.75$ (c) and $2\psi = 2.37$ (d) on the Gowalla dataset and $2\psi = 3.46$ (c) and $2\psi= 2.47$ (d) on synthetic GIRG networks. 
Top row, inlays: the proportions of infected European (blue), US (orange), and other (grey) nodes; with $t$ on a log-scale for the Gowalla dataset.
}
\label{fig:epidemic_curves}
\end{figure*}

The SI-dynamics that we use translates to the mathematical framework of \emph{first passage percolation}. In first passage percolation, each link $(u,v)$ draws an exponential random variable $Y_{uv}$ with rate $r(u,v)$. A path (a consecutive set of links) between the initially infected node $u_0$ and a node $v$ of the form $\pi=(u_0, u_1, u_2,\dots, u_k=v)$ possesses the path-transmission time 
\begin{equation} \label{eq:path-infection-time}
T_\pi:= Y_{u_0u_1} + Y_{u_1u_2}+ \dots Y_{u_{k-1}v}.
\end{equation}
Node $v$ gets infected at time $T_{\pi}$ if it is still susceptible, equivalently, if $\pi$ carries the shortest path-transmission time among all paths connecting $u_0, v$. The infection time $T(u_0,v)$ satisfies 
\begin{equation}\label{eq:tau-v}
\begin{aligned}
T(u_0, v)= \min_{\pi: \text{path from $u_0$ to $v$}} T_{\pi}.
\end{aligned}
\end{equation}
$T(u_0,v)$ can be thought of as the length of the \emph{shortest path} connecting $u_0,v$ with respect to the random link-weights. We can express the early growth of the epidemic curve as
\[
I(t) = \#\big\{ v: T(u_0, v)\le t\big\}.
\]
The shortest-path representation  reduces the problem of studying a randomly growing infection processes to studying shortest paths in a link-weighted network. For symmetric link-transmission rates in \eqref{eq:transition-rate} (i.e., $\mu=\nu$), the direction of the infection on the path can be ignored. In our theoretical results we shall establish bounds on the length of the shortest path between any far away pair of nodes $u,v$ as a function of the spatial distance $\|u-v\|$. Illustratively, suppose we show that $T(u_0,v)\sim \|v-u_0\|^{1/\psi}$ for some number $\psi\ge 1$. This then leads to the size estimate
\begin{equation}\label{eq:poly-growth}
I(t) \sim \#\big\{ v: \|v-u_0\|^{1/\psi}\le t\big\} \sim t^{d\psi},
\end{equation}
and we obtain polynomial early growth. If $\psi=1$ we call this \emph{pure geometric} growth. 
Similarly, if the shortest path $T(u_0,v)\sim (\log \|v-u_0\|)^{1/\varphi}$ for some exponent $\varphi > 0$, then 
\begin{equation}\label{eq:stretched-exp}
I(t) \sim \#\big\{ v: (\log \|v-u_0\|)^{1/\varphi}\le t\big\} \sim \exp( d t^{\varphi}),
\end{equation}
and we obtain quasi-exponential growth curves. We prove that in synthetic GIRG networks $\varphi\le 1$, while $\varphi>1$ is not (robustly) realisable, instead, the early growth jumps into yet another regime. This phase is a rather counter-intuitive possibility: the transmission time $T(u_0,v)$ does \emph{not} depend on the spatial distance $\|v-u_0\|$, but satisfies that for any positive value of $t$,  $T(u_0, v)\le t$ happens with positive probability $q_t$. Then, in a finite network with $n\gg 1$ nodes, (regardless on how large $n$ is), we observe 
\begin{equation}\label{eq:explosive-finite}
    I(t) \sim \#\big\{ v: T(u_0, v)\le t\big\} \sim q_{t-S} n,
\end{equation}
for all $t>S$ where $S$ is a random shift in time corresponding to a random kick-off time that again does not depend on $n$, but on the local neighbourhood of the source node, see \cite{komjathy2020explosion}. An epidemic curve still exists and depends on $t$ in finite networks, but the `saturation time' to reach say $99\%$ of the final size does not depend on the network size. Individual runs of the pandemic show different random shifts $S$, evidenced by less concentration of $I(t)$ on Figure \ref{fig:epidemic-curves-explosive} in this phase.   We call this phase \emph{explosive early growth}.
While such an extremely fast spreading seems unlikely for biological diseases, it provides a plausible explanation for why online viruses can sweep over the whole world with remarkable rapidity. An illustrative case is the Wannacry ransomware which reached over $200,000$ computers in over 150 countries in less than a day \cite{jonssoninformation}.

 \textbf{Mathematical formulation.} In synthetic GIRG networks, each node is equipped with a random node-weight $W_u$ that, together with the spatial distance, governs the link probabilities between nodes. A node with a given weight $W_u$ shall have, in expectation, an (explicit) constant $c$ times $W_u$ many links in total. For the readability of proofs we change  the link-transmission rates in \eqref{eq:transition-rate} to depend on node weights rather than on actual degrees. This change only eases the technicalities, as the degree of a node $\deg(u)$ shows strong concentration around its expectation $cW_u$~\cite{bringmann2019geometric}. Our theoretical findings are summarized in the following theorem.

\begin{theorem}\label{theorem:summary}
Consider an SI epidemic with transmission rates $(W_uW_v)^{-\mu}\|u-v\|^{-\zeta}$ on any link connecting nodes $u,v$ on a synthetic GIRG network with degree power law $\tau>2$ and long-range parameter $\alpha>1$ of dimension $d\ge 1$. Then the early epidemic growth curve $I(t)$ in \eqref{eq:It} satisfies
\begin{enumerate}
\setlength\itemsep{0em}  
    \item[(i)] $I(t)$ grows explosively if $\mu + \zeta/d < (3-\tau)/2$.
    \item[(ii)] $I(t)$ grows quasi-exponentially $\sim \exp( d t^{\varphi})$ for some $\varphi\le 1$ if $\mu+\zeta/d > (3-\tau)/2$, and
    at least one of the following holds additionally:
    \begin{enumerate}
    \item[a)] $\zeta/d < 2-\alpha$;
    \item[b)] $\mu+\zeta/d < 3-\tau$. 
    \end{enumerate}
    \item[(iii)] $I(t)$ grows polynomially $\sim t^{d\psi}$ for some $\psi>1$ if $\mu+\zeta/d > 3-\tau$ and $\zeta/d > 2-\alpha$ both hold, and
    at least one of the following holds additionally: 
    \begin{enumerate}
    \item[a)] $\alpha<2$ and $\zeta/d < 2-\alpha +1/d$; 
    \item[b)] $\alpha>2$, $\tau<3$ and  $\mu < (1-\zeta) \cdot (\frac{1}{d}+ \frac{3-\tau}{d(\alpha-2)})$;
    \item[c)] $\tau<3$ and $\mu+\zeta/d < 3-\tau+1/d$.

    \end{enumerate}
    \item[(iv)] $I(t)$ grows purely geometrically as $\sim t^{d}$ when none of the conditions above are satisfied. In particular, for all $\mu\ge 0, \zeta\ge 0$ when $\tau>3, \alpha>2$. 
\end{enumerate}
For a proof see the Supplementary Information.
\end{theorem}
Figure \ref{fig:phase-diagrams} illustrates the phase diagram(s)  described by Theorem \ref{theorem:summary}.  
We give a detailed interpretation.
Take a synthetic network sampled from a scale-free GIRG model with $\tau\in(2,3)$ and $\alpha>1$ on $n$ nodes (as in Figure \ref{fig:girg-low-tau-high-alpha}), (i.e., finite average degree but infinite asymptotic degree variance). 
 As we increase the exponents $\mu$ and/or $\zeta$ in \eqref{eq:transition-rate} of the link-transmission rates, we decrease the rate of the transmission on long links and links to and from high degree nodes (hubs). Theorem \ref{theorem:summary} says that the slow-down effects are negligible when $\mu + \zeta/d<(3-\tau)/2$ (area $A$ on all diagrams of Fig.~\ref{fig:phase-diagrams}), and the pandemic spreads explosively as in \eqref{eq:explosive-finite}. 
As the degree and/or spatial dependence of transmission rates are gradually increased, the pandemic slows down: $I(t)$ grows quasi-exponentially, then polynomially, and finally settles on purely geometric growth. The phase boundaries depend only on the power-law exponent $\tau$ and the long-range parameter $\alpha$, but not on other parameters (e.g. average degree) of the network. On scale-free networks with $\tau\in(2,3)$ all four phases are present. 
The strict polynomial phase with growth $\sim t^{d\psi}$, $\psi>1$ is \emph{rare} in network modeling, and our model is the first that robustly produces this phenomenon in network contexts, see \cite{komjathy2023four} for related works.  

\subsubsection{An intuitive explanation of the four phases of growth}\label{sec:proof}
The interpretation of our rigorous derivations gives insight into the driving forces of the pandemic in each of the phases.  

\emph{Explosive early growth.} 
In this phase the pandemic is driven by the presence of `superspreaders' or hubs and their hierarchical structure in networks  with high degree heterogeneity ($\tau<3$). 
The infection first explores the local neighborhood of the initial vertex $u_0$. Upon reaching some nodes with relatively high degree (say $K$), some of these nodes will transmit to a set of nodes with even higher degree ($K^s$ for some $s>1$). These nodes in turn pass on the infection to even higher degree nodes of degree $(K^s)^s$ and so on, until global hubs are reached. We call this the `degree-increasing phase', followed by a `degree descending phase' where global hubs will infect a substantial proportion of all slightly smaller hubs, who in turn infect a substantial proportion of slightly smaller hubs and so on, until local neighborhoods of typical nodes are reached. The total duration of these phases do not depend on network-size because of the very quick (summable) transmission times between hubs. Geometry only plays a role in the bounded spatial neighborhoods of the initial node and the typical target node. The infection time $T(u,v)$ converges to a sum of two random variables $Y_{u_0}+Y_{v}$ that respectively describe how long it takes to leave/enter the local neighborhood of $u_0$ and $v$. Here $Y_{u_0}$ represents the random time it takes for the pandemic to take off, and $Y_{u_0}$ enters the infection time for all nodes in the network, resulting in a random shift in the epidemic curve,  recovering \eqref{eq:explosive-finite}:
\begin{equation} I(t) \sim n\mathbb P( Y_v <t-S \mid Y_{u_0}=S)=n q_{t- S}.\end{equation}

\begin{figure*}
    \centering

    \begin{minipage}[c]{0.25\textwidth}
        \centering

        \begin{tikzpicture}[scale=1]
        
        \fill[color=yellow!20] (0,0)--(0,0.115)--(0.23,0)--(0,0);
        \fill[color=blue!20] (0,0.115)--(0.23,0)--(1.7,0)--(1.7,0.5)--(0,0.5)--(0,0.115);
        \fill[color=magenta!20] (1.7,0)--(1.7,0.5)--(2.7,0.5)--(2.7,0)--(1.7,0);
        \fill[color=lime!20] (2.7,0)--(2.7,0.5)--(3.1,0.5)--(3.1,0)--(2.7,0);
        
        \fill[yellow!20] (-0.15,0.25) circle (3pt) node {\tiny\color{black}{$A$}}; 
        \fill[blue!20] (0.85,0.25) circle (1pt) node {\tiny\color{black}{$B$}};
        \fill[magenta!20] (2.2,0.25) circle (1pt) node {\tiny\color{black}{$D$}};
        \fill[lime!20] (2.9,0.25) circle (1pt) node {\tiny\color{black}{$G$}};
        
        \draw[->] (-0.05,0)--(3.15,0);
        \draw[->] (0,-0.05)--(0,0.55);
        
        \draw[very thick] (0,0.115) -- (0.23,0);
        \draw (1.7,0)--(1.7,0.5);
        \draw (2.7,0)--(2.7,0.5);
        
        \node[right] at (3.15,0) {\tiny$\zeta$};
        \node[above] at (0,0.55) {\tiny$\mu$};
        \node[left] at (0,0) {\tiny$0$};
        \node[below] at (0,0) {\tiny$0$};
        \node[left] at (0,0.5) {\tiny$0.5$};
        \node[below] at (1,0) {\tiny$1$};
        \node[below] at (2,0) {\tiny$2$};
        \node[below] at (3,0) {\tiny$3$};

        \draw (1,0.05) -- (1,-0.05);
        \draw (2,0.05) -- (2,-0.05);
        \draw (3,0.05) -- (3,-0.05);

        \end{tikzpicture}

        \subcaption{}
        \label{fig:gowalla-phase-diagram}

        \begin{tikzpicture}[scale=1.3]
        
        \fill[color=yellow!20] (0,0)--(0,0.3)--(0.6,0)--(0,0);
        \fill[color=blue!20] (0,0.3)--(0.2,0.2)--(0.8,0.2)--(0.8,0.5)--(0,0.5)--(0,0.3);
        \fill[color=magenta!20] (0.8,0.2)--(1.8,0.2)--(1.8,0.5)--(0.8,0.5)--(0.8,0.2);
        \fill[color=cyan!20] (0.2,0.2)--(0.8,0.2)--(1.2,0)--(0.6,0)--(0.2,0.2);
        \fill[color=red!20] (0.8,0.2)--(1.8,0.2)--(2.2,0)--(1.2,0)--(0.8,0.2);
        \fill[color=lime!20] (1.8,0.2)--(2.2,0)--(2.4,0)--(2.4,0.5)--(1.8,0.5)--(1.8,0.2);
        
        \fill[yellow!20] (0.18,0.09) circle (1pt) node {\tiny\color{black}{$A$}}; 
        \fill[blue!20] (0.4,0.35) circle (1pt) node {\tiny\color{black}{$B$}};
        \fill[cyan!20] (0.7,0.1) circle (1pt) node {\tiny\color{black}{$C$}};
        \fill[magenta!20] (1.3,0.35) circle (1pt) node {\tiny\color{black}{$D$}};
        \fill[red!20] (1.5,0.1) circle (1pt) node {\tiny\color{black}{$F$}};
        \fill[lime!20] (2.1,0.25) circle (1pt) node {\tiny\color{black}{$G$}};
        
        \draw[->] (-0.05,0)--(2.45,0);
        \draw[->] (0,-0.05)--(0,0.55);
        
        \draw[very thick] (0,0.3) -- (0.6,0);
        \draw[dashed] (0.2,0.2)--(1.8,0.2);
        \draw (0.8,0.2)--(0.8,0.5);
        \draw (0.8,0.2) -- (1.2,0);
        \draw (1.8,0.2)--(1.8,0.5);
        \draw (1.8,0.2)--(2.2,0);
        
        \node[right] at (2.45,0) {\tiny$\zeta$};
        \node[above] at (0,0.55) {\tiny$\mu$};
        \node[left] at (0,0) {\tiny$0$};
        \node[below] at (0,0) {\tiny$0$};
        \node[left] at (0,0.5) {\tiny$0.5$};
        \node[below] at (1,0) {\tiny$1$};
        \node[below] at (2,0) {\tiny$2$};

        \draw (1,0.05) -- (1,-0.05);
        \draw (2,0.05) -- (2,-0.05);
        
        \end{tikzpicture}

        \subcaption{}
        \label{fig:mu-zeta-small-alpha}
        
    \end{minipage}
    \begin{minipage}[c]{0.25\textwidth}
        \centering

        \begin{tikzpicture}[scale=1.1]

        \fill[color=yellow!20] (0,0)--(0,0.4)--(0.8,0)--(0,0);
        \fill[color=cyan!20] (0,0.4)--(0.8,0)--(1.6,0)--(0,0.8)--(0,0.4);
        \fill[color=red!20] (0,0.8)--(1.6,0)--(2.6,0)--(0.5,1.05)--(0,1.05)--(0,0.8);
        \fill[color=purple!20] (0,1.05)--(0.5,1.05)--(0,2.1)--(0,1.05);
        \fill[color=lime!20] (2.6,0)--(0.5,1.05)--(0,2.1)--(0,2.3)--(2.9,2.3)--(2.9,0)--(2.6,0);
        
        \fill[yellow!20] (0.25,0.125) circle (1pt) node {\tiny\color{black}{$A$}}; 
        \fill[cyan!20] (0.6,0.3) circle (1pt) node {\tiny\color{black}{$C$}};
        \fill[purple!20] (0.18,1.3) circle (1pt) node {\tiny\color{black}{$E$}};
        \fill[red!20] (1.04,0.52) circle (1pt) node {\tiny\color{black}{$F$}};
        \fill[lime!20] (1.6,1.3) circle (1pt) node {\tiny\color{black}{$G$}};
        
        \draw[->] (-0.05,0)--(2.95,0);
        \draw[->] (0,-0.05)--(0,2.35);
        
        \draw[very thick] (0,0.4) -- (0.8,0);
        \draw (0,0.8) -- (1.6,0);
        \draw[dashed] (0,1.05)--(0.5,1.05);
        \draw (2.6,0)--(0.5,1.05);
        \draw (0.5,1.05)--(0,2.1);
        
        \node[right] at (2.95,0) {\tiny$\zeta$};
        \node[above] at (0,2.35) {\tiny$\mu$};
        \node[left] at (0,0) {\tiny$0$};
        \node[below] at (0,0) {\tiny$0$};
        \node[below] at (1,0) {\tiny$1$};
        \node[below] at (2,0) {\tiny$2$};
        \node[left] at (0,1) {\tiny$1$};
        \node[left] at (0,2) {\tiny$2$};

        \draw (1,0.05) -- (1,-0.05);
        \draw (2,0.05) -- (2,-0.05);
        \draw (0.05,1) -- (-0.05,1);
        \draw (0.05,2) -- (-0.05,2);

        \end{tikzpicture}

        \subcaption{}
        \label{fig:mu-zeta-large-alpha}
        
    \end{minipage}
    \begin{minipage}[c]{0.25\textwidth}
        \centering
        
        \begin{tikzpicture}[scale=2]

        \fill[color=yellow!20] (1,2)--(2.9,2)--(2.9,2.2)--(1,2.2)--(1,2);
        \fill[color=blue!20] (1,2.2)--(1.5,2.2)--(1.9,2.6)--(1.9,3.2)--(1,3.2)--(1,2.2);
        \fill[color=magenta!20] (1.9,2.6)--(2,2.7)--(2,3.2)--(1.9,3.2)--(1.9,2.6);
        \fill[color=purple!20] (2,3)--(2.15,2.85)--(2,2.7)--(2,3);
        \fill[color=cyan!20] (1.9,2.6)--(2.9,2.6)--(2.9,2.2)--(1.5,2.2)--(1.9,2.6);
        \fill[color=red!20] (1.9,2.6)--(2.9,2.6)--(2.9,2.85)--(2.15,2.85)--(1.9,2.6);
        \fill[color=lime!20] (2.15,2.85)--(2,3)--(2,3.2)--(2.9,3.2)--(2.9,2.85)--(2.15,2.85);
        
        \fill[yellow!20] (1.7,2.1) circle (1pt) node {\tiny\color{black}{$A$}}; 
        \fill[blue!20] (1.4,2.6) circle (1pt) node {\tiny\color{black}{$B$}};
        \fill[cyan!20] (2.4,2.4) circle (1pt) node {\tiny\color{black}{$C$}};
        \fill[magenta!20] (1.95,2.85) circle (1pt) node {\tiny\color{black}{$D$}};
        \fill[purple!20] (2.05,2.85) circle (1pt) node {\tiny\color{black}{$E$}};
        \fill[red!20] (2.4,2.72) circle (1pt) node {\tiny\color{black}{$F$}};
        \fill[lime!20] (2.4,2.925) circle (1pt) node {\tiny\color{black}{$G$}};
        
        \draw[->] (0.97,2)--(2.95,2);
        \draw[->] (1,1.97)--(1,3.25);
        \draw[very thick] (1,2.2)--(2.9,2.2);
        \draw[dashed] (1.5,2.2) -- (2.15,2.85);
        \draw (1.9,2.6)--(1.9,3.2);
        \draw[dashed] (2,2.7)--(2,3.2);
        \draw (2,3)--(2.15,2.85);
        \draw (1.9,2.6)--(2.9,2.6);
        \draw (2.15,2.85)--(2.9,2.85);
        \draw[very thick] (2,3)--(2,3.2);
        
        \draw[dotted] (1,3)--(2.9,3);
        \draw[dotted] (2,2)--(2,2.7);
        
        \node[right] at (2.95,2) {\tiny$\alpha$};
        \node[above] at (1,3.25) {\tiny$\tau$};
        \node[left] at (1,3) {\tiny$3$};
        \node[left] at (1,2) {\tiny$2$};
        \node[below] at (1,2) {\tiny$1$};
        \node[below] at (2,2) {\tiny$2$};

        \draw (2,2.03) -- (2,1.97);
        \draw (1.03,3) -- (0.97,3);
        
        \end{tikzpicture}
        
        \subcaption{}
        \label{fig:alpha-tau}
        
    \end{minipage}
    \begin{minipage}[c]{0.2\textwidth}
        \centering
        
        \begin{tikzpicture}[scale=2] 
        
        \fill[yellow!20] (0,0) circle (2pt) node {\tiny\color{black}{$A$}}; 
        \fill[blue!20] (0,-0.2) circle (2pt) node {\tiny\color{black}{$B$}};
        \fill[cyan!20] (0,-0.4) circle (2pt) node {\tiny\color{black}{$C$}};
        \fill[magenta!20] (0,-0.6) circle (2pt) node {\tiny\color{black}{$D$}};
        \fill[purple!20] (0,-0.8) circle (2pt) node {\tiny\color{black}{$E$}};
        \fill[red!20] (0,-1) circle (2pt) node {\tiny\color{black}{$F$}};
        \fill[lime!20] (0,-1.2) circle (2pt) node {\tiny\color{black}{$G$}}; 

        \node[right] at (0.1,0) {Explosive};
        \node[right] at (0.1,-0.2) {Weak-tie quasi-exponential};
        \node[right] at (0.1,-0.4) {Hub quasi-exponential};
        \node[right] at (0.1,-0.6) {Weak-tie polynomial};
        \node[right] at (0.1,-0.8) {Hybrid polynomial};
        \node[right] at (0.1,-1) {Hub polynomial};
        \node[right] at (0.1,-1.2) {Geometric};
        
        \end{tikzpicture}

    \end{minipage}
    
    \caption{Phase diagrams of epidemic growth for the contact-dependent SI. On all diagrams, parameter choices falling in region $A$ yield explosive spread. Parameter choices in $B$ and $C$ yield quasi-exponential growth $I(t)\sim \exp(dt^{\varphi})$, while in $D$, $E$ and $F$ they yield polynomial growth with $I(t) \sim t^{d\psi}, \psi>1$, while parameter choices in $G$ yield purely geometric growth $I(t) \sim t^{d}$. Optimal infection paths use weak ties for parameters in region $B$ and $D$, hubs for $C$ and $F$, and weak ties between medium-sized hubs for $E$. The bold lines indicate discontinuous phase transitions, while the other transitions are smooth. Figures (a)-(c) show phase diagrams when the penalty parameters $\mu$ and $\zeta$ vary for fixed $\tau,\alpha$ and $d=2$: (a) uses the estimated parameters of the Gowalla dataset $\tau_{\mathrm{Gow}} = 2.78, \alpha_{\mathrm{Gow}} = 1.2$; (b) uses $\tau=2.4, \alpha=1.6$; (c) uses $\tau=2.2, \alpha=2.25$. Figure (d) shows a phase diagram when $\tau$ and $\alpha$ vary for fixed $d=4$, $\mu = 0.3$, $\zeta=0.4$.
    }   
    \label{fig:phase-diagrams}
\end{figure*}
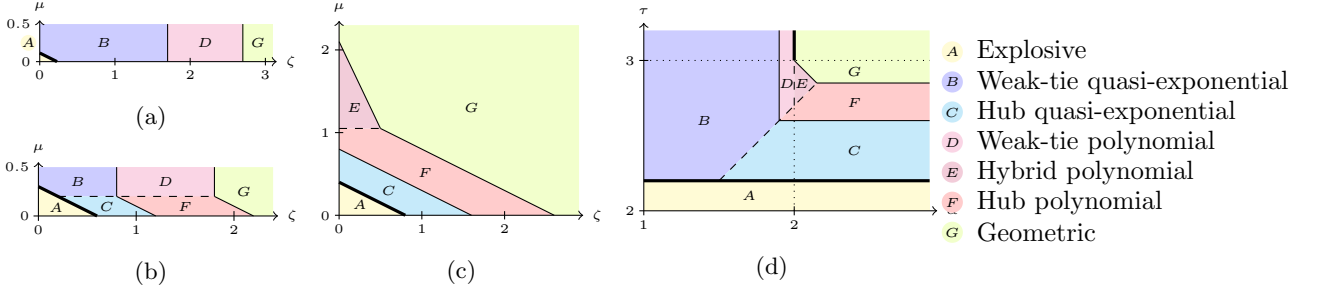

\emph{Quasi-exponential and polynomial early growth.} 
In both of these phases, the pandemic shows a fractal structure and spreads via alternating between local infections and using links that span large distances. 
The process starts by a take-off phase using local infections near the source node, until a few relatively long links transmit the disease to new areas.  At some of these new areas, via local infections again, even longer links are found that transmit the disease to even further areas. As the total area of local infection grows, the chance of finding even longer links  increases, leading to self-similarity and fractal-like growth patterns, see Fig.~\ref{fig:heatmaps-exponential}-\ref{fig:heatmaps-polynomial}. The difference between the quasi-exponential and the polynomial phase is the time it takes to pass the long links that lead the infection to new areas. In the quasi-exponential phase long links are passed so quickly that the infection time across such a link does not depend on the spatial distance the link spans. In this phase the process can enter new areas in constant time and we see almost exponential growth — the limitations essentially come from geometry saturating the process locally on all scales, inhibiting growth from exponential to quasi-exponential. In the polynomial phase all long links experience significant slow-downs, and carry transmission times polynomial in the distance they span. Fractal-like spreading with alternating local infections and long links has been observed in real pandemics, see \cite{PACURAR2020110073, long2023multifractal, natalia2023fractal}. In our theoretical results, the type of long links that drive the pandemic subdivide each phase into subphases: optimal infection paths typically use either many weak ties on fairly low-degree nodes (areas $B$ and $D$ in Fig.~\ref{fig:phase-diagrams}) or long links between locally dominant hubs (areas $C$ and $F$). In the polynomial phase a third, hybrid way of spreading by using atypically long weak ties between less prominent hubs is the optimal strategy that dominates the spread under specific parameter settings (area $E$). We illustrate typical infection paths on Figure S7 in SI.

\emph{Pure geometric growth.} When the above transmission paths become too slow to play a role in the spreading process due to the penalty on high degree nodes and/or long-distance links, the pandemic propagates using essentially local, short-range infections only and avoids hubs. The process then propagates following the properties of the underlying geometric space. Synthetic GIRG networks assume a random yet fairly homogeneous node distribution, yielding to the circular growth pattern of Fig.~\ref{fig:heatmaps-geometric}. We conjecture that for GIRGs a \emph{shape theorem} holds, as is universally observed in first passage percolation on geometrically embedded (lattice-like and random) graphs \cite{cox1981some, coletti2023limiting}. Despite the high variations in node density in the Gowalla dataset, the geometric pandemic growth pattern is still visible there in the same phase, see Fig.~\ref{fig:heatmaps-geometric}.

Optimizing the type and length of links that the pandemic typically uses to enter new regions allows us to rigorously obtain explicit exponents of early growth on synthetic GIRG networks. To state our results, we introduce $\varphi:=\lim_{t\to \infty} \tfrac{\log\log I(t)}{\log t}$ and $\psi:=\lim_{t\to \infty} \tfrac{\log I(t)}{\log (t^d)}$, as in \eqref{eq:poly-growth}-\eqref{eq:stretched-exp}, which is meant on the infinite geometric network. Alternative definitions on finite networks $G_n$ are possible; by defining $\varphi$ and $\psi$ using the growth of $I(t)$ on the interval $[0,t_n]$, where $t_n\to \infty$ is a time well before saturation, and taking the large network limit $n\to\infty$. 

\begin{theorem}\label{theorem:EXACT} In the setting of Theorem \ref{theorem:summary}, for explicit functions $\varphi=\varphi(d, \mu, \zeta, \tau, \alpha)$ and $\psi=\psi(d, \mu, \zeta, \tau, \alpha)$, it holds that 
\begin{equation}\label{eq:EXACT-both}
\begin{aligned}
I(t) &\lesssim \exp( t^{(1+o(1))\varphi}) \quad \text{in phase (ii),}\\
I(t)&\sim t^{d\psi (1+o(1))} \quad\quad  \text{in phases (iii)-(iv)}.
\end{aligned}
\end{equation}
\end{theorem}

The limit functions $\varphi, \psi$ are different on each of the areas of the phase diagram on Figure \ref{fig:phase-diagrams}, with exact values in \eqref{eq:delta-values} and \eqref{eq:eta-values}. They indicate the type of the long links that carry the infection to new areas.
In the quasi-exponential phase, these long links are either typically weak ties (region $B$ in Figure \ref{fig:phase-diagrams}) or links between two hubs (region $C$). For the polynomial phase (iii), based on the exponent $\psi$, there are three possible driving forces: long links between resp. weak ties (region $D$), hubs (region), and hybrid (region $E$), i.e.\ weak ties between medium-sized hubs. In the pure geometric phase, short-range connections drive the pandemic (region $G$). 

These results can be summarised in \emph{phase diagrams} (Fig.~\ref{fig:phase-diagrams}). In a natural scenario, we think of the network as fixed, and we only vary the penalisation exponents of the transmission dynamics $\mu, \zeta$ in \eqref{eq:transition-rate}. This gives rise to Fig.~\ref{fig:gowalla-phase-diagram} where we set $\tau=2.78, \alpha=1.2$, giving the phase diagram of a GIRG with the same tail-exponents as the Gowalla network. Theoretically, we may also fix $\mu$ and $\zeta$, then vary the parameters $\alpha$ and $\tau$ of the underlying (synthetic) network, giving rise to the phase diagram in Fig.~\ref{fig:alpha-tau}, where the same colour corresponds to the same function for $\varphi$ and $\psi$.   
Our results imply that boundaries of the regions $A$, $B\cup C$, $D\cup E \cup F$, and $G$ on Figure \ref{fig:phase-diagrams} are true phase boundaries, while the boundaries between $B$ and $C$ and those between $D$, $E$ and $F$ describe continuous shifts in the way the pandemic is most likely to propagate. 

\section*{Discussion and Conclusions}
Different behavioural factors in the transmission dynamics can lead to significantly different early epidemic growth curves on the same underlying network for even the same disease in its pandemic state.  The phenomenon has been observed in consecutive waves of the Covid-19 pandemic in various countries, where growth varied between super-exponential, quasi-exponential and polynomial growth across different waves. 
Our theory suggests a plausible explanation for this phenomenon and a simple model of transmission dynamics where the phenomenon can be reproduced both on synthetic and on real networks. We have studied the effects of ego-network-dependent transmission dynamics in disease spreading, where the random transmission time across a link is influenced by the number of total contacts of both the infector and the receiver individual as well as the spatial distance between them. We have found that our model can  can reproduce super-exponential, (quasi)-exponential, polynomial and pure geometric early epidemic growth on the same spatially embedded network. Our results are based both on a rigorous mathematical analysis of growth profiles in a synthetic network model and on data-driven simulation of growth dynamics on a social network.

 Based on existing mathematical theory, we conjecture that the results are robust also in the choice of network: other synthetic network models will show a subset of these phases -- for instance, non-spatial networks with power-law degree distributions will reproduce the explosive and exponential phase \cite{adriaans2018weighted, bhamidi2017universality} but will lack the polynomial and geometric phases, as those are driven by geometry. Homogeneous lattice-based models will show the geometric phase, while lattice-based long-range models may also show the quasi-exponential phase \cite{biskup2004scaling} and the polynomial phase but will lack the explosive phase as that is driven by superspreaders/hubs.    
  
A natural limitation of our theory is that it uses the SI model of infection. This is a natural choice for the onset of a pandemic, but does not cover later phases when saturation of subpopulations factor in. This is particularly relevant for the slower growth regimes, and we want to caution about the polynomial but not purely geometric growth phase in geometrically embedded networks. In that phase and that phase only, the pandemic is driven by long-range contacts with long transmission times. These transmissions may be filtered out by an SIR dynamics where healing is also present, or need to assume the presence of individuals with long healing times.

 Beyond the scientific interests, our results may contribute to the better design of epidemic forecasts and intervention strategies during an ongoing pandemic. We highlight the importance of the underlying social contact network, in particular the effect of long-range connections and hubs (superspreaders) as the potential driving sources of the pandemic whose presence may push the pandemic into a faster universality class of growth. These observations could lead to different testing and intervention strategies (e.g., focusing on limiting the number of social contacts and long-range travels). Crucially, our theory predicts that the efficacies of different interventions may vary for different growth regimes, since they are driven by different types of links. Thus, identifying the growth regime may help identifying the most effective interventions for any given wave.

\section*{Materials and Methods}
\subsection*{Construction of the Gowalla Network}
For the analysis performed on the Gowalla network we used spatial locations of $n=107,092$ users. Users typically use multiple login locations; while in the constructed network we assign to each user a unique location as follows: 
we identify for each user the modal set of coordinates (rounded to the nearest $.25$), then choose the most common login location from within the corresponding $.25\times .25$ longitude-latitude box, breaking all ties uniformly at random. This method is consistent with the method in \cite{gowalla}, where an accuracy of $85\%$ for approximate home location is estimated. With this procedure we obtain a graph on $n = 107, 092$ nodes with $456,830$ links, with an average degree $8.53$. We then pass to the largest connected component of this network, which has $96,953$ nodes with $455,026$ links and average degree $9.39$.

\subsection*{Parameter estimation for synthetic GIRG networks}
 We chose the degree power-law exponent $\tau$ and the long-range parameter $\alpha$ of our synthetic network to match the Gowalla dataset. We choose $\tau$ by estimating the tail exponent of the power-law distribution of the Gowalla network's degree sequence, using a double-bootstrapped Hill's estimator \cite{hill1975simple} as applied to scale-free networks in \cite{voitalov2019scale} (see SI). Estimating $\alpha$ is harder, as the link-length distribution of the Gowalla dataset is far less regular than in a GIRG; this would be true even for a GIRG generated on the Gowalla node-set, since coastal nodes have far fewer nodes within a given geographical distance than inland nodes. Nevertheless, a power-law is still a good fit when truncated to shorter edges, and in SI we develop a non-linear regression method to estimate $\alpha$ based on this approximation. The estimator achieves reliable recovery of the true $\alpha$ on synthetic GIRG networks. We then apply the same estimator to the Gowalla dataset: the estimator is robust under the choice of the truncation thresholds.

\subsection*{Generating Geometric Inhomogeneous Random Graphs}
Synthetic GIRG networks were generated as follows. The number of nodes was sampled from a Poisson distribution with mean $n$. Each of these nodes was assigned a position sampled uniformly at random from the unit torus \([0,\sqrt{n})^2\), and a node-``fitness'' value \( w_u \), drawn independently from a Pareto distribution with exponent \( \tau \). A link between any two nodes \( u \) and \( v \) was added with probability  
\begin{equation}\label{eq:girg-formula}
    \mathrm{Prob}(u,v) =  \min\big\{ \big( \tfrac{ w_u w_v}{\|x_u - x_v\|^d} \big)^\alpha, 1 \big\},
\end{equation}
where \( \|x_u - x_v\| \) is the distance on the torus between the nodes. We employed a linear time GIRG sampling algorithm to generate these synthetic networks~\cite{bringmann2017sampling} \cite{GIRG_sampling_github}.  

\subsection*{Random Reference Networks}
To generate a random reference network from the Gowalla dataset, we repeatedly iterate over all edges. For each edge $(u,v)$, we take a uniformly random edge $(u',v')$ and attempt to rewire the two edges to $(u, u')$ and $(v, v')$. This is essentially a random walk on the switch chain, which is known to mix rapidly when the degree sequence of the original graph follows a power law~\cite{egmmss-switch}. Thus the result is an approximately uniform graph with the same positions and degree sequence as the original Gowalla network and no multiple edges or loops. We use the implementation from the graph-tool library~\cite{peixoto_graph-tool_2014}.

\subsection*{Growth Exponent of Epidemic Curves}
In Theorem \ref{theorem:EXACT} we identified the growth exponents of the epidemic curve $I(t)$ in the quasi-exponential and the polynomial phase on synthetic GIRG networks. Here we give the value of these exponents. 
In the quasi-exponential phase (ii) of Theorems \ref{theorem:summary} \ref{theorem:EXACT}, it holds that  $\lim_{t\to \infty}\log \log I(t)/ \log t=\varphi(d,\mu,\zeta,\tau,\alpha)=\max(\varphi_1, \varphi_2)\le 1$, where 
 \begin{equation}\label{eq:delta-values}
    \begin{aligned}
     \varphi=\varphi_1&=1-\log_2 (\alpha+\zeta/d)), \qquad \text{on region B} \\
     \varphi=\varphi_2&=1-\log_2(\tau-1+\zeta/d+\mu) \qquad \text{on region C},   
     \end{aligned}
 \end{equation}
 and $\varphi_1$ enters the maximum only if $\zeta/d<2-\alpha$, while $\varphi_2$ only if $\mu +\zeta/d<3-\tau$. 
  In the polynomial phase (iii) of Theorems \ref{theorem:summary}-\ref{theorem:EXACT}, $\lim_{t\to \infty}\log I(t)/ \log (t^d)=\psi(d,\mu,\zeta,\tau,\alpha)=\max(\psi_1, \psi_2, \psi_3)> 1$
where
 \begin{equation}\label{eq:eta-values}
    \begin{aligned}
     \psi&=\psi_1=1/\big(\zeta-d(2-\alpha)\big), \qquad \text{on region D}\\
     \psi&=\psi_2=1/\big(\zeta + \mu d \tfrac{ (\alpha-2)}{\alpha-(\tau-1)}\big), \qquad \text{on region E}\\
     \psi=\psi_3&=1/\big(\zeta+\mu d -d(3-\tau)\big), \qquad \text{on region F,} 
     \end{aligned}
 \end{equation}
 where $\psi_1$ enters the maximum only if $\alpha<2$ and $\zeta/d \in(2-\alpha, 2-\alpha+1/d)$, and $\psi_2$ only if $\alpha>2, \tau<3$ and $\mu<(1-\zeta)(\tfrac{1}{d}+\tfrac{3-\tau}{d(\alpha-2)})$ holds, and $\psi_3$ only if $\mu+\zeta/d\in (3-\tau, 3-\tau+1/d)$ according to the conditions in Theorem \ref{theorem:summary}. Finally in the pure geometric phase $\psi=\lim_{t\to \infty}\log I(t)/ \log (t^d)=1$ on area $G$ of Figure \ref{fig:phase-diagrams}. 

 \subsection*{Data Availability} Previously published data used for this work: \cite{gowalla}.

\bibliographystyle{plain}
\bibliography{references}

\appendix
\section{Supporting information}

\subsection{Geometric inhomogeneous random graphs: model and basic properties}\label{sec:model-def-prop}

In this section, we give the mathematical definition of the synthetic network we use both for our theoretical results and for the simulations, and explain some of the key properties of these networks as well as the role of the parameters. The model Geometric Inhomogeneous Random Graph (GIRG) was introduced in \cite{bringmann2019geometric}. To obtain slightly simpler formula, we use an equivalent formulation where we blow up the geometric space by a factor of $n^{1/d}$ in each dimension and the edge connection probability by a factor of $n$~\cite{komjathy2020explosion}. 
We say that a random variable follows a power law with exponent $\tau>1$ if it has a probability density function that decays as $\sim x^{-\tau}$ as $x$ tends to infinity. From now on, in a network we use the words nodes vs vertices exchangeably, and also do so with connections/links/edges.

\begin{definition}[GIRG]\label{def:GIRG}
For parameters $\tau>2$, $\alpha>1$, and dimension $d\in\N$, a \emph{Geometric Inhomogeneous Random Graph} $G=(V,E)$ is obtained by the following three-step procedure:
\begin{enumerate}
    \item[(1)] The vertex set $V$ is given by a Poisson Point Process of intensity $1$ on either the torus or the box $[-n^{1/d}/2, n^{1/d}/2]^d$ of volume $n$ centred at the origin.

    \item[(2)] Assign to each vertex $v\in V$ independently and identically distributed \emph{weight} $W_v$ that follows a power-law with exponent $\tau$.

    \item[(3)] Every pair of vertices $u,v$ is then connected by an edge independently with probability
    \begin{align}\label{eq:girg-connection}
        p_{uv} = c\cdot\min\Big\{\frac{W_uW_v}{\|u-v\|^d}, 1\Big\}^{\alpha},
    \end{align}
    where $c\in(0,1]$ is a constant and $\|\cdot\|$ denotes the Euclidean norm.
\end{enumerate}
\end{definition}
A modification of Step (1) is also often used, when the vertex set $V$ is given by $n$ independent points chosen uniformly in the torus or box $[-n^{1/d}/2, n^{1/d}/2]^d$. In the definition above, there are a random number $N$ of points following a Poisson distribution with mean $n$. Given the value of $N$, the points are again independent and have uniform location in the cube. The advantage of our convention is that any smaller or larger box induces a GIRG model again, just with a different size $n$. Thus one can easily move between models of different sizes.  
\vskip1em

\textbf{Key properties.} 

\emph{Degree distribution.}
In GIRG, the degree of a vertex follows a mixed Poisson distribution with mean proportional to its weight; in particular, the degree distribution is a \emph{power-law} with the same exponent $\tau$ as the one used to sample the vertex-weights $W_v$ for each vertex \cite{bringmann2025average}. 

\emph{Connectivity: component structure.} The model naturally generalises to an infinite model on $\R^d$, using a unit intensity Poisson point process as a vertex set. On $\R^d$, we call the model \emph{supercritical} when there is a connected component containing infinitely many vertices of positive density $\theta=\mathbf{Prob}(v \text{ in the infinite component})>0$.  The infinite model is always supercritical when $\tau\in (2,3)$, regardless of the value of the dimension $d\ge 1$ and the long-range parameter $\alpha>1$. When $\tau>3$ and the dimension is $d\ge2$,  then either the constant $c$ in \eqref{eq:girg-connection} has to be sufficiently large, or the average weight $\mathbb E[W]$ has to be sufficiently large for in infinite connected component to appear in GIRG. Both result in `sufficiently high edge-density'. In dimension $1$, the model is always supercritical when $\tau\in (2,3)$, but when $\tau>3$ then we need the sufficiently high edge density and additionally $\alpha\in(1,2]$ for an infinite component to exist.
Moving to finite GIRGs, the existence of an infinite component naturally translates to the existence of a linear connected component of proportion $\theta$ in finite GIRGs with the same parameters, for all $n$ sufficiently large. It is also known that all smaller components are at most poly-logarithmic in size, see \cite[Theorem 1.1]{jorritsma2025cluster}.

\emph{Typical distance in the large component.} GIRG is \emph{small-worlds} for a large range of relevant parameters. The graph distance $d_G(u,v)$ between two vertices $u,v$ in the same component, also called hopcount, is the minimal number of edges on any path between the two vertices.  We call a graph \emph{small world} if the the graph distance between a uniformly chosen node-pair $u,v$ in the largest component is small compared to the total number of vertices. 
There are three universality classes for the hopcount \cite{bringmann2025average,deprez2015inhomogeneous,berger2004lower}:
\begin{align}\label{eq:graph-distance}
    d_G(u,v) = \#\{ e: \text{edge on shortest path between $u,v$}\}=  \begin{cases}
    \Theta(\log\log n) & \mbox{ if $2 < \tau < 3$,}\\
	(\log n)^{\Theta(1)} & \mbox{ if $\tau>3$ and $1<\alpha<2$,}\\
    \Theta(n^{1/d}) & \mbox{ if $\tau>3$ and $\alpha>2$.}
	\end{cases}
\end{align}
So GIRG is small world if either $\alpha<2$ or $\tau<3$, and in the latter case we speak of \emph{ultrasmall worlds}. Moreover, the graph diameter (i.e., the maximal graph distance among all pairs of nodes in the graph) of the largest connected component of GIRG is also small, namely of order $(\log n)^{O(1)}$, for the most interesting parameter range $\tau\in(2,3)$, and also for larger $\tau$ when $\alpha\in(1,2)$ \cite{bringmann2025average, coppersmith2002diameter,benjert2025diameter}. Finally, the average local clustering coefficient is a measure capturing the friend-of-a-friend is typically also a friend and the presence of communities.  This quantity evaluated on a GIRG network is bounded away from $0$ as $n\to\infty$ \cite{bringmann2019geometric}, indicating strong presence of local communities.

\subsection{Edge-length distribution: theoretical results and estimation methods}\label{sec:edge-length}

While GIRGs are increasingly more popular for real life network modelling, currently there is no established statistical estimation method to infer its long-range parameter $\alpha$.
In contrast, for the other key parameter, the power-law exponent $\tau$ governing the distribution of the hidden variables $W$, there are good and reliable estimators, as $\tau$ can be inferred from the degree distribution, see Section \ref{sec:statistical-methods} below. 

As $\alpha$ governs the edge-length distribution in the graph (but often only between low-degree vertices), one can design estimators based on the theoretical results about edge-length statistics. We found that one needs to be somewhat careful, as estimators based on asymptotic results for the large network limit fail to recover the correct $\alpha$ on synthetic GIRG network of size up to $10^5$ nodes in dimension $2$: in that case the side-length of the box is only $\sim300$, and tail edge statistics are sensitive to this truncation effect, while fitting on very short range edges ($10-50$) leaves a large room for error. To overcome this, we compute the edge-length distribution of finite GIRGs in arbitrary dimensions $d\ge 1$ and then use the computed formula for better statistical fit. This adjustments reliably recovers the true $\alpha$ with narrow confidence intervals.

 We work on the $d$-dimensional torus (in the case of the main text $d=2$). 
Before we start with computations, we note an elementary lemma stating that the product of two Pareto random variables sharing the same power-law exponent has a regularly varying tail with the same tail exponent. 
\begin{lemma}\label{lem:product_distribution}
	Let $X, Y$ be two independent and identically distributed positive random variables with cumulative distributions $F_X(x)=1-x^{-\tau}$ and $F_Y(y)=1-y^{-\tau}$, respectively. Then the cumulative distribution of their product $Z:=XY$ is given by $F_Z(z)=1- \ell(z)z^{-\tau}$ with $\ell^*(z)=1+(\tau-1)\log(z)$, while the density of their product is $f_Z(z)=(\tau-1)^2\log (z)/z^{\tau}$ for $z\ge 1$.
\end{lemma}
\begin{proof}
The density of $f_Z(z)$ can be computed using the variable transform $(x, y)\mapsto (x, z=xy)$ with Jacobian $1/x$. As both $x$ and $z/x$ have to be above $1$, and $f(x)=(\tau-1)x^{-\tau}$, we obtain:
\[ f_Z(z)=\int_1^z f(x)f(z/x) x^{-1} \mathrm dx= (\tau-1)^2 \frac{\log z}{z^{-\tau}}.\]
Integrating the density gives the statement. 
	For more general regularly varying distributions in place of Pareto distributions, $\ell^\star(z)=1+(\tau-1)\ln(z)$ can be replaced by another slowly varying function, and is a consequence of~\cite[Corollary, Page 3]{embrechts1980closure}.
\end{proof}

\begin{theorem}[Asymptotic tail of edge-lengths in GIRG.]\label{thm:edge-length}
Consider GIRG on the $d$-dimensional volume-$n$ torus $\Pi_n:=[-n^{1/d}/2,n^{1/d}/2)^d$ with parameters $\tau>2, \alpha>1$. Let $E_n(L_1, L_2)$ denote the number of edges with Euclidean length in the interval $[L_1, L_2]$. Further, let $E_n(L_1, L_2, \le M)$ denote the number of edges with Euclidean length in the interval $[L_1, L_2]$ connecting two vertices $(u,v)$ with weights satisfying $W_uW_v\le M$. Then there exists explicit constants $c_1=c_1(\alpha, \tau, d), c_2=c_2(\alpha, \tau, d), c_3=c_3(\alpha, \tau, d), c_4=c_4(\alpha, \tau, d, M)$ below in \eqref{eq:constant-prefactors} and \eqref{eq:c4M}, so that  for all $L_1<L_2< n^{1/d}/2$ with $L_1>c_5(\alpha, \tau, d)$ for some constant $c_5$, 
\begin{equation}\label{eq:edge-length-l1l2}
\begin{aligned}
\frac{E_n(L_1, L_2)}{n}&\stackrel{\mathbb P}{\longrightarrow}\  \begin{cases}
 c_1 \cdot\Big(L_1^{-d(\alpha-1)} - L_2^{-d(\alpha-1)}\Big)&\text{ if } \alpha<\tau-1\\
     c_2 \cdot \Big( L_1^{-d(\alpha-1)} \log^2(L_1^d) - L_2^{-d(\alpha-1)} \log^2(L_2^d) \Big)&\text{ if } \alpha=\tau-1\\
     c_3\cdot  \Big(L_1^{-d(\tau-2)} \log(L_1^d) - L_2^{-d(\tau-2)} \log(L_2^d)\Big)   &\text{ if } \alpha>\tau-1 
    \end{cases}\\
    \frac{E_n(L_1, L_2, \le M)}{n}& \stackrel{\mathbb P}{\longrightarrow} \ c_4(M)\cdot \Big(L_1^{-d(\alpha-1)} - L_2^{-d(\alpha-1)}\Big) \text{ for all } L_1>M^{1/d}, \text{ for all }\alpha>1,\tau>2,  
    \end{aligned}
\end{equation}
where $X_n \stackrel{\mathbb P}{\longrightarrow} x$ means that $X_n$ converges to $x$ in probability as $n\to \infty$.
\end{theorem}
A corollary of the proof gives the asymptotic edge-length distribution in the large network limit:
\begin{corollary}[Asymptotic tail of edge-lengths in GIRG.]\label{cor:edge-length-limit}
Consider GIRG on the $d$-dimensional volume-$n$ torus $\Pi_n:=[-n^{1/d}/2,n^{1/d}/2)^d$ with parameters $\tau>2, \alpha>1$. Let $E_n(L)$ denote the number of edges longer than $L$ in the graph, and let $E_n(L, \le M)$ denote the number of edges longer than $L$ connecting two vertices $u,v$ with weights satisfying $W_uW_v\le M$. Then with the same explicit constants as in Theorem \ref{thm:edge-length}, for all fixed $L$, as $n\to \infty$, 
\begin{equation}
\begin{aligned}
\frac{E_n(L)}{n}&\  {\buildrel \mathbb P \over \longrightarrow} \ \begin{cases}
 c_1 L^{-d(\alpha-1)} &\text{ if } \alpha<\tau-1\\
 c_2 L^{-d(\alpha-1)} \log^2(L^d)  &\text{ if } \alpha=\tau-1\\
 c_3 L^{-d(\tau-2)}  \log (L^d) &\text{ if } \alpha>\tau-1. 
    \end{cases}\\
    \frac{E_n(L, \le M)}{n}&\  {\buildrel \mathbb P \over \longrightarrow}\ c_4(M) L^{-d(\alpha -1)} \quad  \text{ for all } L>M^{1/d}, \text{ for all }\alpha>1,\tau>2. 
    \end{aligned}
\end{equation}
\end{corollary}

\begin{proof}[Proof of Theorem \ref{thm:edge-length} and Corollary \ref{cor:edge-length-limit}] 
Let $E_n(L_1, L_2)$ denote the number of edges in $[-n^{1/d}/2,n^{1/d}/2]^d$ with length between $L_1, L_2$, with $L_2$ being possibly larger than $n^{1/d}/2$. We start computing the expectation. 
Let

\begin{equation}\label{eq:lambda_r}
	    \Lambda_e(r):=\E[(1 \wedge W_xW_y/r^d)^{\alpha}],
	\end{equation}
    where the randomness is taken over two independent random weights $W_x$, $W_y$ from the weight distribution. By~\eqref{eq:girg-connection}, then $c\Lambda_e(r)$ is the probability that two vertices $x,y$ with fixed position in distance $r$, but with random weights, form an edge. 
	Let us write $B_L(x)$ for the Euclidean ball of radius $L$ centered around $x$, and $B_L$ when the center is $0$. We also write $A_{L_1, L_2}(x):=B_{L_2}(x)\setminus B_{L_1}(x)$ for the annulus centered around $x$ of inner radius $L_1$ and outer radius $L_2$ on the torus, and abbreviate $A_{L_1, L_2} := A_{L_1, L_2}(0)$. Using conditional expectation on the location of vertices $\mathcal V \subseteq [-n^{1/d}/2,n^{1/d}/2)^d$,
	\begin{equation}\label{eq:e-to-lambdae-pre}
    \begin{aligned}
		\mathbb E[E_n(L_1,L_2)\mid \mathcal V] &=
		\sum_{\substack{x,y \in \mathcal V,\|x-y\| \in [L_1,L_2]}}
		\E\left[\mathbf{1}_{\{xy \textnormal{ is an edge}\}}\right] 
        \\&=
		\sum_{\substack{x,y \in\mathcal V ,\|x-y\|\in [L_1, L_2]}}
		\E\left[c\left(1 \wedge \dfrac{W_xW_y}{|x-y|^d}\right) ^\alpha  \right] \\
		&= c \sum_{x \in \mathcal V} \,
		\sum_{y \in \mathcal V\cap A_{L_1, L_2}(x)}
		\Lambda_e(|x-y|).
	\end{aligned}
    \end{equation}
    Taking now expectation over the location of vertices $\calV$, the outer summation gives a factor $n$, while the inner summation becomes an integral using the translation invariance of the torus and the fact that each vertex falls into a given set proportional to the area of the set. 
    We arrive at: 
    \begin{equation}\label{eq:e-to-lambdae}
    \mathbb E[E_n(L)]=cn\int_{y\in [-n^{1/d}/2,n^{1/d}/2)^d\cap  A_{L_1, L_2}} \Lambda_e(|y|) \mathrm dy
    \end{equation}
    We can compute the above integral by switching to polar coordinates. The Jacobian is $r^{d-1}$ on the segment $L_1<r<n^{1/d}/2$ so we get the surface $\mathrm{Surf(d)}$ of a $d$-dimension unit sphere. If $L_2>n^{1/d}/2$ then for $r>n^{1/d}/2$ we only integrate over part of the surface of a unit sphere depending on $r$. We write $L_2\wedge x:=\min(L_2, x)$:
      \begin{equation}\label{eq:enl-integral-form}
    \mathbb E[E_n(L_1, L_2)]=cn \mathrm{Surf}(d)\int_{r=L_1}^{L_2\wedge n^{1/d}/2} r^{d-1} \Lambda_e(r)\,\mathrm{d}r + cn\mathbf 1_{\{L_2> n^{1/d}/2\}}\cdot \int_{r=n^{1/d}/2}^{L_2\wedge n^{1/d}\sqrt{d}/2} r^{d-1}\phi(2r/n^{1/d}) \Lambda_e(r) \mathrm d r
    \end{equation}
  where in general dimension $d$, $\phi_d(u)$ for $u\in [1/2,\sqrt{d}/2]$ is the total surface area of a $d$-dimensional sphere of radius $u$ that is inside the unit cube $[-1/2,1/2]^d$. So for dimension $2$ this is the total arc-length where the circle of radius $u$ intersects the square $[-1/2,1/2]^d$. Note that $\phi_d(2r/n^{1/d})=\mathrm{Surf}(d)$ if $r= n^{1/d}/2$, and it is monotone decreasing and equals $0$ for all $r\ge n^{1/d}\sqrt{d}/2$.
  
	Next we compute  $\Lambda_e(r)$ in 
 \eqref{eq:lambda_r}. Note that $\Lambda(r)$ only depends on the product $W_xW_y=:Z$, not on the individual weights of the two vertices. By Lemma~\ref{lem:product_distribution}, the distribution of $Z$ is of the form $f_Z(z)=\ell(z)z^{-\tau}$, where $\ell(z)=(\tau-1)^2\log (z)$. 
Recall that $r>L_1$ and rewrite  \eqref{eq:lambda_r} using law of total probability as follows, then cut the integral in two according to the minimum:
\begin{equation}\label{eq:lambda-r-detailed}
\begin{aligned}
    \Lambda_e(r)&= \int_{z=1}^{\infty} \left(1\wedge\dfrac{z}{r^d}\right)^\alpha  \cdot \dfrac{\ell(z)}{z^{\tau}}\,\mathrm{d}z = r^{-d\alpha}\int_{z=1}^{r^d} z^\alpha  \cdot \dfrac{\ell(z)}{z^{\tau}}\,\mathrm{d}z + \int_{z=r^d}^{\infty} \dfrac{\ell(z)}{z^{\tau}} \mathrm{d}z\\
    &= r^{-d\alpha} (\tau-1)^2 \Big[\Big(\frac{\log(z)}{(\alpha+1-\tau)}- \frac{1}{(\alpha+1-\tau)^2}\Big)z^{\alpha-\tau+1}\Big]_{z=1}^{z=r^d} + (\tau-1)^2\Big[-\Big(\frac{\log(z)}{\tau-1} + \frac{1}{(\tau-1)^2}\Big)z^{1-\tau}\Big]_{z=r^{d}}^{\infty}\\ 
    \end{aligned}
\end{equation}
Where we assumed that $\alpha\neq \tau-1$. If $\alpha= \tau-1$, then the first integral evaluates to $(\log r^d)^2/2$.

For $r\to\infty$, the limiting behaviour is
\begin{equation}\label{eq:lambdae-approx}
    \Lambda_e(r)\asymp\begin{cases}r^{-d\alpha}(\tau-1)^2/(\tau-1-\alpha)^2(1+o(1))  &\text{ if } \alpha<\tau-1\\
    r^{-d\alpha} [\log^2(r^d)(\tau-1)^2/2 + (\tau-1)\log(r^d)](1+o(1)) &\text{ if } \alpha=\tau-1\\
    r^{-d(\tau-1)}\log(r^d)(\tau-1)\alpha/(\alpha+1-\tau)(1+o(1)) &\text{ if } \alpha>\tau-1 
    \end{cases}
\end{equation}
Observe that $\Lambda_e(r)$ integrates over the vertex-weights. The very precise form of $\Lambda_e(r)$ is of the form $r^{-d\min(\alpha, \tau-1}) (\tilde c_0 \log (r^d)^2 + \tilde c_1\log (r^d) + \tilde c_2)$ for specific constant $\tilde c_0, \tilde c_1, \tilde c_2$. We listed $\tilde c_2$ when $\alpha<\tau-1$ as in this case $\tilde c_0= \tilde c_1=0$, we listed $\tilde c_0, \tilde c_1$ but neglected $\tilde c_2$ when $\alpha=\tau-1$, and we listed $\tilde c_1$ when $\alpha>\tau-1$ as in this case $\tilde c_0=0$, but neglected $\tilde c_2$. This is justified by the following reasoning: as in the next integral we integrate over $r$ values $r\in[L_1, L_2]$, when $c_1\neq 0$ then we can neglect $\tilde c_2$, as these are negligible compared to $\tilde c_1 \log(r^d)$ when $L_1$ is not too small. In other words, the numerical error caused by the approximation in \eqref{eq:lambdae-approx} will rather present itself in small values for $L_1$, but it will get better and better as $L_1$ increases. More precisely, one can check that if $\alpha>\tau-1$ then the neglected constants $\tilde c_2=(\tau-1)^2[1/(\alpha +1-\tau)^2+1/(\tau-1)^2]$ is smaller than $\log (r^d)\alpha(\tau-1)/(\alpha+1-\tau)$ whenever $\log r^d > \alpha+ (\alpha+1-\tau)^{-1}$.

Nevertheless the negligence only affects the logarithmic correction term and the constant prefactor, but it leaves the main polynomial decay of $\Lambda_e(r)$ unchanged.
We will also see numerically that this negligence does not the cause main errors. 
Using these bounds in \eqref{eq:enl-integral-form}, we arrive at the conclusion that the first integral evaluates to, for $L_1<n^{1/d}/2$, up to a $(1+o(1))$ factor for large $L_1$,
\begin{equation}\label{eq:edge-length-l1l2-2}
cn \mathrm{Surf}(d)\int\limits_{r=L}^{L_2\wedge n^{1/d}/2} r^{d-1} \Lambda_e(r)\,\mathrm{d}r \asymp \begin{cases}
 n c_1 \Big(L_1^{-d(\alpha-1)} - (L_2\wedge \tfrac{n^{1/d}}{2})^{-d(\alpha-1)}\Big)&\text{ if } \alpha<\tau-1\\
    n c_2 \Big( L_1^{-d(\alpha-1)} \log^2(L_1^d) - (L_2\wedge \tfrac{n^{1/d}}{2})^{-d(\alpha-1)} \log^2(L_2\wedge\tfrac{n^{1/d}}{2}) \Big)&\text{ if } \alpha=\tau-1\\
    n c_3 \Big(L_1^{-d(\tau-2)} \log(L_1^d) - (L_2\wedge \tfrac{n^{1/d}}{2})^{-d(\tau-2)} \log(L_2 \wedge \tfrac{n^{1/d}}{2})   &\text{ if } \alpha>\tau-1 
    \end{cases}
\end{equation}
where 
\begin{equation}\label{eq:constant-prefactors}
\begin{aligned}
c_1&=c\cdot \mathrm{Surf}(d)(\tau-1)^2/\Big((\tau-1-\alpha)^2 d(\alpha-1)\Big)\\
c_2&=c\cdot \mathrm{Surf}(d)(\tau-1)^2/\Big(2 d(\alpha-1)\Big)\\
c_3&=c\cdot \mathrm{Surf}(d)\Big((\tau-1)+(\tau-1)^2/(\alpha+1-\tau)\Big)/(d(\tau-2)).
\end{aligned}
\end{equation}
For $L_2<n^{1/d}/2$, this finishes the proof. 
For $L_2>n^{1/d}/2$, the contribution of the second integral in \eqref{eq:enl-integral-form} can be bounded as follows. For any given $L_1<n^{1/d}/2$ and $L_2 \ge n^{1/d}/2$, we can use the upper bound $\phi(r/(n^{1/d}/2))\le \mathrm{Surf}_d$ up to $r\le n^{1/d}\sqrt{d}/2$ and then set  $\phi(r/(n^{1/d}/2))=0$ for $r\ge n\sqrt{d}/2$. Then the formulas in \eqref{eq:edge-length-l1l2-2} gives a lower bound, while a corresponding similar upper bound is obtained by replacing $L_2\wedge n^{1/d}/2$ everywhere in the above formulas by $L_2 \wedge (\sqrt d n^{1/d}/2)$. If $L_1<\varepsilon n^{1/d}$ for a fixed $\varepsilon>0$,  then as $n \to \infty$ the contribution of the terms containing $L_2$  becomes negligible as a function of $\varepsilon$, i.e., they tend to zero with vanishing $\varepsilon$.

This shows the first statement in Theorem~\ref{thm:edge-length} and Corollary \ref{cor:edge-length-limit} in expectation. As the edge set can be looked as a thinned process on \emph{pairs} of points in a Poisson point process, the number of edges in a given length interval is a mixed Poisson distribution with mixing variable $W_v$. It concentrates around its expectation by an Azuma-type concentration inequality~\cite[Theorem~6]{bringmann2025average}, see the proof of~\cite[Theorem~9]{bringmann2025average} for how to apply it. 
 This gives the proof of the first statement in Theorem~\ref{thm:edge-length} and Corollary \ref{cor:edge-length-limit}.

\emph{Edge-length distribution restricted to low-weight vertices.}
Here we modify the formulas and compute the expected number of edges longer than $L$  with endpoints having $W_uW_v<M$. This is a slightly weaker restriction than restricting to the subgraph spanned by vertices of weight $<\sqrt{M}$. We will assume $L_1>M^{1/d}$.  The method above is valid, with a modified formula for $\Lambda_e(r)$ in \eqref{eq:lambda_r}.
Namely, \eqref{eq:lambda-r-detailed} becomes, for $r^d>L_1^d>M>1$,
\begin{equation}\label{eq:lambda-r-detailed-2}
\begin{aligned}
    \Lambda_{e,<M}(r)&= \int_{z=1}^{M} \left(1\wedge\dfrac{z}{r^d}\right)^\alpha  \cdot \dfrac{\ell(z)}{z^{\tau}}\,\mathrm{d}z = r^{-d\alpha}\int_{z=1}^{M} z^\alpha  \cdot \dfrac{\ell(z)}{z^{\tau}}\,\mathrm{d}z \\
    &= r^{-d\alpha} (\tau-1)^2 \Big[\Big(\frac{\log(z)}{(\alpha+1-\tau)}- \frac{1}{(\alpha+1-\tau)^2}\Big)z^{\alpha-\tau+1}\Big]_{z=1}^{z=M} \\
    &= r^{-d\alpha}\Big(\frac{(\tau-1)^2}{(\alpha+1-\tau)^2}+M^{\alpha-\tau+1 }(\tau-1)^2 \frac{ \log M -1/(\alpha+1-\tau)}{\alpha+1-\tau}\Big)=: \tilde c_4(M) \cdot r^{-d\alpha}.
    \end{aligned}
\end{equation}
Based on whether $\alpha<\tau-1$ or $\alpha>\tau-1$, the sign of the term containing $M^{\alpha-\tau+1 }\log M$ is negative/positive, respectively. For $\alpha=\tau-1$, the coefficient $\tilde c_4(M)$ of $r^{-d\alpha}$ above changes to $(\tau-1)^2\log^2(M)/2$. Using this function in \eqref{eq:enl-integral-form} we obtain that the first integral in \eqref{eq:enl-integral-form} turns into
\begin{equation}\label{eq:edge-lengths-M}
   cn^d \mathrm{Surf}(d)\int_{L_1}^{L_2\wedge n/2}r^{d-1} \Lambda_{e,\le M}(r)\mathrm dr= cn^d \mathrm{Surf}(d) \tilde c_4(M) \Big[-\frac{r^{d(1-\alpha)}}{d(\alpha-1)}\Big]_{r=L_1}^{r=L_2\wedge \tfrac{n}{2}} = n^d c_4(M) \Big( L_1^{-d(\alpha-1)} - (L_2\wedge \tfrac{n}{2})^{-d(\alpha-1)}\Big).
\end{equation}
with 
\begin{equation}\label{eq:c4M}
c_4(M):=\frac{c\cdot \mathrm{Surf}(d)\tilde c_4(M)}{d(\alpha-1)} = \frac{c\cdot \mathrm{Surf}(d)}{d(\alpha-1)}\cdot \Big(\frac{(\tau-1)^2}{(\alpha+1-\tau)^2}+M^{\alpha-\tau+1 }(\tau-1)^2 \frac{ \log M -1/(\alpha+1-\tau)}{\alpha+1-\tau}\Big)
.
\end{equation}
This finishes the proof of Theorem \ref{thm:edge-length} for $L_2\le n/2$, and gives also a lower bound for $L_2\ge n/2$. For $L_2> n/2$, the contribution of the second integral to $\mathbb E[E_n(L_1, L_2, \le M)]$ can be again upper bounded using that $\phi_d(r/(n/2))\le \mathrm{Surf}_d$ for  $r\in[n/2, \sqrt{d}n/2]$ and  $\phi_d(r/(n/2))=0$ for all $r\ge \sqrt{d}n/2$. 
This gives the upper bound 
\[ 
\mathbb E[E_n(L_1, L_2, \le M)]\le  n^d c_4(M) \Big( L_1^{-d(\alpha-1)} - (L_2\wedge \sqrt{d}\tfrac{n}{2})^{-d(\alpha-1)}\Big),
\]
and finishes the proof.
\end{proof}

\subsection{One-dependent SI on GIRG}\label{sec:1-SI}
In this section we define the transmission dynamics of the SI-epidemic and state the main theorems about transmission time between far away vertices. We emphasize that in the SI model, every link $(u,v)$ of the network only transmits the disease at most once, when one endpoint is still infected and the other end is still susceptible, while the transmission time itself is random and can be represented as an exponential random variable $Y_{uv}$ with rate $r(u,v)$. $I(t)$ is then the set of nodes that got infected before time $t$. Equivalently, one may sample the random transmission times $Y_{uv}$ \emph{in advance}, store these variables and use them (only) when they are needed for the infection process. We also make use of the probabilisitc identity that an exponential variable of rate $r$ can be represented as an exponential random variable of unit rate, multiplied by the number $1/r$. 
Sampling the transmission times $Y_{uv}$ at the moment of the infection of $u$ or in advance does not change the course of the pandemic, but it allows us to switch viewpoints. This gives rise to the following definition, that is commonly known as `first passage percolation' in the mathematical literature \cite{auffinger201750}.
Below, we switch from the term \emph{transmission time} to \emph{transmission cost}, and use these phrases equivalently.

\begin{definition}[1-dependent SI (1-SI)]\label{def:1-FPP}
Consider a GIRG $G=(V,E)$ with the associated sequence of vertex-weights $(W_v)_{v\in V}$. For every edge $uv\in E$, draw an i.i.d.\ exponential random variable $Y_{uv} \sim \mathrm{Exp}(1)$, and set the \emph{(transmission) cost} of the edge $uv$ as 
\begin{equation}\label{eq:cost}
\cost{uv} := Y_{uv}\cdot(W_uW_v)^{\mu}\|u-v\|^{\zeta},    
\end{equation}
for fixed parameters $\mu,\zeta\ge0$. The costs define a \emph{cost-distance} $d_{\calC}(u,v)$ between any two vertices $u$ and $v$, which is the minimal total cost of any path between $u$ and $v$. We call $d_{\calC}$ the 1-dependent first passage percolation.
\end{definition}
Note that the cost-distance between any two nodes $u,v$ is the same as the transmission time of the SI epidemic to $v$, assuming that $u$ is the source vertex. The advantage of the cost-distance definition is that we do not need to specify the source node in advance.

With this definition, we can state our main theoretical results, which gives an (almost)  complete characterisation of how the cost-distance between two vertices scales as a function of their geometric distance. Here and below, we abbreviate with high probability by whp, and by this we mean that the probability of the respective event tends to $1$ in the large network limit.
In this theorem, we write $f(x)= \Theta(g(x))$ or $f(x)=O(g(x))$ if there exists strictly positive and strictly finite constants $c_1, c_2$ so that $c_1<f/g<c_2$, or $f/g < c_2$, respectively. We also write $f(x)=o(g(x))$ if $f/g$ tends to $0$ as $x$ tends to infinity.
Recall that  a supercritical GIRG has a linear sized connected component. 
\begin{theorem}\label{thm:main-supporting}
Consider $1$-SI as in Definition \ref{def:1-FPP} with parameters $\mu, \zeta\ge 0$ on a supercritical GIRG $G=(V,E)$ from Definition \ref{def:GIRG} on $n$ vertices with parameters $\tau>2, \alpha>1, d\ge 1$. Let $u_0 \in V$ denote the vertex at the origin and let $v\in V$ be a vertex\footnote{As we work on the translation invariant torus wlog we can assume that there is a vertex at $0$. The theorem holds for arbitrary vertices $v$ whose choice is predetermined before the generation of the graph. E.g. one may choose $v$ to be the closest vertex to the endpoint of the vector $f(n) \underline e$ for some function $f(n)<n^{1/d}/2$ and $\underline e$ a unit vector, or one may choose $v$ uniformly among all vertices.} , and let $\|v\|$ be its Euclidean distance from $0$. Then the transmission time $d_{\calC}(u_0,v)$ behaves as follows, on the event that both $u_0$ and $v$ are in the largest component.
  \begin{enumerate}
\setlength\itemsep{0em}  
    {\setlength\itemindent{25pt}
    \item[Phase (i)] If $\mu + \zeta/d < (3-\tau)/2$, then $d_{\calC}(u_0,v) = \Theta(1)$, i.e., the transmission time is not growing with the network size~$n$.\footnote{Formally, the probabilistic statement is that for every $\eps >0$ there is $C>0$ such that the bound $\Pr[d_{\calC}(u_0,v) > C] \le \eps$ holds uniformly for all $n\in \mathbb{N}$.}
    }
    {\setlength\itemindent{25pt}
    \item[Phase (ii)] 
    \begin{enumerate}
    \item[b)] If $\zeta/d < 2-\alpha$, then whp $d_{\calC}(u_0,v) \le (\log\|v\|)^{\Delta_1+o(1)}$, where $\Delta_1 := 1/(1-\log_2(\alpha+\zeta/d))$.
    {\setlength\itemindent{25pt}
    \item [c)] If $\mu+\zeta/d < 3-\tau$, then whp $d_{\calC}(u_0,v) \le (\log\|v\|)^{\Delta_2+o(1)}$, where $\Delta_2 := 1/(1-\log_2(\tau-1+\mu+\zeta/d))$.}
    \end{enumerate}
    }
    {\setlength\itemindent{25pt}
    \item[Phase (iii)] 
    \begin{enumerate}
    \item[d)] If $\alpha<2$ and $\zeta/d \le  2-\alpha +1/d$, then whp $d_{\calC}(u_0,v) \le \|v\|^{\eta_1+ o(1)}$, where $\eta_1 := \zeta-d(2-\alpha)$.
    {\setlength\itemindent{25pt}
    \item[e)] If $\alpha>2$, $\tau<3$ and  $\mu \le  (1-\zeta) \cdot (\frac{1}{d}+ \frac{3-\tau}{d(\alpha-2)})$, then whp $d_{\calC}(u_0,v) \le  \|v\|^{\eta_2\pm o(1)}$, where $\eta_2 := \zeta + \mu d \frac{\alpha-2}{\alpha-\tau+1}$.}
    {\setlength\itemindent{25pt}
    \item[f)] If $\tau<3$ and $\mu+\zeta/d \le 3-\tau+1/d$, then whp $d_{\calC}(u_0,v) \le \|v\|^{\eta_3+o(1)}$, where $\eta_3 := \zeta+\mu d-d(3-\tau)$.}
    \end{enumerate}
    }
    {\setlength\itemindent{25pt}
    \item[Phase (iv)] If none of the above systems of inequalities are satisfied on the parameters $\alpha, \tau, d, \mu, \zeta$, then if additionally $d\ge 2$ then whp $d_{\calC}(u_0,v)\le C\|v\|$ for some constant $C$. 
    
    For dimension $1$, if none of the above systems of inequalities are satisfied then  $d_{\calC}(u_0,v) \le  \|v\|^{1+o(1)}$. (We recall that the model is supercritical in dim $1$ if either $\tau<3$ or $\alpha<2$).
}
\end{enumerate}  
\end{theorem}
Note that Theorem \ref{thm:main-supporting} gives -- except in Phase (iv) -- only upper bounds on the transmission times. As there is some overlap between the conditions of each subphase, one needs to take the sharpest upper bound for each given set of parameters --  i.e., the one with the smallest exponent $\Delta$ or $\eta$ -- so that the conditions of the bound are satisfied by the specific parameters. This gives rise to the optimal upper bound exponent $\eta_\star$ that we shortly define.  In the polynomial growth Phase (iii) and in the pure geometric Phase (iv), we also have matching lower bounds. To be able to state the matching lower bound, we need to define, for each value of $\alpha, \tau, d, \mu, \zeta$ the optimal exponent.  The following definition makes this precise:  
\begin{definition}[Optimal spreading exponents]\label{def:lower-exponent}
Consider the parameters $\tau>2, \alpha>1, d\ge 1$ and $\mu, \zeta\ge 0$. Assume that $\zeta/d>2-\alpha$, and $\mu+\zeta/d>3-\tau$. Define the optimal upper-bound exponent $\eta_\star$  and an optimal lower-bound exponent $s_\star$ as
\begin{equation}\label{eq:eta-star}
\begin{aligned}
\eta_\star&:=\min\big\{\eta_1^{\mathbf 1\{\alpha<2,\, \zeta/d\le 2-\alpha+1/d\}},\eta_2^{\mathbf 1\{\alpha>2,\, \tau<3, \,\mu \le (1-\zeta)(1/d + (3-\tau)/(d(\alpha-2))\}}, \eta_3^{\mathbf 1\{\tau<3,\, \mu+\zeta/d \le 3-\tau+1/d\}}\big\}\\
s_\star&:=\min\Big\{\frac{\eta_1}{\mathbf 1\{\alpha\le 2\}},\frac{\eta_2}{\mathbf 1\{\alpha>2,\, \tau\le 3\}}, \frac{\eta_3}{\mathbf 1\{\tau\le 3\}}\Big\}.
\end{aligned}
\end{equation}
where $\mathbf 1\{a<b, \dots, c<d \}=1$ if the listed inequalities $a<b, \dots, c<d$ are all satisfied and otherwise it equals $0$. 
\end{definition}
 The conditions $\zeta/d>2-\alpha$ and $\mu+\zeta/d>3-\tau$ exclude the possibility of (the much faster) quasi-exponential spreading. The formula for $\eta_\star$ filters out the minimal exponent among $\eta_1, \eta_2, \eta_3$ respecting the conditions for the exponent to `exist' in Phase (iii), and otherwise it gives $\eta_\star=1$, implying that we are in the pure geometric growth Phase (iv). The optimal lower-bound exponent $s_\star$ is related to the lower bound. As we will see below in Claim \ref{claim:admissible}, $\eta_\star=\min(s_\star, 1)$. We now state the corresponding lower bound for the polynomial and pure geometric spreading. 
\begin{theorem}[Corresponding lower bound on 1-SI]\label{thm:main-lower}
Consider the same setting as in Theorem \ref{thm:main-supporting} and assume $\zeta/d>2-\alpha$, $\mu+\zeta/d>3-\tau$. If $s_\star\le 1$ then whp $d_{\mathcal C}(u_0,v)\ge \|v\|^{s_\star - o(1)}$. Further, if $s_\star >1 $ then whp $d_{\mathcal C}(u_0,v)\ge c \|v\|$ for some constant $c>0$ for all dimensions $d\ge 1$.
\end{theorem}
\begin{remark}[Relation between upper and lower bound] Below, in Claim \ref{claim:admissible} we prove that $\eta_\star=\min(s_\star,1)$. This means that up to a $o(1)$ term in the exponent corresponding to logarithmic correction terms, the upper and the lower bounds in the polynomial growth phase (iii) are matching up to $o(1)$ additive term in the exponent, as in phase (iii) $\eta_\star=s_\star<1$. Further, in the pure geometric phase Theorems \ref{thm:main-supporting} and \ref{thm:main-lower} are matching up to constant factor. Finally, Theorem \ref{thm:main-lower} also gives lower bounds when $\alpha=2$ and/or when $\tau=3$. 
In this sense,  the lower bound is stronger than the upper bound, as it includes the phase boundaries $\alpha=2$ and $\tau=3$. The lower bound is valid also in dimension $1$.  However, when  $\tau>3, \alpha>2$, it holds trivially as $0 $ and $v$ are unlikely to be in the same small component, see \cite{deijfen2013scale}.
\end{remark}
\begin{remark}
When $v$ is chosen uniformly at random in the box or torus $[-n^{1/d}/2, n^{1/d}/2]^d$ in the above theorem, whp the Euclidean norm of $v$ satisfies $\|v\|\sim n^{1/d}$. Theorems \ref{thm:main-supporting}-\ref{thm:main-lower} can be translated to get the growth of the number of infected vertices $I(t)$ as explained in the main text. Generally, if one has a monotone increasing function $f$ with inverse $f^{(-1)}$ so that $d_{\mathcal C}(0,v)\asymp f(\|v\|)$, then for sufficiently large $t$ it holds in the infinite network that 
\[ 
I(t)=\{v: d_\mathcal C(u,v)\le t\}\sim \{v: f(\|v\|)\le t\} \sim f^{(-1)}(t)^d,
\]
\end{remark}
For finite networks one observes the growth $I(t)\sim f^{(-1)}(t)^d$ after the end of the initial phase where the pandemic exists its local neighbourhood, and before saturation effects become significant. In the statement above, the $\sim$ symbol can be made explicit based on the exact form of $f$: for polynomial $f$, one recovers the growth exponent $\psi= 1/\eta$ in $t^{d\psi(1+o(1))}$ but not constant prefactors or additive shifts. For quasi-exponential growth one recovers the stretch exponent $\varphi=1/\Delta$ in  $\exp( t^{\varphi+o(1)})$, but not constant prefactors in the exponent.

\subsubsection*{Phase transitions} 
The exact values of the $\Delta_i$s and $\eta_i$s enable us to unfold the \emph{phase transitions} between the four phases (i)--(iv).
Our results imply that the curves $\mu+\zeta/d=(3-\tau)/2$ and $\mu+\zeta/d = 3-\tau$ are phase boundaries for networks with power-law exponent $2<\tau<3$. 
Remarkably, it is rather `rare' to find strictly exponential growth (with $\Delta=1$) in synthetic geometric network models. Nevertheless this is conjectured when $\tau>3$, $\alpha=1$ and $\zeta=0$, and is proven for the case $\tau=\infty$ (i.e.\ Poisson degree distribution), $\alpha=1$ and $\zeta=0$ in \cite{trapman2010growth}. This is consistent with our formula for $\Delta$ as $\alpha+\zeta/d\downarrow 1$ and $\tau>3$ is fixed, as we also obtain $\Delta_1\to 1/(1-\log_2(1))=1$.
Interestingly, when we approach the phase boundary of explosion Phase (1) (Area A on Figures \ref{fig:enlarged-gowalla-phase-diagram}, \ref{fig:enlarged-mu-zeta-small-alpha}, \ref{fig:enlarged-mu-zeta-large-alpha}, \ref{fig:enlarged-alpha-tau}) , from the hub-dominated part of the quasi-exponential phase (Area C on the same figures) by letting $\mu+\zeta/d \downarrow (3-\tau)/2$ (and keeping $\alpha>2$ fixed), then $\Delta_2$ does not converge to $1$, but to
$1/(2-\log_2(\tau+1)) =:\Delta_{\tau}$. 
So, for the whole range 
$\tau\in(2,3)$, 
$\Delta_{\tau} \ge 1/(2-\log_2(3))>2.4>1$.
This leaves two possibilities: either our upper bound on $\Delta$ is not sharp; or the early growth shows a discontinuous phase transition between the explosive and the quasi-exponential phases. As  either $\tau\uparrow 3$, or as $\tau+\zeta/d+\mu\uparrow 3$, both $\Delta_\tau$ and $\Delta_2$ approach infinity, indicating that we are nearing the polynomial phase. Similarly, when letting either $\zeta/d\downarrow 2-\alpha$ or $\mu+\zeta/d \downarrow 3-\tau$, and we approach the phase boundary between polynomial and quasi-exponential growth from the polynomial side (that is, we approach the boundaries between areas D and B or the boundary between areas C and F),  then the exponents $\eta_1, \eta_3$ in the polynomial phase tend to $0$. This indicates that the transition between the quasi-exponential phase and the polynomial phase is smooth for fixed $\tau\in(2,3)$. 

We explain now the boundary between area D and G on Figure \ref{fig:enlarged-alpha-tau}. Here, Theorem \ref{thm:main-supporting} proves a gap in the polynomial regime as $\tau$ crosses $3$, the threshold between infinite and finite degree variance. Namely, as $\tau\uparrow 3$ while keeping $\alpha>2$ fixed, the exponents $\eta_1,\eta_2$ approach $\zeta+\mu d$. Further, since we assume $\mu+\zeta/d<3-\tau+1/d$, we have that $\zeta+\mu d <1$ with strict inequality. Hence, when $\alpha>2$ and throughout the whole range $\tau\in (2,3)$, the exponent $\eta$ stays strictly below $1$ (and tends to $\zeta+\mu d$ as $\tau\uparrow3$).
On the other hand, as soon as $\tau$ gets larger than $3$ we get in the purely geometric phase (iv), where the exponent of the cost-distance is $1$.  
This indicates a non-smooth phase transition of the cost-distance as $\tau$ crosses $3$ for fixed $\alpha>2$. This is not so surprising since $\tau=3$ is exactly the phase transition between ultra-small and linear graph distances (for $\alpha>2$) in the network, see \eqref{eq:graph-distance}.

\subsection{Proof sketch of the polynomial and quasi-exponential phases}
In this section we provide more details on the proof of phases $(ii)-(iii)$ of Theorem \ref{thm:main-supporting}. Our goal is to find a path $\pi(u_0,v)$ present in the network between $u_0$ and any far away node $v$  with cost (transmission time)  $\mathcal C(\pi)$ that satisfies the desired bound of Theorem \ref{thm:main-supporting}. While this path may not provide the actual shortest path, it does give a possible infection path, and thus serves as an upper bound on the actual (shortest) infection path, so we may conclude our theorem.

We are going to \emph{construct} a path $\pi$ in $G$ between $0$ and $v$ such that its cost $\cost{\pi}$ satisfies the corresponding upper bound. Instead of a sequential construction starting from the source node $0$ and adding edges one-by-one to the path, we rather build the path starting from its approximate `middle' corresponding to the longest edge qua Euclidean length, then connecting both endpoints of this edge to $u_0$ and $v$, respectively, in an iterative manner. 

Before diving into the proof, we introduce some useful notation. Recall that each vertex is equipped with a vertex-weight $W_v$ and that the transmission rate also depends on this weight. For a vertex $v\in V$, a length $r>0$ and a weight interval $I_w \subset [1,\infty)$, we denote by $B_r(v)$ the box of side length $r$ centred at $v$ (intersected with the ground space $[-n^{1/d}/2, n^{1/d}/2]$), and by $B_r(v) \times I_w$ the set of vertices located in $B_r(v)$ with weight in $I_w$. For a subset $\pi \subset E$, we denote by $\cost{\pi} := \sum_{e\in\pi}\cost{e}$ its total transmission time or \emph{cost}.

\subsubsection*{Idea of the iterative construction}
Despite the different scaling of $\cost{\pi}$ in these two phases, the same approach is used to construct $\pi$. The idea is the following. For some constant $\gamma = \gamma(d, \mu, \zeta, \tau, \alpha)\in(0,1)$ (which will be optimised to yield the best possible bound), we search for a cheap edge connecting the balls $B_{r_1}(0)$ and $B_{r_1}(v)$ of radius $r_1=\|v\|^{\gamma}$. More precisely, we actually need to find a cheap edge $b_0b_v$ connecting some vertices $b_0\in B_{r_1}(0), b_v\in B_{r_1}(v)$ such that these vertices are themselves connected to low-weight vertices via cheap edges $b_0y_0\subset B_{r_1}(0)$ and $b_vy_v\subset B_{r_1}(v)$. Indeed, if we just require the edge $b_0b_v$ to be cheap, this could lead to choosing high-weight vertices $b_0,b_v$ such that any other outgoing edge has a too-large cost because of the weight penalty. The three-edge path $y_0b_0b_vy_v$ will be now roughly in the middle of the construction.  

After having found such a $3$-edge path $y_0b_0b_vy_v$ with low-weight end nodes $y_0, y_v$ connecting $B_{r_1}(0)$ and $B_{r_1}(v)$, 
we may now search for two paths connecting $0$ and $y_0$, and connecting $v$ and $y_v$, respectively. In particular, the distance between the node-pairs we need to connect is now at most  $\frac{\sqrt{d}}{2}r_1 \le \frac{\sqrt{d}}{2}\|v\|^{\gamma}$, as $\gamma <1$ and $\|v\|$ large, is much smaller than the original distance $\|v\|$ (between $0$ and $v$). Moreover, we can use the same idea to connect these pairs of nodes, which allows us to find cheap $3$-edge paths connecting some low-weight vertices in $B_{r_2}(0)$ and $B_{r_2}(y_0)$, with $r_2=r_1^\gamma=\|v\|^{\gamma^2}$. Also, we need cheap $3$-edge paths between  $B_{r_2}(v)$ and $B_{r_2}(y_v)$, respectively. We keep iterating this procedure for some appropriately chosen value $k=k(\|v\|, d, \mu, \zeta, \tau, \alpha)$ iterations, until we can connect the remaining gaps using linear-cost paths that have negligible costs. See Figure \ref{fig:hierarchy} for an illustration of the first and second iterations of this construction. We note that after each iteration the number of new paths to be found doubles, thus after $k$ iterations we have to fill $2^k$ `gaps'.

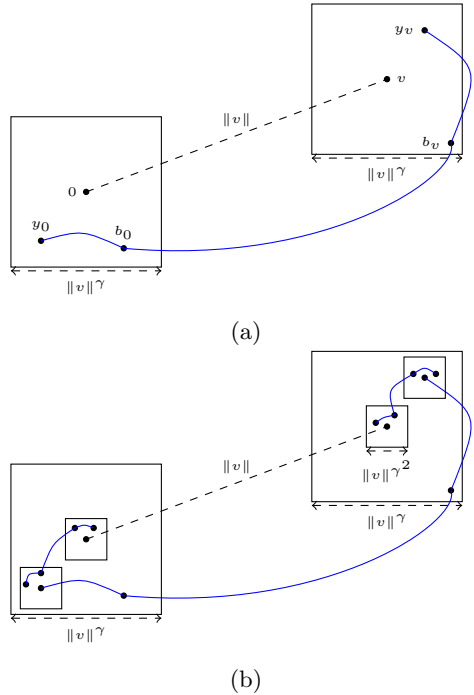
\begin{figure}[!h]
\centering

\begin{minipage}[c]{0.5\textwidth}
\centering

\begin{tikzpicture}[scale=0.5]
    \filldraw (0,0) circle (2pt) node[left] {\tiny $0$};
    \filldraw (8,3) circle (2pt) node[right] {\tiny $v$};
    \draw[dashed] (0,0) -- (8,3);
    \draw (4,1.5) node[above] {\tiny $\|v\|$};
        
    \draw (-2,-2) rectangle (2,2);
    \draw[dashed, <->] (-2,-2.1) -- (2,-2.1);
    \draw (0,-2.1) node[below] {\tiny $\|v\|^{\gamma}$};
    \draw (6,1) rectangle (10,5);
    \draw[dashed, <->] (6,0.9) -- (10,0.9);
    \draw (8,0.9) node[below] {\tiny $\|v\|^{\gamma}$};
        
    \filldraw (1,-1.5) circle (2pt) node[above] {\tiny $b_0$};
    \filldraw (9.7,1.3) circle (2pt) node[left] {\tiny $b_v$};
    \draw[blue] (1,-1.5) .. controls (6.5,-2) and (10,0.5) .. (9.7,1.3);
    \filldraw (-1.2,-1.3) circle (2pt) node[above] {\tiny $y_0$};
    \draw[blue] (1,-1.5) .. controls (-0.1,-1) .. (-1.2,-1.3);
    \filldraw (9,4.3) circle (2pt) node[left] {\tiny $y_v$};
    \draw[blue] (9,4.3) .. controls (10.5,3.3) .. (9.7,1.3);
\end{tikzpicture}
\subcaption{}
\label{fig:hierarchy-first-iter}
\end{minipage}
\begin{minipage}[c]{0.5\textwidth}
\centering

\begin{tikzpicture}[scale=0.5]
    \filldraw (0,0) circle (2pt);
    \filldraw (8,3) circle (2pt);
    \draw[dashed] (0,0) -- (8,3);
    \draw (4,1.5) node[above] {\tiny $\|v\|$};
        
    \draw (-2,-2) rectangle (2,2);
    \draw[dashed, <->] (-2,-2.1) -- (2,-2.1);
    \draw (0,-2.1) node[below] {\tiny $\|v\|^{\gamma}$};
    \draw (6,1) rectangle (10,5);
    \draw[dashed, <->] (6,0.9) -- (10,0.9);
    \draw (8,0.9) node[below] {\tiny $\|v\|^{\gamma}$};
        
    \filldraw (1,-1.5) circle (2pt);
    \filldraw (9.7,1.3) circle (2pt);
    \draw[blue] (1,-1.5) .. controls (6.5,-2) and (10,0.5) .. (9.7,1.3);
    \filldraw (-1.2,-1.3) circle (2pt);
    \draw[blue] (1,-1.5) .. controls (-0.1,-1) .. (-1.2,-1.3);
    \filldraw (9,4.3) circle (2pt);
    \draw[blue] (9,4.3) .. controls (10.5,3.3) .. (9.7,1.3);
    
    \draw (-0.55,-0.55) rectangle (0.55,0.55);
    \draw (-1.75,-1.85) rectangle (-0.65,-0.75);
    \draw (7.45,2.45) rectangle (8.55,3.55);
    \draw[dashed, <->] (7.45,2.35) -- (8.55,2.35);
    \draw (8,2.35) node[below] {\tiny $\|v\|^{\gamma^2}$};
    \draw (8.45,3.75) rectangle (9.55,4.85);  

    \filldraw (-1.6,-1.2) circle (2pt);
    \filldraw (-1.2,-0.9) circle (2pt);
    \draw[blue] (-1.6,-1.2) .. controls (-1.55,-0.9) .. (-1.2,-0.9);
    \filldraw (-0.3,0.3) circle (2pt);
    \filldraw (0.2,0.3) circle (2pt);
    \draw[blue] (-0.3,0.3) .. controls (-0.05,0.4) .. (0.2,0.3);
    \draw[blue] (-0.3,0.3) .. controls (-0.9,-0.1) .. (-1.2,-0.9);
    \filldraw (7.7,3.1) circle (2pt);
    \filldraw (8.2,3.3) circle (2pt);
    \draw[blue] (7.7,3.1) .. controls (7.9,3.25) .. (8.2,3.3);
    \filldraw (8.7,4.4) circle (2pt);
    \filldraw (9.3,4.4) circle (2pt);
    \draw[blue] (8.7,4.4) .. controls (9,4.6) .. (9.3,4.4);
    \draw[blue] (8.7,4.4) .. controls (8,4) .. (8.2,3.3);
\end{tikzpicture}
\subcaption{}
\label{fig:hierarchy-second-iter}
\end{minipage}

\caption{Illustration of the $3$-edge paths constructed during the first (a) and second (b) iterations of the procedure.}   
\label{fig:hierarchy}

\end{figure}

\subsubsection*{Probability computation for one bridge}
After $i$ iterations of the construction described above, we are left with $2^i$ gaps: we need to connect pairs of (low-weight) nodes of the form $(z,z')$ with $\|z-z'\|=O(r_i)$ with $r_i =\|v\|^{\gamma^i}$. For such a pair $(z,z')$, we want to find a cheap $3$-edge bridge $y_zb_zb_{z'}y_{z'}$ with $b_z,y_z\in B_{r_{i+1}}(z)$, $b_{z'},y_{z'}\in B_{r_{i+1}}(z')$, and $y_z,y_{z'}$ having low weight. For the sake of simplicity, we focus on the easier task of finding a cheap edge $b_zb_{z'}$ with $b_z\in B_{r_{i+1}}(z)$, $b_{z'}\in B_{r_{i+1}}(z')$, but without restrictions on the weights of these vertices. It turns out that, for a given cost-bound $K$ and failure probability $\eps$, whenever such an edge of cost $\Theta(K)$ can be found with probability $\ge 1-\eps$, then a $3$-edge path of cost $\Theta(K)$ satisfying the above conditions can also be found with probability $\ge 1-O(\eps)$, see \cite{komjathy2023four}, which permits this simplification. 

For a fixed factor $x=x(d, \mu, \zeta, \tau, \alpha)\in[0,1]$ (which will be optimised jointly with $\gamma=\gamma(d, \mu, \zeta, \tau, \alpha)$ to yield the best possible bounds), define $w_x := r_{i+1}^{x d/(\tau-1) }=\|v\|^{xd\gamma^{i+1}/(\tau-1)}$ and the weight interval $I_x := [w_x, 2w_x]$. We do this parametrization since we search near the typical maximum degree in the ball, which corresponds to taking $x=1$. We will search for cheap edges connecting the balls $B_{r_{i+1}}(z) $ and $B_{r_{i+1}}(z')$ using vertices with weight in the interval $I_x$. Note that any pair of vertices $b_z\in B_{r_{i+1}}(z), b_{z'}\in B_{r_{i+1}}(z')$ satisfies $\|b_z-b_{z'}\| = \Theta(r_i)$, since $\|z-z'\|$ is of order $r_i$, which is much larger than $r_{i+1}=r_i^\gamma$.
For a given constant  $s\in [0,1]$, we can compute the expected number of edges $\E[|E(z,z')|]$ between $B_{r_{i+1}}(z) \times I_{x}$ and $B_{r_{i+1}}(z') \times I_{x}$ of cost at most $r_i^s$ as follows, using that the probability that any node has weight in $I_x$ is $\Theta(w_x^{-(\tau-1)})$, and that the transmission cost of the edge is $\Theta(w_x^{2\mu} r_i^\zeta)$:
\begin{align}\label{eq:expected-cheap-edges}
    \E[|E(z,z')|]= \Theta\big(\big(r_{i+1}^{d} w_x^{-(\tau-1)}\big)^2\big) \cdot \Theta\big(\min\big\{w_x^2r_i^{-{d}}, 1\big\}^{\alpha}\big) \cdot \pr\big(\mathrm{Exp}(1) \le \Theta(w_x^{-2\mu} \cdot r_i^{-\zeta} \cdot r_i^s\big).
\end{align}
The first term corresponds to the (expected) number of pairs of vertices that can be found in the balls $B_{r_{i+1}}(z)$ with weight in the interval $I_x$, formally of the size of the set:
\[
(b_z,b_{z'})\in (B_{r_{i+1}}(z) \times I_{x}) \times (B_{r_{i+1}}(z') \times I_{x}),
\]
since these two boxes contain $\Theta(r_{i+1}^{d})$ vertices in expectation, each of them with probability $\Theta(w_x^{-(\tau-1)})$ to have weight in the interval $I_x = [w_x, 2w_x]$. The second term in \eqref{eq:expected-cheap-edges} is then the probability that such a vertex pair $(b_z,b_{z'})$ actually forms an edge, given by \eqref{eq:girg-connection}. Finally, if the edge $b_zb_{z'}$ exists, by \eqref{eq:cost} its cost is given by 
\[
Y_{b_zb_{z'}}\cdot(W_{b_z}W_{b_{z'}})^{\mu}\cdot(\|b_z-b_{z'}\|\vee 1)^{\zeta} = Y_{b_zb_{z'}}\cdot \Theta(w_x^{2\mu} \cdot r_i^{\zeta}),
\]
and since $Y_{b_zb_{z'}}$ is an $\mathrm{Exp}(1)$ variable, the probability that this is smaller than $r_i^{s}$ is given by the third term in \eqref{eq:expected-cheap-edges}. Note that $\pr(\mathrm{Exp}(1) \le a^\rho) = \Theta(a^{\rho \wedge 0})$ as $a\to\infty$. Using this and that $r_{i+1}=r_i^\gamma$ while  $w_x = r_i^{xd\gamma/(\tau-1)}$, the expression in \eqref{eq:expected-cheap-edges} becomes
\begin{align}
\begin{split}\label{eq:expected-cheap-edges-simplified}
    &\Theta\big(r_{i}^{2(1-x)d\gamma}\big) \cdot \Theta\big(r_i^{\alpha (\frac{2xd\gamma}{\tau-1}-d) \wedge 0}\big) \cdot \Theta\big(r_i^{((s-\zeta) - \frac{2x\mu d\gamma}{\tau-1}) \wedge 0}\big) \\
    &\qquad= \Theta\Big(r_i^{2(1-x)d\gamma + \alpha(\frac{2xd\gamma}{\tau-1}-d) \wedge 0 + (s-\zeta-\frac{2x\mu d\gamma}{\tau-1}) \wedge 0 }\Big)
    =: \Theta(\|v\|^{\Lambda(u, \gamma, x, d, \mu, \zeta, \tau, \alpha)\cdot\gamma^i}),
\end{split}
\end{align}
where we define the function in the exponent as
\[ \Lambda(s, \gamma, x, d, \mu, \zeta, \tau, \alpha) = 2(1-x)d\gamma + \alpha\Big(\frac{2xd\gamma}{\tau-1}-d\Big) \wedge 0 + \Big(s-\zeta-\frac{2x\mu d\gamma}{\tau-1}\Big) \wedge 0. \]
We mention here that it is at this step of the proof that we assume that $\|v\|$ is large. Namely, we will set this exponent to be just slightly bigger than $0$ (say, $\epsilon$), but as $\|v\|$ is large, this will still guarantee us in expectation a large number of cheap connecting edges. 
\subsubsection*{Total failure probability and cost}
We choose the parameters $s,\gamma,x$ so that $\Lambda = \Lambda(s, \gamma, x, d, \mu, \zeta, \tau, \alpha)>0$. Later we will see how to make this choice in order to optimise the cost of the constructed path. When $\Lambda>0$, the number of edges between $B_{r_{i+1}}(z) \times I_{x}$ and $B_{r_{i+1}}(z') \times I_{x}$ of cost at most $r_i^{s}$ is concentrated around its expectation, and in particular, since this expectation is given by \eqref{eq:expected-cheap-edges-simplified}, the probability that no such edge exists is at most $\exp(-\Theta(r_i^{\Lambda}))$, which follows by standard Chernoff bounds for indicator variables. Remember that after $i$ iterations there are $2^i$ pairs of vertices $(z,z')$ for which we are searching for such an connecting edge. Since $\gamma<1$ and so $r_i= \|v\|^{\gamma^i}$ gets doubly-exponentially smaller as $i$ increases, the last iteration contains the main probabilistic error: the failure probability of $k$ iterations of this construction can be upper bounded by
\begin{align}\label{eq:failure-prob}
    p_{\mathrm{fail}}(\Lambda, \gamma, k) \le \sum_{i=0}^{k-1} 2^i \exp(-\Theta(\|v\|^{\Lambda\gamma^i})) \le 2^k \exp(-\Theta(\|v\|^{\Lambda\gamma^{k-1}}).
\end{align}
We shall choose $k$ so that this error probability is still small yet the distances left to connect the gaps after the $k$ iterations are relatively small. 
If the first $k$ iterations of the construction are successful, call $\calP_k$ the set of edges selected this way (note that $|\calP_k| = \sum_{i=0}^{k-1} 2^i = 2^k - 1$). At this stage, there remains to connect $2^k$ pairs of low-weight vertices, each pair separated by distance $O(r_k)=O(\|v\|^{\gamma^k})$. This can be done using paths following the Euclidean geometry on such short distances if order $\|v\|^{\gamma^k}$, which yields a path $\pi$ between $0$ and $v$ of total cost
\begin{align}\label{eq:total-cost}
    \cost{\pi} = \cost{\calP_k} + O(2^k\|v\|^{\gamma^k}).
\end{align}

This iterative strategy for building a cheap path creates complex dependencies: the geometric locations where we are searching for cheap edges in one iteration depend on the edges found in the previous iteration; and there are weight dependencies as well. It is possible to overcome those dependencies by following a probabilistic procedure: first, we expose all the weights and locations of vertices in advance, and show that these are well scattered, before we start to build the path. Well-scatteredness follows from concentration estimates of Poisson variables. Then, one can use a sequential revealment of the edges of the graph, called multi-round exposure. While the technicalities are tedious, the concept is well-established, and we work it out in  \cite{komjathy2023four}, to which we refer the reader for further details.

Here, we focus on the optimisation problem and investigate the total cost of the construction further. We now distinguish between two cases, depending on whether there exists choices of $\gamma$ and $x$ such that $\Lambda(0, \gamma, x, d, \mu, \zeta, \tau, \alpha)>0$ for $s=0$ or whether there exists a choice $0<s<1$ and at the same time $\Lambda(0, \gamma, x, d, \mu, \zeta, \tau, \alpha)>0$. We shall see that if none of these two options is possible, we are in the pure geometric regime.

If our choice of $s,\gamma,x$ that satisfy $\Lambda(u, \gamma, x, d, \mu, \zeta, \tau, \alpha)>0$ is such that $0<s<1$, this means that the (upper bound on the) cost of the very first edge is $\|v\|^s$ (in other words, $\cost{\calP_1} = O(\|v\|^s)$). More generally, for any constant $k$, for all sufficiently large $\|v\|$, 
\[
    \cost{\calP_k} = \sum_{i=0}^{k-1} 2^i \cdot O(\|v\|^{s\gamma^i}) = O(\|v\|^s).
\]
In this case the total cost is dominated by the cost of the very first (and longest) edge, but it is still much shorter than the Euclidean distance $\|v\|$. In this case, we choose $k$ so that $\gamma^k < s$. Since $s\in(0,1)$ and $\gamma<1$ are constants, this condition is satisfied for (a large enough) constant $k$. This results in $\|v\|^{\gamma^k}< \|v\|^s$. Therefore, by \eqref{eq:total-cost} we have constructed a path $\pi$ between $0$ and $v$ of cost
\begin{align}\label{eq:polynomial-cost-bound} 
    \cost{\pi} \le O(\|v\|^s) + O(2^k\|v\|^{\gamma^k}) = O(\|v\|^s).
\end{align}
By \eqref{eq:failure-prob}, the failure probability of this construction satisfies
\begin{equation}\label{eq:pfail}
    p_{\mathrm{fail}}(\Lambda, \gamma, k) \le 2^k \exp(-\Theta(\|v\|^{\Lambda\gamma^{k-1}}),
\end{equation}
which goes to $0$ as $\|v\|\to\infty$ since $\Lambda, \gamma>0$ and $k$ is constant (i.e.\ does not grow with $\|v\|$).

Suppose on the other hand that there exists some $\gamma,x$ such that for $s=0$ it holds that  $\Lambda(0, \gamma, x, d, \mu, \zeta, \tau, \alpha)>0$. Then, the cost of each edge entering the iterative construction is bounded by $O(\|v\|^s) = O(1)$, which means that $\cost{\calP_k} = O(2^k)$. Hence, our goal is to iterate until the remaining distances to bridge are also of the same order. i.e.\ we choose $k=k(\|v\|, d, \mu, \zeta, \tau, \alpha)$ such that $\|v\|^{\gamma^k} = \Theta(1)$. This is satisfied for $k=\frac{\log\log\|v\|}{\log (1/\gamma)}$. By \eqref{eq:total-cost}, this yields a path of total cost
\begin{align}\label{eq:polylog-cost-bound}
    \cost{\pi} \le O(2^k) + O(2^k) = O(2^k) = O\Big((\log\|v\|)^{\log 2 / \log (1/\gamma)}\Big).
\end{align}
One may compute that for this choice of $k$, $\|v\|^{\gamma^k}=\Theta(1)$, so unfortunately the failing probability of the construction in \eqref{eq:pfail} is too large. So, we choose $k=k(\|v\|, d, \mu, \zeta, \tau, \alpha)$ a little bit smaller: set $k=\frac{\log\log\|v\|}{\log (1/\gamma)}-\Theta(\log\log\log \|v\|)$, and with this choice it holds that 
\[
(\log\log \|v\|)^{\omega(1)} \le \|v\|^{\gamma^k} \le (\log\|v\|)^{o(1)},
\]
which allows us to keep the upper bound on $\cost{\pi}$ in \eqref{eq:polylog-cost-bound} (almost) unchanged, while making the failure probability in \eqref{eq:failure-prob} satisfy $p_{\mathrm{fail}}(\Lambda, \gamma, k) = o(1)$. The cost estimate in \eqref{eq:polylog-cost-bound} only changes by an additive $o(1)$ term in the exponent, as in this case $2^k \|v\|^{\gamma^k} = O( (\log \|v\|)^{o(1)+\log 2/\log(1/\gamma)})$.

This explains, in general, the main stategy of finding a path between $0,v$ with either polynomial or a polylogarithmic cost as a function of $\|v\|$.

\subsubsection*{Optimisation of the exponents: polylogarithmic cost}
We now turn to optimising the cost of the path obtained via this iterative construction. We repeat the value of the exponent $\Lambda$ in \eqref{eq:expected-cheap-edges-simplified} defined as 
\begin{align}\label{eq:Lambda-def}
    \Lambda(s,\gamma,x,d,\mu,\zeta,\tau,\alpha) := 2(1-x)d\gamma + \alpha d\Big(\tfrac{2x\gamma}{\tau-1}-1\Big) \wedge 0 + \Big(s-\zeta-\tfrac{2x\mu d\gamma}{\tau-1}\Big) \wedge 0.
\end{align}
Given the graph parameters $\tau>3$ and $\alpha>1$ and the transmission parameters $\mu, \zeta>0$, we would like to find the minimal $s$  so that for some $0\le x\le 1$ and $0<\gamma<1$, it holds that $\Lambda>0$. The constraint $\gamma<1$ is necessary in our construction for the balls to shrink, while $x\le 1$ encodes that we use a weight range that is below the typical meximum in the ball under consideration, so we actually find vertices with those weights. 

For the constructed path to have polylogarithmic cost, we need to find values of $\gamma\in(0,1), x\in[0,1]$ such that setting $s=0$ yields $\Lambda(0,\gamma,x,d,\mu,\zeta,\tau,\alpha)>0$.  
Setting $s=0$ in \eqref{eq:Lambda-def} gives
\begin{align}\label{eq:Lambda-polylog}
    \Lambda(0,\gamma,x,d,\mu,\zeta,\tau,\alpha) = 2(1-x)d\gamma + \Big(\alpha d\Big(\tfrac{2x\gamma}{\tau-1}-1\Big) \wedge 0 \Big)-\zeta-\tfrac{2x\mu d\gamma}{\tau-1}.
\end{align}
Since in \eqref{eq:polylog-cost-bound} the $x$ parameter does not enter the cost-bound (unlike $\gamma$), we can start by maximising $\Lambda$ with respect to $x$. Note that $\Lambda(0,\gamma,x,d,\mu,\zeta,\tau,\alpha)$ is piecewise linear in $x$, so the extremal points can be attained either at the boundary values for $x$ (namely $x=0$ and $x=1$), or at the point where $\Lambda=\Lambda(x)$ switches from one linear function of $x$ to another, namely at $x=\tfrac{\tau-1}{2\gamma}$, assuming this value is in the interval $[0,1]$. It is important to keep in mind that the bound in \eqref{eq:polylog-cost-bound} is minimised when the exponent $\tfrac{\log 2}{\log (1/\gamma)} =\tfrac{1}{\log_2 (1/\gamma)}$ is smallest. In other words, in order to optimise our cost-bound (for a fixed choice of $x\in\{0, \tfrac{\tau-1}{2\gamma}, 1\}$) we need to find the smallest (infimum) $\gamma\in(0,1)$ that satisfies $\Lambda(0,\gamma,x,d,\mu,\zeta,\tau,\alpha)>0$. We now analyse the three cases $x=0$, $x=1$, and $x=\tfrac{\tau-1}{2\gamma}$ separately.

Setting $x=0$ in \eqref{eq:Lambda-polylog} yields
\[
    \Lambda(0,\gamma,0,d,\mu,\zeta,\tau,\alpha) = 2d\gamma - \alpha d - \zeta.
\]
Since $\gamma\in(0,1)$, the above expression can only be positive if
\begin{align}\label{eq:weak-polylog-condition}
    \zeta/d < 2-\alpha.
\end{align}
Assuming this inequality holds, the smallest $\gamma$ that makes $\Lambda$ non-negative in this case is $\gamma:= \tfrac{\alpha+\zeta/d}{2}\in(0,1)$. For this choice we obtain $\Lambda=0$, but we can obtain a positive $\Lambda$ for any $\gamma > \tfrac{\alpha+\zeta/d}{2}$, and we may take the infimum over all these since this only adds an  additive $o(1)$ in the exponent of the cost. For ease of presentation, we give the calculation directly for $\gamma= \tfrac{\alpha+\zeta/d}{2}\in(0,1)$, which yields an exponent of 
\begin{equation}\label{eq:Delta-1-def}
\Delta_1:=\frac{1}{\log_2(2/(\alpha+\zeta/d))} = \frac{1}{1-\log_2(\alpha+\zeta/d)}
\end{equation}\
in the cost-bound \eqref{eq:polylog-cost-bound}, corresponding to Phase $(ii.a)$ in Theorem \ref{thm:main-supporting}. Observe that in this case the long links that we use to build the low-cost path uses edges that occur on constant-weight (hence, low-degree) vertices. These edges are weak ties and only `accidentally present' as the effect of $\alpha$ drives the long-range connections. Thus in this phase the pandemic uses weak-tie links to propagate far.  

If we choose $x=\tfrac{\tau-1}{2\gamma}$, then \eqref{eq:Lambda-polylog} becomes
\[
    \Lambda(0,\gamma,\tfrac{\tau-1}{2\gamma},d,\mu,\zeta,\tau,\alpha) = 2d\gamma - (\tau-1)d - \zeta - \mu d.
\]
The smallest $\gamma$ that achieves $\Lambda\geq 0$ is given by $\gamma:=(\tau-1+\mu+\zeta/d)/2$. Note in particular that it satisfies $\gamma \ge \tfrac{\tau-1}{2}$, which is a necessary condition to have $x\in[0,1]$ in this case. For a valid solution we also need that $\gamma$ is in the interval $(0,1)$, which now results in the condition
\begin{align}\label{eq:strong-polylog-condition}
   \mu+\zeta/d < 3-\tau.
\end{align}
Summarising, when $\mu+\zeta/d<3-\tau$, then $\gamma=(\tau-1+\mu+\zeta/d)/2$ (or strictly speaking, any $\gamma$ than this value) and $x=(\tau-1)/(2\gamma)=(\tau-1)/(\tau-1+\mu+\zeta/d)$ is a valid choice that results in $\Lambda >0$.
 This then yields an exponent in the cost-bound in \eqref{eq:polylog-cost-bound} that equals 
 \begin{equation}
 \Delta_2:=\frac{1}{\log_2(2/(\tau-1+\mu+\zeta/d))} = \frac{1}{1-\log_2(\tau-1+\mu+\zeta/d)}
 \end{equation}
 corresponding to phase $(ii.b)$ in Theorem \ref{thm:main-supporting}. We observe that the value $x=(\tau-1)/(2\gamma)$ means that we use vertices with weight roughly $r_{i+1}^{d/2}$, in a ball of radius $r_{i+1}$. These vertices are the hubs of the ball: one can show that they all connect to each other, and thus we use hub-driven shortest path construction. Thus in this phase the pandemic uses hubs to propagate far.  

Finally, setting $x=1$ in \eqref{eq:Lambda-polylog} yields
\[
    \Lambda(0,\gamma,1,d,\mu,\zeta,\tau,\alpha) = \alpha d(\tfrac{2\gamma}{\tau-1}-1) \wedge 0 -\zeta-\tfrac{2\mu d\gamma}{\tau-1}.
\]
Note that in this case we always have $\Lambda \le 0$, i.e.\ the choice of $x=1$ does not yield any solution for polylogarithmic cost.

\subsubsection*{Optimisation of the exponents: polynomial cost}
We now turn to the polynomial case, namely when $\Lambda>0$ cannot be satisfied for $s=0$, but it may be satisfied for some $s\in(0,1)$ still resulting in sublinear cost. Our goal is to minimize $s$ so that there is a choice of $x\in[0,1]$ and $\gamma\in(0,1)$ so that $\Lambda>0$. As before, we will solve this optimization problem with the condition $\Lambda \ge 0$ to avoid working with infima. 
In this case the cost of the path is dominated by the cost of the first (and longest) edge in the construction, see \eqref{eq:polynomial-cost-bound}. In particular this cost is a decreasing function of $\gamma$ (since increasing the size of the two boxes around $0$ and $v$ increases the number of "candidate" edges). Analytically this means that the first term $2\gamma (1-x)d$, corresponding to the number of combinatorial options to choose the longest edge from, is increasing in $\gamma$. Therefore, we should set $\gamma$ to its maximal value (namely $1$, and later $1-o(1)$ to ensure $\gamma<1$). As $x$ sweeps through $[0,1]$ we may compensate for this choice in the other terms in $\Lambda$ by decreasing $x$ if necessary. Setting $\gamma=1$ yields
\begin{align}\label{eq:Lambda-polynomial}
    \Lambda(s,1,x,d,\mu,\zeta,\tau,\alpha) = 2(1-x)d + \alpha d(\tfrac{2x}{\tau-1}-1) \wedge 0 + (s-\zeta-\tfrac{2x\mu d}{\tau-1}) \wedge 0.
\end{align}

From there, our goal is to find the smallest $s\in(0,1]$ such that there exists $x\in[0,1]$ satisfying $\Lambda(s,1,x,d,\mu,\zeta,\tau,\alpha) \ge 0$. 
Since the term that contains $s$ is inside the minimum $(s-\zeta-\tfrac{2x\mu d}{\tau-1}) \wedge 0$, choosing $s > \zeta +\tfrac{2x\mu d}{\tau-1}$ never helps. 

Case 1. We start by the case where the optimum is $s = \zeta +\tfrac{2x\mu d}{\tau-1}$, in which case \eqref{eq:Lambda-polynomial} becomes
\begin{align*}
    \Lambda = 2(1-x)d + \alpha d(\tfrac{2x}{\tau-1}-1) \wedge 0.
\end{align*}
Since $s = \zeta +\tfrac{2x\mu d}{\tau-1}$ is an increasing function of $x$, our goal is then to find the smallest $x\in[0,1]$ satisfying $\Lambda\ge 0$. Choosing $x=\tfrac{\tau-1}{2}$ yields a positive $\Lambda = (3-\tau)d$, so we can focus on the range $x\in[0,\tfrac{\tau-1}{2}]$, where we have
\begin{align}\label{eq:lambda-temp-1}
    \Lambda = 2(1-x)d + \alpha d(\tfrac{2x}{\tau-1}-1) = (2-\alpha)d + 2xd\tfrac{\alpha-(\tau-1)}{\tau-1}.
\end{align}
 We start with the subcase $\alpha\ge 2$ (which implies $\alpha > \tau-1$ in particular). Then the smallest $x$ satisfying $\Lambda \ge 0$ is given by $x=\tfrac{(\alpha-2)(\tau-1)}{2(\alpha-(\tau-1))}$, which yields $s = \zeta +\tfrac{2x\mu d}{\tau-1} = \zeta + \tfrac{(\alpha-2)\mu d}{\alpha-(\tau-1)}$. In order to have $s<1$, we need the condition
\begin{align}\label{eq:hybrid-polynomial-condition}
    \mu < (1-\zeta)\tfrac{\alpha-(\tau-1)}{(\alpha-2)d} = (1-\zeta) (\tfrac{1}{d}+ \tfrac{3-\tau}{(\alpha-2)d}),
\end{align}
and this results in the cost-exponent $s$ to take value
\begin{equation}
\eta_2:=\zeta + \tfrac{(\alpha-2)\mu d}{\alpha-(\tau-1)}.
\end{equation}
This corresponds to phase $(iii.e)$ in Theorem \ref{thm:main-supporting}. Observe that here the long edges that we use to build the path occur on degrees that are neither constant nor near the hubs having weight $r_{i+1}^{d/2}$ in the ball (corresponding to $x=(\tau-1)/(2\gamma)$), but somewhere in between. Further, the links are present only because the effect of the long-range parameter $\alpha$, so these are weak ties. We call this the hybrid regime: in this phase the pandemic uses weak-tie links between medium sized hubs to propagate far.

The other subcase is if $\alpha<2$. In this case in \eqref{eq:lambda-temp-1} the optimal choice is to set $x=0$, which does yield $\Lambda>0$.
Recall that we are in the case where $s=\zeta+\tfrac{2x\mu d}{\tau-1}$, so now  when $\alpha<2$ with $x=0$ we achieve $s = \zeta$. We will see that this is not an optimal solution, i.e., we can get a smaller exponent $s$ when $\alpha<2$.

Case 2. We now turn to the case where  choose $s < \zeta +\tfrac{2x\mu d}{\tau-1}$. Then $\eqref{eq:Lambda-polynomial}$ becomes
\begin{align*}
    \Lambda(s,1,x,d,\mu,\zeta,\tau,\alpha) = 2(1-x)d + \alpha d(\tfrac{2x}{\tau-1}-1) \wedge 0 + s-\zeta-\tfrac{2x\mu d}{\tau-1}.
\end{align*}
While we want to solve this for $\Lambda\ge 0$, we warn that a solution is only valid if $s < \zeta +\tfrac{2x\mu d}{\tau-1}$, resulting in that the sum of the first two terms,
\begin{equation}
\Lambda_{12}(x,d,\mu,\zeta,\tau,\alpha) = 2(1-x)d + \alpha d(\tfrac{2x}{\tau-1}-1) \wedge 0 > 0
\end{equation}
holds at the same time. This is since if $\Lambda_{12}\le 0$ then there are no edges present on which we try to optimise the cost. 
The smallest $s$ satisfying $\Lambda\ge 0$ is then given by 
\begin{align}\label{eq:u_min}
    s = \zeta + \tfrac{2x\mu d}{\tau-1} - 2(1-x)d - \alpha d(\tfrac{2x}{\tau-1}-1)\wedge 0.
\end{align}
This is a piecewise linear function of $x\in[0,1]$, so by the same reasoning as before, its minimum will be attained by one of the values in $x\in\{0, \tfrac{\tau-1}{2}, 1\}$. We analyse these three subcases separately.

Case 2.1. Setting $x=0$ in \eqref{eq:u_min} yields $s = \zeta-(2-\alpha)d$, which is smaller than $1$ if
\begin{align}\label{eq:weak-polynomial-condition}
    \zeta/d < 2 - \alpha + 1/d.
\end{align}
The condition that $\Lambda_{12}>0$ results in the condition that $\alpha<2$, which indeed guarantees that edges are actually present between $B_{r_{i+1}}(z) \times I_{0}$ and $B_{r_{i+1}}(z') \times I_{0}$. 
Summarizing, we obtain a solution $s$ that becomes:
\begin{equation}
\eta_1:= \zeta-(2-\alpha)d \text{ if } \alpha<2 \text{ and } \zeta<  2 - \alpha + 1/d.
\end{equation}
This corresponds to phase $(iii.d)$ in Theorem \ref{thm:main-supporting}. Note in particular that this is always a better solution than $s=\zeta$ that was obtained before in Case 1 under the condition $\alpha<2$. Again, we remark that in this phase we use $x=0$, corresponding to constant-weight  and hence constant degree vertices and weak-tie links. Thus in this phase the pandemic uses weak-tie links to propagate far.  

Case 2.2. 
If we choose $x=\tfrac{\tau-1}{2}$, which satisfies $x\in[0,1]$ only if $\tau<3$, \eqref{eq:u_min} becomes $s=\zeta +(\tau+\mu-3) d$. This is smaller than $1$ if
\begin{align}\label{eq:strong-polynomial-condition}
    \mu+\zeta/d < 3-\tau+1/d.
\end{align}
and results in the spreading exponent 
\begin{equation}
\eta_3=\zeta +\mu d+(\tau-3) d. 
\end{equation}
This corresponds to phase $(iii.f)$ in Theorem \ref{thm:main-supporting}. We observe that the value $x=(\tau-1)/(2\gamma)$ again means that we use vertices with weight roughly $r_{i+1}^{d/2}$, in a ball of radius $r_{i+1}$. These vertices are the hubs of the ball: one can show that they all connect to each other, and thus we use hub-driven shortest path construction. Thus in this phase the pandemic uses hubs to propagate far.

Finally, setting $x=1$ in \eqref{eq:u_min} yields $s = \zeta + \tfrac{2\mu d}{\tau-1}$ since $\tfrac{2}{\tau-1} > 1$. Notice that
\begin{align*}
   \tfrac{2\mu d}{\tau-1} > \mu d > (\tau+\mu-3)d, 
\end{align*}
so this is never an improvement compared to $u=\zeta +(\tau+\mu-3) d$ that was obtained right before.

We summarise the proof. If we can find parameters $x, \gamma\in(0,1) $ so that the inequality $\Lambda(s,\gamma,x,d,\mu,\zeta,\tau,\alpha)>0$ holds for $s=0$, then we obtain polylogarithmic distances, which in turn result in quasi-exponential growth of $I(t)$.  If we can find parameters $x, \gamma\in(0,1) $ so that the inequality $\Lambda(s,\gamma,x,d,\mu,\zeta,\tau,\alpha)>0$ holds for $s\in(0,1)$, then we obtain polynomial distances, which in turn result in polynomial growth of $I(t)$ faster than $t^{d}$. The infimum value of $s$ in $[0,1]$ so that a choice is possible gives the optimal spreading exponent. If the parameters $\alpha, \tau, \mu, \zeta$ are such that none of these are possible, then we use a new strategy to connect $0$ to $x$ at linear cost in $\|x\|$ that we describe in the next section, using short edges.
\subsection{Proof of upper bound of pure geometric growth}
When $\tau<3$, one can always construct examples with linear distances in $d\ge 2$ and almost-linear distances in $d\ge 1$. We sketch the arguments. For $d\ge 2$ we refer to \cite[Section 3]{komjathy2024polynomial} for details; the $d=1$ case is treated in \cite{komjathy2023four}.

\medskip
\noindent\textbf{Case dimension $d\ge 2$.}
Fix a large constant $M$ and consider the graph $G^M$ induced by vertices with weights in $[M,2M]$, keeping only edges whose costs satisfy $\calC_e\le M^{1+2(\mu+\zeta/d)}$. In a GIRG, the connection probability in gives that  \eqref{eq:girg-connection}
 whenever $W_u,W_v\in[M,2M]$ and $\|u-v\|\le M^{2/d}$, an edge between $u$ and $v$ is present with probability at least a constant $c$, by \eqref{eq:girg-connection}. Moreover, such an edge has transmission cost
$
\calC_{uv}=(W_uW_v)^\mu \|u-v\|^\zeta E_{uv}
  \le 4^\mu M^{2\mu} M^{2\zeta/d} E_{uv}.
$ where $E_{u,v}$ is an exponential variable with mean $1$.
Set $C_M := M\cdot 4^\mu M^{2\mu} M^{2\zeta/d}$. Then
\[
\mathbb P\!\bigl(\calC_{uv}\le C_M\bigr)
=\mathbb P\!\bigl(4^\mu M^{2\mu} M^{2\zeta/d} E_{uv}\le C_M\bigr)
=\mathbb P(E_{uv}\le M),
\]
which is at least $4/5$ for $M$ sufficiently large. Thus, in $G^M$ we retain only edges between $[M,2M]$-weight vertices whose cost is at most $\calC_e\le M^{1+2(\mu+\zeta/d)}$, and then ignore the costs to view a simple graph.

By \cite[Corollary 3.9]{komjathy2024polynomial}, for $M$ large enough (relative to the connectivity lower bound $4c/5$), this graph on $\R^d$ has an infinite component; denote it by $\mathrm{C}_\infty^M$. The same corollary yields that for any $u^\star,v^\star\in \mathrm{C}_\infty^M$, the graph distance (number of edges) along a shortest path is $\Theta(\|u^\star-v^\star\|)$. Since each edge in $G^M$ has cost at most $C_M$, cost-distances within $\mathrm{C}_\infty^M$ are likewise linear.

Finally, for arbitrary vertices $u,v$ we connect them to nearby $u^\star,v^\star\in \mathrm{C}_\infty^M$ using that $\mathrm{C}_\infty^M$ comes near any location with probability close to $1$ \cite{deuschel1996surface}, implying $d_{\mathcal C}(u,u^\star)=\Theta(\|u-u^\star\|)$; the same holds for $v,v^\star$. We ensure these linear cost-bounds hold \emph{simultaneously} for all candidate anchors $u^\star$ near $u$. In \cite{komjathy2024polynomial} this is implemented by a renormalisation that maps $G^M$ to a site–bond percolation on $\Z^d$, allowing us to pull back density and distance estimates from \cite{antal1996chemical,deuschel1996surface}.

\medskip
\noindent\textbf{Case dimension $d=1$.}
In one dimension, $G^M$ has no infinite component, so the above renormalisation route to \cite{antal1996chemical,deuschel1996surface} is unavailable. Instead, \cite{komjathy2023four} takes a finite-size approach: we consider the same $G^M$ but choose $M=M_{\|u-v\|}$ depending on $\|u-v\|$ to guarantee a large connected subgraph in the segment between $u$ and $v$. We then establish the requisite density and distance bounds directly, using paths along which the vertex weights increase, followed by a renormalisation argument at this scale. For finite graphs (e.g.\ the GIRG $G_n$ in Def.~\ref{def:GIRG}), we additionally exploit that (near)-shortest paths in $G^M$ deviate only slightly from the straight line, so “unwrapping’’ the torus to a box yields a sufficiently accurate approximation.

\subsection{Proof idea of the explosive phase}
We make use of the presence of `superspreaders' or hubs and the hierarchical structure of geometric graphs with high degree heterogeneity ($\tau<3$) to find a transmission path $\pi(u_0,v)$ between any nodes $u_0$ and $v$  with cost (transmission time)  $\mathcal C(\pi_{u_0, v})$ that does not depend on the spatial distance between $u_0, v$ as $\|v-u_0\|\to \infty$. While this path may not provide the actual shortest path, it does give a possible transmission path, and thus serves as an upper bound on the transmission cost of the actual shortest path. 

The construction of the path is as follows: 
From both $u_0$ and $v$ it takes a random time to `reach' a node of some large but constant degree $K$, which we will denote by $T_{u_0, K}, T_{v,K}$. Then, we show that for some positive $\delta>0$ and $a>1$, this node can transmit the information with at most  $< 1/K^{\delta}$ tansmission cost to some node with degree $K^{a}$, who then in turn can transmit with at most $1/K^{a\delta}$ cost to a node of degree $(K^{a})^a=K^{a^2}$, and so on. We iterate this construction until we reach the highest degree nodes in the network, then show that any two such nodes pass the information to each other within negligible time. Starting from both $u$ and $v$ this constructs a path containing $\sim \log \log \|v-u_0\|$ many links whose total transmission time is upper bounded by $T_{u_0, K}+T_{v,K}+ \sum_{i=1}^\infty 2/K^{a^i\delta}\le T_{u_0, K}+T_{v,K} + \varepsilon$ for some small $\varepsilon$ as the sum is finite and vanishes as $K$ gets large. Moreover, the random variables $T_{u_0, K}, T_{v,K}$ are locally determined and they converge to almost surely finite random variables $Y_u, Y_v$. We can then use the following argument to show the epidemic curve on a finite graph from a single source $u_0$:
\[ 
\begin{aligned}
\frac{I(t)}{n}&=\frac1n \{\# v: T_{u_0, v}\le t\} \ge \frac{1}{n}\{\# v: T_{u_0, K} +T_{v, K}+\varepsilon \le t\} \\
&=\frac{1}{n}\{\# v:T_{v, K} \le t- T_{u_0, K}\}= \frac{1}{n}\sum \mathbf 1\{T_{v, K} \le t- \varepsilon - T_{u_0, K}\} 
\end{aligned}
\]
The latter sum of indicators are weakly dependent and their average can be shown to converge to $\mathbb P(T_{v, K} \le t- \varepsilon - T_{u_0, K} \mid T_{u_0, K}):=q_{t-\varepsilon -S}$, where the random shift $S$ in Eq. [7] of the main text represents the random time $T_{u_0, K}$ it takes to reach a sufficiently high-degree node from the source node. This random time appears as part of (almost) all infection paths to different nodes $v$ across the network. 
A similar lower bound can also be given: one can show that it is impossible to avoid high degree nodes and still have fast transmission paths. So, we obtain the lower bound
$T_{u,v}\ge T_{u_0, K} + T_{v, K}$ on the transmission cost, and this leads to the upper bound $q_{t-S}$ on $I(t)$. Finally one lets $K\to \infty$ wich makes the $\varepsilon$ vanish. The full proof can be found in \cite{reubsaet2022topology}, see also \cite{komjathy2020explosion}.

\subsection{Lower bounds via multiscale analysis}
In this section we provide the proof for the lower bounds in Theorem \ref{thm:main-lower}. We will set up a multiscale analysis, in the same fashion as in \cite{komjathy2024polynomial}. The multiscale analysis presented in \cite{komjathy2024polynomial} carries through for the model we consider here once we make sure that we can pass between scales. In the multiscale analysis, on scale $k$ we consider boxes of side-length $A_k$ where the side-length $A_k=A_1 (k!)^2$ grows rapidly with the scales, and partition the box with scale-$(k-1)$ boxes. We call a box with side-length $A_k$ $s$-good if all but $3^d$ of its subboxes are good (except for some initial scale $k_0$ where we omit this condition) and if it does not contain any edges of length \emph{longer} than $L_k:=A_k/(100k^2)$ and transmission cost at most $L_k^s$, and the same is true for slightly shifted versions of this box. As $A_k$ grows rapidly with the scales, $L_k^{1+\varepsilon}= A_k^{1+\varepsilon}/(100k^2)^{(1+\varepsilon)} \gg A_k$, so we will be apply  the following lemma to bound the probability that a box is good, given that enough of its subboxes are good. 

The goal of this lemma is to estimate the number of edges inside a box that have length $L$ roughly the same as the side-length of the box, and transmission cost a polynomial of the length $\approx L^s$. Because we look at the power $s$ of the cost $\approx L^s$ vs the length of the edge $\approx L$, we will restrict the edge to be of length in the interval $[L, L^{1+\varepsilon}]$. In other words, we count much longer edges with the same cost at the appropriate larger scale.
\begin{lemma}\label{lem:no_long_cheap_edge}
Consider $1$-SI of Definition \ref{def:1-FPP} with parameters $\mu, \zeta\ge 0$ on a GIRG $G=(V,E)$ from Definition \ref{def:GIRG} on $n$ vertices with parameters $\tau>2, \alpha>1, d\ge 1$.
    For all sufficiently small $\eps > 0$, if $L>0$ is sufficiently large relative to $\eps$, then the following holds. Let $n/2>A>L$ and $s > 0$. Let $E(A, L, s)$ be the number of edges inside the box $[-A/2,A/2]^d$ with length in the interval $[L, L^{1+\varepsilon}]$ and cost at most $L^{s}$. 
    
    Then for all $\tilde \varepsilon >0$ small as $A, L\gg 1$ the following bounds hold for some explicit nonnegative constants $c_1, c_2$ and $c_3:=\max(\zeta(\tau-1-\alpha )/\mu - d(\alpha-1), 0)$:
    
    \noindent If $s\le \zeta$, then
\begin{equation}\label{eq:no_long_cheap_edge-1}
\mathbb E[E(A,L, s)]\le 	 
	c_2 A^d \big( L^{s - \zeta -d(\alpha-1) +\tilde \varepsilon} + L^{s - \zeta- d\mu- d(\tau-2)+\tilde \varepsilon}\Big).
    \end{equation}
    If $\zeta<s\le \zeta+d\mu$, we have three cases based on the value $\alpha, \tau, \mu$:
    \begin{align}\label{eq:no_long_cheap_edge-2}
\mathbb E[E(A, L, s)]\le 	 \begin{cases}
 c_2 A^d \Big( L^{-d(\alpha-1)+\tilde \varepsilon}  + L^{s-\zeta-d(\tau-2+\mu-\tilde \varepsilon)} +L^{-d(\alpha-1)+(s-\zeta)(\alpha-\tau+1)/\mu+\tilde \varepsilon + c_3 \varepsilon} \Big),
 & \mbox{ if $\alpha\le \tau-1$,}\\
c_2 A^d \Big( L^{s-\zeta-d(\tau-2+\mu)+\tilde \varepsilon} 
+ L^{-d(\alpha-1)+(s-\zeta)(\alpha-\tau+1)/\mu+\tilde \varepsilon + c_3 \varepsilon}\Big)
   &\mbox{ if $\tau-1<\alpha\le \tau-1+\mu$,}\\
   c_2 A^d \Big( L^{s-\zeta-d(\tau-2+\mu)+ \tilde \varepsilon}  + L^{-d(\alpha-1)+(s-\zeta)(\alpha-\tau+1)/\mu+\tilde \varepsilon} \Big) 
   &\mbox{if $\alpha>\tau-1+\mu$.}
	\end{cases}
\end{align}
Finally, if $s>\zeta+d\mu$, then for some constant $\hat c_2\ge 0$, the following bounds hold:
\begin{equation}\label{eq:eal-lower-bound}
\hat c_2 A^d\Big( L^{d(1-\alpha)}
+ L^{-d(\tau-2)} \Big) \le \mathbb E[E[A,L,s]]\le  c_2 A^d\Big( L^{d(1-\alpha)+ \tilde \varepsilon}
+ L^{-d(\tau-2) + \tilde \varepsilon} \Big).
\end{equation}
\end{lemma}
We postpone the proof of the lemma for later, but we give an intuition where the formulas come from, ignoring terms of order $O(\tilde \varepsilon)$ in the exponents. We start with \eqref{eq:no_long_cheap_edge-1}. When $s<\zeta$, then $L^{s}\ll L^{\zeta}$ so these edges carry a much shorter transmission time than typical. The first term in \eqref{eq:no_long_cheap_edge-1} corresponds to edges of length of order $L$ connecting two low-degree vertices. There are $A^dL^{d}$ low-weight pairs of vertices at distance of order $L$, each  of which forms an edge with probability $L^{-d\alpha}$. The expected transmission time on these edges is $L^\zeta$, so the probability of their cost being below $L^{s}$ is of order $L^{s-\zeta}$. Multiplying these together gives $A^dL^{d-d\alpha+s-\zeta}$, the first term in \eqref{eq:no_long_cheap_edge-1}. The second term in the first row corresponds to again to edges of length of order $L$, but now connecting two vertices whose degrees (vertex weights) multiply to order of $L^{d}$. There are roughly $A^d L^{-d(\tau-2)}$ such edges, and the probability that the cost on each of these edges is at most $L^{s}$ is of order  $L^{s-\zeta-d\mu}$.

The first term in the first row in \eqref{eq:no_long_cheap_edge-2} again counts  edges of length of order $L$ connecting two low-degree vertices, which yields roughly $A^d L^{-d(\alpha-1)}$. But now the expected transmission time on these edges, $L^\zeta$, is much shorter than $L^s$, so we count essentially all such edges. The second term here again counts  edges of length of order $L$, connecting two vertices whose degrees (vertex weights) multiply to roughly $L^{d}$, and with atypically low cost $L^{s-\zeta-d\mu}$.

We now move to the third row in \eqref{eq:no_long_cheap_edge-2}.  The new term     gives the exponent $\eta_3$, and it corresponds to edges of length roughly $L$ connecting two vertices whose weights multiply to roughly $L^{(s-\zeta)/\mu}$. The expected transmission time on any such edge is roughly $ L^{\mu (s-\zeta)/\mu }L^\zeta\asymp L^s$ so typically once the edge exist it is counted towards $E(A, L,s)$. It is slightly more complicated to count how many such edges are there. There are $A^d L^{d}$ pairs of vertices at distance $L$, out of which a proportion of $L^{-(\tau-1)(s-\zeta)/\mu}$ pairs of vertices have the given weight, and finally $L^{\alpha(s-\zeta)/\mu}/L^{d\alpha}$ is probability that they also form an edge. Multiplying these together gives the final expression $L^{-d(\alpha-1) + (s-\zeta)(\alpha-\tau+1)}$ after rearrangement. The appearing smaller order correction terms bound logarithmic factors appearing in the integrals. 

The next claim forms the bases of obtaining the lower bound on the transmission times between two far away vertices. 
\begin{claim}\label{claim:admissible}
Let us call $s> 0$ an admissible lower-bound exponent if $\mathbb E[E(A, A, s)]\to  0$ as $A \to \infty$ holds, and call $\mathcal S:=\mathcal S(\alpha, \tau, d, \mu, \zeta)$ the interval of admissible exponents. Then with $\eta_1=\zeta-d(2-\alpha)$, $\eta_2=\zeta+\mu d (\alpha-2)/(\alpha-\tau+1)$, and $\eta_3=\zeta+\mu d-(3-\tau)d$ from Theorem \ref{thm:main-supporting}, any $s> 0$ that satisfies the following inequality is admissible: 
\begin{equation}\label{eq:admissible-s}
    s < \min\Big(\frac{\eta_1}{\mathbf 1\{\alpha\le 2\}}, \frac{\eta_2}{\mathbf 1\{\alpha>2, \tau\le 3\}}, \frac{\eta_3}{\mathbf 1\{\tau\le 3\}}\Big):=s_\star,
\end{equation}
where we set $x/0:=\infty$, which means we do not count $x$ in the minimum\footnote{If the right hand side gives infinity, then all $s>0$ are admissible.}.

Let $s^\star:=\sup\{s: s \in \mathcal S\}$. Then for all $\tau\neq 3, \alpha\neq 2$ it holds that $\eta_\star=\min(1, s_\star)$, where $\eta_\star$ is given in Definition \ref{def:lower-exponent} is the optimal spreading exponent of the upper bound Theorem \ref{thm:main-supporting}. 
Further, on the line $\alpha=2$, we have that  $s_\star=\min(\eta_1, \eta_3/\mathbf1\{\tau\le 3\})$,
while on the line $\tau=3$ we have that $s_\star=\min(\eta_1/\mathbf1\{\alpha\le 2\},  \eta_2/\mathbf1\{\alpha>2\}, \eta_3)$. 
\end{claim}
\begin{proof}
The proof is purely analytic. We start by proving that if $\tau>3, \alpha>2$, then any $s>0$ is admissible. This is in line with the statement as in this case $s_\star=\infty$.  Indeed, we may assume setting $s>\mu d + \zeta$ as increasing $s$ always increases $\E[E[A,A,s]]$ because then more edges are counted towards the expectation. Then we can use \eqref{eq:eal-lower-bound}:  
\begin{equation}\label{eq:aas-lower-bound-simplified}
\mathbb E[E(A, A, s)] \le c_2\Big( A^{d(2-\alpha)} + A^{d(3-\tau)}\Big),
\end{equation}
and the right hand side tends to zero when $\alpha>2$ and $\tau>3$. 

Note that for $s$ to be admissible it is sufficient if, after setting $L=A$, the total exponent of $A$ in each term in either \eqref{eq:no_long_cheap_edge-1} or the relevant row of \eqref{eq:no_long_cheap_edge-2} are negative. We start with $s\le \zeta$. All exponents are negative if:
\begin{equation}\label{eq:no_long_cheap_edge-1-simplified}
    \begin{aligned}
        s&<\min(\zeta - d(2-\alpha), \zeta +\mu d - d(3-\tau)) -d\tilde \varepsilon \quad \text{if } s\le \zeta. 
    \end{aligned}
\end{equation}
Recall that $\tilde \varepsilon$ was arbitrarily small (to bound logarithmic factors), so we may omit it due to the strict inequality in \eqref{eq:no_long_cheap_edge-1-simplified}. The first exponent is $\eta_1$. If $\alpha \le 2$ then $s<\eta_1$ follows directly from the first term in \eqref{eq:admissible-s}, and for $
\alpha >2$ it follows from $s\le\zeta < \eta_1$. 
The second exponent is $\eta_3$, and for $\tau \le 3$ the condition $s<\eta_3$ follows directly from \eqref{eq:admissible-s}, while for $\tau>3$ it follows from $s\le\zeta < \eta_3$. 

We now analyse the case $s\in(\zeta, \zeta+d\mu]$.
If $\alpha\le 2$ then $s\le \eta_1 < \zeta$, so this case is void. Therefore we may assume $\alpha >2$. Note that the first term $A^{d(2-\alpha)+\tilde \varepsilon(s-\zeta)/\mu}$ in \eqref{eq:no_long_cheap_edge-2} then automatically tends to zero if $\tilde\varepsilon$ is sufficiently small. 
The remaining conditions can be summarized as follows for $s\in(\zeta, \zeta+\mu d]$, when choosing both $\varepsilon$ and $\tilde \varepsilon$ arbitrarily small:
\begin{equation}\label{eq:other-s}
    \begin{aligned}
        s&< \min\Big(\zeta + d\mu -d(3-\tau), \zeta + d\mu (\alpha-2)/(\alpha+1-\tau)\Big)
    \end{aligned}
\end{equation}
The two terms on the right hand side are $\eta_3$ and $\eta_2$. If $\alpha\le\tau-1$ then $\tau \ge \alpha+1 >3$, and hence $s<\eta_3$ is satisfied by \eqref{eq:admissible-s}. For $\alpha>\tau-1$, the first condition $s<\eta_3$ follows directly from~\eqref{eq:admissible-s} if $\tau\le 3$, and for $\tau > 3$ it follows from $s\le \zeta+d\mu <\eta_3$. The second condition $s<\eta_2$ again follows immediately from \eqref{eq:admissible-s} if $\tau\le 3$ because we are in the case $\alpha >2$. For $\tau >3$ we use $(\alpha-2)/(\alpha+1-\tau) > 1$, and hence $s \le \zeta+d\mu < \eta_3$. 
This finishes the proof that all $s$ satisfying \eqref{eq:admissible-s} are admissible.

We now explain the relation between $
\eta_\star$ and $s_\star$ in \eqref{eq:eta-star}.
As 
\begin{equation}
\eta_\star:=\min\{\eta_1^{\mathbf 1\{\alpha<2,\, \zeta/d\le 2-\alpha+1/d\}},\eta_2^{\mathbf 1\{\alpha>2,\, \tau<3, \,\mu \le  (1-\zeta)(1/d + (3-\tau)/(d(\alpha-2))\}}, \eta_3^{\mathbf 1\{\tau<3,\, \mu+\zeta/d \le 3-\tau+1/d\}}\},
\end{equation}
we see that $s_\star$ and $\eta_\star$ both combine the minimum of the three possible growth exponents $\eta_1, \eta_2, \eta_3$ on domains where they exist, while they truncate them slightly differently. 
Namely, $\eta_\star$ takes the first exponent $\eta_1$ into account in the minimum if the corresponding edges in the constructive proof of the upper bound exist in the graph ($\alpha<2$) and the construction gives at-most-linear transmission path, i.e., $\eta_1\le 1$, which holds exactly if $\zeta/d\le 2-\alpha + 1/d$.  Similarly, $\eta_\star$ takes $\eta_2$ into account in the minimum if the edges of the construction leading to this upper bound exist, i.e., if $\alpha>2$ and $\tau< 3$, and additionally if $\eta_2\le 1$, corresponding to the inequality on $\mu$. Finally $\eta_\star$ takes $\eta_3$ into account in the minimum if the edges of the corresponding construction exist, i.e., if  $\tau< 3$, and additionally if $\eta_3\le 1$, corresponding to the inequality $\mu+\zeta/d\le 3-\tau+1/d$. As we can always construct paths of linear costs in dimension $d\ge 2$ and paths of cost $\|x\|^{1+o(1)}$ in dimension $1$, $1$ is always a valid upper bound exponent, so these exponents are all truncated if they take values higher than $1$ or if the corresponding long and low-cost edges do not exist in the graph.

At the same time, $s_\star$ is not truncated at $1$ and can take values higher than $1$ as well. On $\alpha\le 2$, it takes $\eta_1$ into account in the minimum, on $\alpha>2,\tau\le 3$ it takes $\eta_2$ into account and on $\tau\le 3$ it also takes $\eta_3$ into account (the first two domains do not overlap, while the third domain overlaps with the other two domains). 
So unless $\alpha=2$ or $\tau=3$, it holds that $\eta_\star=\min(s_\star, 1)$. 
On the line $\alpha=2$, we have that  $s_\star=\min(\eta_1, \eta_3/\mathbf1\{\tau\le 3\})$,
while on the line $\tau=3$ we have that $s_\star=\min(\eta_1/\mathbf1\{\alpha\le 2\}, \eta_2/\mathbf1\{\alpha>2\}, \eta_3)$. 
The lower bound is valid on a larger domain. 
\end{proof}
\begin{proof}[Sketch proof of Theorem \ref{thm:main-lower}]
We aim to show that every path from a vertex $0$ to $x$ incurs large cost. The key input is Lemma \ref{lem:no_long_cheap_edge}, which quantifies the heuristic that “most long edges are expensive.” With this in hand the renormalisation in \cite{komjathy2024polynomial} carries through.
We sketch the key steps. In the multiscale analysis, on scale $k$ we consider boxes of side-length $A_k$ with $A_k=A_1 (k!)^2$. We partition the box with scale-$(k-1)$ boxes (of side-length $A_{k-1}$ and then we also  slightly shift these partitions along each axes to obtain $3^d$ slightly shifted partitions. All these boxes are called the child-subboxes of $Q_k$. We call a box with side-length $A_k$ $\eta$-good if all but $3^d$ of its subboxes are good and if the large box with side-length $A_k$ does not contain any edges of length \emph{longer} than $L_k:=A_k/(100k^2)$ and transmission cost at most $L_k^\eta$, and the same is true for slightly shifted versions of this box along each of the axes. (The slight shifts are technical but necessary for the proof), see \cite[Definition 2.3]{komjathy2024polynomial}. By replacing \cite[Lemma 2.2]{komjathy2024polynomial} by Lemma \ref{lem:no_long_cheap_edge} here,  as $\mathbb E[A_k,L_k,\eta]$ tends to zero whenever $\eta<s_\star$ if $s_\star<1$ and with $\eta=1$ if $s_\star>1$, we can also show the box $Q_k$ surrounding the origin is $\eta$-good for all sufficiently large enough scale $k$ (almost surely), following the proof of \cite[Proposition 2.4]{komjathy2024polynomial}. The initialisation -- i.e., that box $Q_{k_0}$ is good, carries through word by word as the proof estimates the minimal transmission cost over all edges in a box and it shows that it is  at least some small constant $u$ whp, even when we do not use any penalties on the edges. 
This ensures that the multi-scale analysis is also initialised well in our case where spatial penalisation occurs. 

We then prove, deterministically and by induction on scale, that once a box is good, the transmission cost between any two vertices in the box whose Euclidean distance is linear in the box size is large, see \cite[Proposition 2.6]{komjathy2024polynomial}. Here “large” means either linear, or polynomial with exponent $\eta<\min(s_\star,1)$, in the Euclidean distance. The polynomial regime appears exactly when $s_\star\le 1$ and we obtain linear lower bound when $s_\star>1$.

The key to this proposition is the observation that if we fix two sufficiently separated vertices $u,v$ in a good box $Q_k$, then any path $\pi_{u,v}$ between them must either (i) use a long—and hence expensive—edge, or (ii) contain many long, disjoint subsegments lying inside good child boxes of $Q_k$, each contributing a controlled cost by the induction hypothesis. For any segment $\pi_s$ of the path, we call $\mathcal D(\pi_s)$ the Euclidean distance between the endpoints of the segment. 
In the strict polynomial case $s_\star\le 1$, we bound the total contribution of  the disjoint subsequents $\pi_{s}, s\in I$ lying in good child-boxes as follows, under the induction hypothesis that  for a given $\eta<s_\star$ $\mathrm{Cost}(\pi_s)\ge C_k\mathcal D(\pi_s)^{\eta}$ for some constant $C_k$ that depends on the scale $k$ but it converges to a constant $C$ as $k$ tends to infinity. 
\[ 
\mathrm{Cost}(\pi_{u,v}) \ge \sum_{s\in I} \mathrm{Cost}(\pi_s) \ge C_k \sum_{s\in I} \mathcal D(\pi_s)^{\eta} \ge C_k \Big(\sum_{s\in I} \mathcal D(\pi_s)\Big)^{\eta}.
\]
Here, the last inequality is sublinearity: as $\eta\le 1$, it holds for $x, y\ge 1$ that $x^{\eta} + y^{\eta}>(x+y)^{\eta}$.
The proof is finished by proving that the total length of good segments is a constant times the distance that  $\pi_{u,v}$ bridges, i.e., $\|u-v\|$. For more details we refer to \cite{komjathy2024polynomial}.
\end{proof}
   Finally, we provide the proof of Lemma \ref{lem:no_long_cheap_edge} which is the most technical part of the lower bound. As it is new compared to \cite{komjathy2024polynomial}, we provide a detailed proof.

\begin{proof}[Proof of Lemma \ref{lem:no_long_cheap_edge}]
Take any $s> 0$ and let $E(A,L,s)$ denote the number of edges in $[-A/2,A/2]^d$ with Euclidean length in the interval $[L, L^{1+\varepsilon}]$ and cost at most $L^{s}$. 
We follow the same structure as for the edge-length distribution in Section \ref{sec:edge-length}.  The rate of transmission along an edge of length $r$ with endpoints having weight $W_1, W_2$ is $(W_1W_2)^{-\mu}r^{-\zeta}$, and thus the transmission time can be modeled as $r^{\zeta}(W_1 W_2)^{\mu} \cdot X$ with $X$ an exponential random variable with mean $1$. We can then define
	\begin{equation}\label{eq:lambda_cr}
	    \Lambda_c(r):=\E\Big[(1 \wedge W_xW_y/r^d)^{\alpha} \mathbb P\big(r^{\zeta} (W_1W_2)^\mu X \le L^{s}\big)\Big],
	\end{equation} 
    where the second probability expresses the conditional probability that the transmission cost $T_{xy}$ on the edge is at most $L^{s}$.
As in Section \ref{sec:edge-length}, we write $A_{L_1, L_2}(x)$ for the annulus in $\R^d$ of inner radius $L_1$ and outer radius $L_2$ centered at $x\in \R^d$.	Using conditional expectation on the location of vertices $\mathcal V$, (and in the second row below on their weights) we have	
	\begin{equation}\label{eq:e-to-lambdae-pre2}
    \begin{aligned}
		\mathbb E[E(A,L,s)\mid \mathcal V] &=
		\sum_{\substack{x,y \in \mathcal V\cap  [-A/2,A/2)^d\\\|x-y\| \in [L, L^{1+\varepsilon}]}}
		\E\left[\mathbf{1}_{\{xy \textnormal{ is an edge}\}}\mathbf{1}_{\{T_{xy}\le L^{s}\}}\right] \\
        &=
		\sum_{\substack{x,y \in\mathcal V \cap [-A/2,A/2))^d\\\|x-y\|\in [L, L^{1+\varepsilon}]}}
		\E\left[c\left(1 \wedge \dfrac{W_xW_y}{|x-y|^d}\right) ^\alpha \mathbb P\big(\|x-y\|^{\zeta} (W_1W_2)^\mu X \le L^{s}\big) \right] \\
		&\le  c \sum_{x \in \mathcal V \cap [-A/2,A/2))^d}
		\sum_{y \in \mathcal V \cap [-A/2,A/2)^d\cap A_{L, L^{1+\varepsilon}}(x)}
		\Lambda_c(\|x-y\|)
	\end{aligned}
    \end{equation}
    Here, in the third row we use a weaker restriction on $y$ for values of $x$ closer to the boundary of the box, note that this restriction is translation invariant. We can take now expectation over the Poisson point process: the expectation over $x$ gives a factor $n^d A^d/n^d=A^d$. We can easily use the variable switch $x-y:=y$ then switch also to polar coordinates. 
    \begin{equation}\label{eq:e-to-lambdac}
    \mathbb E[E(A,L,s)]\le A^d \int_{y\in (B_{\sqrt{d}A}\cap A_{L, L^{1+\varepsilon}}(0))} \Lambda_c(\|y\|) \mathrm dy \le  A^d \mathrm{Surf}(d) \int_{r= L}^{L^{1+\varepsilon}} r^{d-1}\Lambda_c(r) \mathrm dr.
    \end{equation}
\textbf{Estimating $\Lambda_c(r)$}. In what follows we analyse the function $\Lambda_c(r)$ using Lemma \ref{lem:product_distribution}. Given that $W_1W_2=z$, we can bound the probability concerning the exponential variable $X$ in \eqref{eq:lambda_cr} by $P(X\le L^{s}r^{-\zeta}z^{-\mu})\le \min(1, L^{s}r^{-\zeta}z^{-\mu})$ as $1-e^{-x}\le \min(1,x)$.
Hence, using the density of $W_1W_2$ from Lemma \ref{lem:product_distribution}, and writing $\ell(z):=(\tau-1)^2\log(z)$,
\begin{equation}\label{eq:lambdac-integralform}
\Lambda_c(r)\le c\int_{z=1}^{\infty} \min(1,z/r^d)^{\alpha} \min(1, L^{s}r^{-\zeta}z^{-\mu})\ell(z)z^{-\tau}\mathrm dz. 
\end{equation}
As we will sweep $s$ through $[0,1]$, we have to distinguish three cases: 

\emph{Case 1: Relative large $r$.} For $r> L^{s/\zeta}$, the second minimum is for all $z\ge 1$ at the second value. In this case we can evaluate the integral by distinguishing where the first minimum is taken:
\begin{equation}\label{eq:start1-lambda-c}
\begin{aligned}
\Lambda_c(r)\mathbf 1_{\{r>L^{s/\zeta}\}}&\le \int_{z=1}^{r^d} c z^\alpha r^{-d\alpha} L^{s}r^{-\zeta}z^{-\mu}\ell(z)z^{-\tau} \mathrm dz + \int_{z=r^d}^{\infty}c L^{s}r^{-\zeta}z^{-\mu}\ell(z)z^{-\tau} \mathrm dz 
\end{aligned}
\end{equation}
We write $I_{l1}, I_{l2}$ for the two integrals. Let us introduce
\begin{equation}
\begin{aligned}
\ell_b(x)&:= \frac{(\tau-1)^2}{x^{1-b}}\int^x z^{-b} \log(z)\mathrm dz=(\tau-1)^2  (\log(x)/(b-1) - 1/(b-1)^2) \text{ for } b>1 \\
\tilde \ell_b(x)&:=\frac{(\tau-1)^2}{x^{b+1}} \int^x z^{b} \log(z)\mathrm dx=(\tau-1)^2  (\log(x)/(b+1) - 1/(b+1)^2) \text{ for } b>-1 \\
\ell_{2\star}(x)&:=(\tau-1)^2\int^x z^{-1}\log (z) \mathrm{d}z= (\tau-1)^2\log(x)^2/2.
\end{aligned}
\end{equation} 
Using this notation, the two integrals in \eqref{eq:start1-lambda-c} can be evaluated:
\begin{equation}
    \begin{aligned}
        I_{l1}&=c(1+o(1))\cdot\begin{cases} 
L^{s}r^{-\zeta-d\alpha} \ell_{\tau+\mu-\alpha}(1) &\text{ if } \alpha<\tau-1+\mu\\
L^{s}r^{-\zeta-d\alpha} \ell_{2\star}(r^d) &\text{ if } \alpha=\tau-1+\mu\\
L^{s} r^{-\zeta-d(\tau-1+\mu)} \tilde \ell_{\alpha-\mu-\tau}(r^d) &\text{ if } \alpha>\tau-1+\mu\\
        \end{cases}\\
        I_{l2}&=c(1+o(1))\cdot L^{s}r^{-\zeta-d(\tau-1+\ mu)} \ell_{\tau+\mu}(r^d)
    \end{aligned}
\end{equation}
Observe that $c\ell_b(r^d), c\tilde \ell_b(r^d),c \ell_2(r^d)$ are all at most $c_1 r^{d\tilde \varepsilon}$ for any small $\tilde \varepsilon$, and some constant $c_1$, uniformly for all $r\ge L$. When we estimate these functions from above by $c_1 r^{d\tilde \varepsilon}$, we have the same bound on the third case of $I_{l1}$ as on $I_{l2}$, and the first two cases are dominant only if $\alpha\le \tau-1+\mu$, and otherwise they are not the dominant term. Thus we can upper bound the integral as  
\begin{equation}\label{eq:lambdac-large-r}
\Lambda_c(r)\mathbf 1_{\{r>L^{u/\zeta}\}} \le c_1 L^{s}r^{-\zeta-d\alpha+d\tilde \varepsilon} + c_1 L^{s}r^{-\zeta-d(\tau-1+\mu)+d\tilde \varepsilon}.
\end{equation}
\emph{Case 2: relative medium range $r$: $r\in [L^{s/(d\mu+\zeta)}, L^{s/\zeta}]$}. As $z$ varies in \eqref{eq:lambdac-integralform}, the second minimum switches values at $z=L^{s/\mu}r^{-\zeta/\mu}$ (and equals $1$ for smaller $z$), while the first minimum switches at $r^d$. In this `relative medium $r$' case the order is $1<L^{s/\mu}r^{-\zeta/\mu}<r^d$.
So we can cut the integral at these values:
\begin{equation}\label{eq:medium-integral-cut}
    \begin{aligned}
        \Lambda_c(r)\mathbf 1_{\{L^{s/(d\mu+\zeta)}\le r\le L^{s/\zeta}\}} &\le \int_{z=1}^{L^{s/\mu}r^{-\zeta/\mu}} c z^{\alpha} r^{-d\alpha} \ell(z)z^{-\tau}\mathrm{d}z\\
        &+ \int_{z=L^{s/\mu}r^{-\zeta/\mu}}^{r^d}c z^{\alpha} r^{-d\alpha} L^{s} r^{-\zeta}z^{-\mu} \ell(z)z^{-\tau}\mathrm{d}z + \int_{z= r^d}^{\infty}c L^{s} r^{-\zeta}z^{-\mu} \ell(z)z^{-\tau}\mathrm{d}z. 
    \end{aligned}
\end{equation}
We denote the three integrals by $I_{m1},I_{m2}, I_{m3}$. 
Then, 
\begin{equation}
\begin{aligned}
    I_{m1}&= c(1+o(1))\cdot 
    \begin{cases} r^{-d\alpha}\ell_{\tau-\alpha}(1)
&\text{ if } \alpha<\tau-1\\
    r^{-d\alpha} \ell_{2\star}(L^{s/\mu} r^{-\zeta/\mu}) &\text{ if } \alpha=\tau-1\\
  r^{-d\alpha-\zeta(\alpha+1-\tau)/\mu} L^{s(\alpha+1-\tau)/\mu} \tilde \ell_{\alpha-\tau}(L^{s/\mu} r^{-\zeta/\mu})&\text{ if } \alpha>\tau-1 
    \end{cases}\\
    I_{m2}&=c(1+o(1))\cdot \begin{cases} r^{-d\alpha-\zeta - \zeta(\alpha+1-\mu-\tau)/\mu} L^{s+s(\alpha+1-\mu-\tau)/\mu} \ell_{\tau+\mu-\alpha}(L^{s/\mu}r^{-\zeta/\mu})
&\text{ if } \alpha<\tau-1+\mu\\
    r^{-d\alpha-\zeta} L^{s} \ell_{2\star}(r^d) &\text{ if } \alpha=\tau-1+\mu\\
  r^{-d\alpha-\zeta} r^{d(\alpha+1-\mu-\tau)} L^{s} \tilde \ell_{\alpha-\mu-\tau}(r^d) &\text{ if } \alpha>\tau-1+\mu 
    \end{cases}\\
    I_{m3}&=c(1+o(1)\cdot r^{-\zeta}r^{-d(\tau-1+\mu)} L^{s} \ell_{\tau+\mu}(r^d).
\end{aligned}
\end{equation}
Some simplification is possible of the exponents appearing in $I_{m2}$:
\begin{equation}
      I_{m2}=c(1+o(1))\cdot \begin{cases} r^{-d\alpha+ \zeta(\tau-1-\alpha)/\mu} L^{-s(\tau-1-\alpha)/\mu} \ell_{\tau+\mu-\alpha}(L^{s/\mu}r^{-\zeta/\mu})
&\text{ if } \alpha<\tau-1+\mu\\
    r^{-d\alpha-\zeta} L^{s} \ell_{2\star}(r^d) &\text{ if } \alpha=\tau-1+\mu\\
  r^{-d(\tau-1+\mu)-\zeta}  L^{s} \tilde \ell_{\alpha-\mu-\tau}(r^d)  &\text{ if } \alpha>\tau-1+\mu 
    \end{cases}\\
\end{equation}
Summing the three integrals, as $\mu>0$, we have $5$ cases based on the value of $\alpha$. We use that we can upper bound all slowly varying functions $c\ell_b(x), c\tilde \ell_b(x),c \ell_{2\star}(x)$ by $c_1 x^{\tilde \varepsilon}$ if their argument is $\gg 1$. We upper bound all constants by the maximum one involved, and write a generic $c_1$ for it. Some simplifications are possible. Using the upper bounds $c_1 x^{\tilde \varepsilon}$ on the slowly varying functions, the bounds on $I_{m1},I_{m2}$ agree when $\alpha>\tau-1$ and $\alpha<\tau-1+\mu$, and also the bounds on $I_{m2}$ and $I_{m3}$ agree if $\alpha>\tau-1+\mu$.
\begin{equation}\label{eq:lambdac-medium-r-00}
    \Lambda_c(r)\mathbf 1_{\{L^{s/(d\mu+\zeta)}\le r\le L^{s/\zeta}\}} \le 
    \begin{cases}
       \begin{aligned}& \tilde T_1:=c_1 r^{-d\alpha+ \zeta(\tau-1-\alpha)/\mu} L^{-s(\tau-1-\alpha)/\mu} L^{\tilde \varepsilon s/\mu}r^{-\tilde \varepsilon \zeta/\mu} \\
       &+c_1 r^{-d\alpha} + c_1 r^{-\zeta-d(\tau-1+\mu)+\tilde \varepsilon d} L^{s} 
       \end{aligned}
       &\text{ if } \alpha<\tau-1\\
 \begin{aligned}& \tilde T_2:=c_1 r^{-d\alpha+ \zeta(\tau-1-\alpha)/\mu} L^{-s(\tau-1-\alpha)/\mu} L^{\tilde \varepsilon s/\mu}r^{-\tilde \varepsilon \zeta/\mu} \\
       &+c_1 r^{-d\alpha} L^{\tilde \varepsilon s/\mu}r^{-\tilde \varepsilon \zeta/\mu} + c_1 r^{-\zeta-d(\tau-1+\mu)+\tilde \varepsilon d} L^{s} \end{aligned}      
         &\text{ if } \alpha=\tau-1\\
     \begin{aligned}& \tilde T_3:=c_1 r^{-d\alpha+ \zeta(\tau-1-\alpha)/\mu} L^{-s(\tau-1-\alpha)/\mu} L^{\tilde \varepsilon s/\mu}r^{-\tilde \varepsilon \zeta/\mu} \\
       &+ c_1 r^{-\zeta-d(\tau-1+\mu)+\tilde \varepsilon d} L^{s}  \end{aligned}       
         & \text{ if } \alpha \in (\tau-1, \tau-1+\mu),\\
        \begin{aligned}& \tilde T_4:= c_1 r^{-d\alpha-\zeta+\tilde \varepsilon d} L^{s}  \\
       &+c_1 r^{-d\alpha-\zeta(\alpha+1-\tau)/\mu} L^{s(\alpha+1-\tau)/\mu} L^{\tilde \varepsilon s/\mu}r^{-\tilde \varepsilon \zeta/\mu} \\
       &+ c_1 r^{-\zeta-d(\tau-1+\mu)+\tilde \varepsilon d} L^{s}\end{aligned}
       &\text{ if } \alpha=\tau-1+\mu\\
    \begin{aligned}& \tilde T_5:=c_1 r^{-d\alpha-\zeta(\alpha+1-\tau)/\mu} L^{s(\alpha+1-\tau)/\mu} L^{\tilde \varepsilon s/\mu}r^{-\tilde \varepsilon \zeta/\mu} \\
       &+ c_1 r^{-\zeta-d(\tau-1+\mu)+\tilde \varepsilon d} L^{u} \end{aligned}      
          & \text{ if } \alpha>\tau-1+\mu
    \end{cases}
\end{equation}
Further, the cases $\alpha=\tau-1$ and $\alpha=\tau-1+\mu$ can be merged with the cases $\alpha<\tau-1$ and $\tau-1<\alpha< \tau-1+\mu$, respectively,  as the exponents of $r$ and $L$ agree, the difference is only in the slowly varying part. When $\tau-1<\alpha\le \tau-1+\mu$ we can use that in this range of $r$, $L^{s/\mu}r^{^{-\zeta/\mu}}<r^d$ per assumption and estimate the terms containing $\tilde \varepsilon$ coming from $I_{m2}$ above by $r^{d\tilde \varepsilon}$. We arrive at the upper bound with three remaining cases:
\begin{equation}\label{eq:lambdac-medium-r}
    \Lambda_c(r)\mathbf 1_{\{L^{s/(d\mu+\zeta)}\le r\le L^{s/\zeta}\}} \le 
    \begin{cases}
       \begin{aligned}& T_1:=c_1 r^{-d\alpha+ \zeta(\tau-1-\alpha)/\mu} L^{-s(\tau-1-\alpha)/\mu} L^{\tilde \varepsilon s/\mu}r^{-\tilde \varepsilon \zeta/\mu} \\
       &+c_1 r^{-d\alpha} L^{\tilde \varepsilon s/\mu}r^{-\tilde \varepsilon \zeta/\mu} + c_1 r^{-\zeta-d(\tau-1+\mu)+\tilde \varepsilon d} L^{s} 
       \end{aligned}
       &\text{ if } \alpha\le \tau-1\\

     \begin{aligned}& T_2:=c_1 r^{-d\alpha+ \zeta(\tau-1-\alpha)/\mu+d\tilde \varepsilon} L^{-s(\tau-1-\alpha)/\mu} \\
       &+ c_1 r^{-\zeta-d(\tau-1+\mu)+\tilde \varepsilon d} L^{s}  \end{aligned}       
         & \text{ if } \alpha \in (\tau-1, \tau-1+\mu],\\
    \begin{aligned}& T_3:=
       c_1 r^{-d\alpha-\zeta(\alpha+1-\tau)/\mu} L^{s(\alpha+1-\tau)/\mu} L^{\tilde \varepsilon s/\mu}r^{-\tilde \varepsilon \zeta/\mu} \\
       &+ c_1 r^{-\zeta-d(\tau-1+\mu)+\tilde \varepsilon d} L^{s} \end{aligned}      
          & \text{ if } \alpha>\tau-1+\mu
    \end{cases}
\end{equation}
\emph{Case 3: relative small $r$: $r<L^{s/(d\mu+\zeta)}$.} In this case $r^d< L^{s/\mu}r^{-\zeta/\mu}$, so in \eqref{eq:lambdac-integralform} the first minimum switches earlier than the second minimum. Recalling from Lemma \ref{lem:product_distribution} that $\ell(z)=(\tau-1)^2\log(z)$,
\begin{equation}\label{eq:lambdac-smallr-upper}
    \begin{aligned}
        \Lambda_c(r)\mathbf 1_{\{r<L^{s/(d\mu+\zeta)}\}} &\le \int_{z=1}^{r^d}c z^{\alpha} r^{-d\alpha} \ell(z)z^{-\tau} \mathrm{d}z\\
        &+ \int_{z=r^d}^{L^{s/\mu}r^{-\zeta/\mu}}c \ell(z) z^{-\tau}\mathrm{d}z + \int_{z= L^{s/\mu}r^{-\zeta/\mu}}^{\infty} cL^{s} r^{-\zeta}z^{-\mu} \ell(z)z^{-\tau}\mathrm{d}z 
    \end{aligned}
\end{equation}
Here only the first integral needs a case distinction. Denoting the three integrals by $I_{s1}, I_{s2}, I_{s3}$, we obtain with the same $\ell_b$ and $\tilde \ell_b$ as before,
\begin{equation}\label{eq:lambdac-smallr-detailed}
\begin{aligned}
    I_{s1}&=c(1+o(1))\cdot
    \begin{cases} 
    r^{-d\alpha} \ell_{\tau-\alpha}(1)& \text{ if } \alpha<\tau-1, \\
    r^{-d\alpha} \ell_{2,\star}(r^d) &\text{ if } \alpha=\tau-1,\\
    r^{-d(\tau-1)} \tilde \ell_{\alpha-\tau}(r^d) &\text{ if } \alpha>\tau-1,\\
    \end{cases}\\
    I_{s2}&=c(1+o(1))\cdot r^{-d(\tau-1)}\ell_{\tau}(r^d)\\
    I_{s3}&=c(1+o(1))\cdot  L^{-s(\tau-1)/\mu} r^{\zeta(\tau-1)/\mu} \tilde \ell_{\mu+\tau}(L^{s/\mu}r^{-\zeta/\mu}),\\
\end{aligned}
\end{equation}
where we already simplified the exponents in $I_{s3}$.
Estimating again slowly varying functions $c\ell_b(x), c\tilde \ell_b(x), c\ell_{2\star}(x)$ from above by $c_1 x^{\tilde \varepsilon}$, we observe that the third case of $I_{s1}$ and $I_{s2}$ can be bounded by the same function, while the first two cases are only dominant if $\alpha\le \tau-1$, otherwise they are negligible compared to $I_{s2}$. So we can use a single formula for $\Lambda_{c}(r)$ in this range:
\begin{equation}\label{eq:lambda_c-smallr}
\begin{aligned}
  \Lambda_c(r)\mathbf 1_{\{r<L^{s/(d\mu+\zeta)}\}} &\le c_1 r^{-d\alpha+d\tilde \varepsilon} + c_1 r^{-d(\tau-1)+d\tilde \varepsilon} + c_1   L^{s(1-\tau)/\mu + \tilde \varepsilon s/\mu} r^{\zeta(\tau-1)/\mu- \tilde \varepsilon \zeta/\mu}
\end{aligned}
\end{equation}
For this case we also compute a lower bound on $\Lambda_c(r)$. Returning to its initial formula in \eqref{eq:lambda_cr} and \eqref{eq:e-to-lambdae-pre2}, we notice that $1-e^{-x}\ge \min(x,1)/2$ is a lower bound on the cumulative distribution function of an exponential random variable. Hence, the direction of the inequality in \eqref{eq:lambdac-integralform} can be reversed if we multiply the right-hand-side by $1/2$. Using that, we arrive to same integrals as in \eqref{eq:lambdac-smallr-detailed}, multiplied by a factor $1/2$ to give a matching lower bound. We can estimate all slowly varying function from below by a constant $\widehat c_1$ as they are all growing in their arguments. 
With formulas,
\begin{equation}\label{eq:lambdac-smallr-lower-bound}
\Lambda_c(r)\mathbf 1_{\{r<L^{s/(d\mu+\zeta)}\}} \ge \widehat c_1 r^{-d\alpha} + \widehat c_1 r^{-d(\tau-1)} + \widehat c_1   L^{s(1-\tau)/\mu } r^{\zeta(\tau-1)/\mu}
\end{equation}
Having computed $\Lambda_c(r)$ in \eqref{eq:lambdac-large-r}, \eqref{eq:lambdac-medium-r}, \eqref{eq:lambda_c-smallr} for all ranges of $r$, we can now evaluate the integral in \eqref{eq:e-to-lambdac}. Note that the three cases have boundaries $L^{s/(d\mu+
\zeta)}$ and $L^{s/\zeta}$ for $r$, but the integration in \eqref{eq:e-to-lambdac} only starts at $r\ge L$.
We again distinguish three cases for how $L=L^1$ relates to these exponents. We start with the easiest case: 

\emph{Case 1: small $s$: $s\le \zeta$.} In this case $L\ge L^{s/\zeta}$. That is, when we integrate over $r$ in \eqref{eq:e-to-lambdac} then we only see the `relative large $r$' case for $\Lambda_c$ as $\mathbf 1_{\{r\ge L^{s/\zeta}\}}=1$. 
We can use that $\alpha>1$ and $\tau>2$ so both terms in the integral below are integrable and arrive at
\begin{equation}\label{eq:small-u}
\begin{aligned}
\mathbb E[E(A, L, s)] &\le A^d \mathrm{Surf}(d) c_1  L^{s}\int_{r=L}^{L^{1+\varepsilon}} r^{d-1}\big(r^{-\zeta-d\alpha+d\tilde \varepsilon} + r^{-\zeta-d(\tau-1+\mu)+d\tilde \varepsilon}\big)\mathrm{d}r \\
&\le  c_2 A^d \big( L^{s - d(\alpha-1) -\zeta +d\tilde \varepsilon} + L^{s - d(\tau-2+\mu) -\zeta +d\tilde \varepsilon}\Big)
\end{aligned}
\end{equation}
Observe that $\tilde \varepsilon$ was arbitrary small bounding polylogarithmic correction terms, hence by re-labeling we can set its coefficient to $1$. This gives the formula in \eqref{eq:no_long_cheap_edge-1}. The interpretation below the Lemma statement follows by observing where the dominant terms of the integrals are attained. 

\emph{Case 2: medium $s$: $\zeta<s\le d\mu+\zeta$}. In this case $L\in [L^{s/(d\mu+\zeta)}, L^{s/\zeta})$, so we need to integrate over part of the relative medium-$r$ range. We may assume that $\varepsilon$ is small enough so that $L^{1+\varepsilon}<L^{s/\zeta}$, and then we do not need to integrate over the relative large $r$ range. 
 \begin{equation}\label{eq:medium-u}
\begin{aligned}
    \mathbb E[E(A, L, s)] &= A^d \mathrm{Surf}(d) c_1 \int_{r=L}^{L^{1+\varepsilon}} r^{d-1}\Lambda_c(r) \mathrm{d} r=:A^d \mathrm{Surf}_d c_1\tilde I_1.
\end{aligned}
\end{equation}

We need to use the three different cases in \eqref{eq:lambdac-medium-r}. We go through the three cases there one-by-one, starting with $\alpha\le \tau-1$. Now we integrate $r^{d-1}T_1$ (from $L$ to $L^{1+\varepsilon}$). We notice that we need to do another case-distintion whether $\zeta(\tau-1-\alpha)/\mu+d(\alpha-1)$ is positive or negative.  We add a factor $L^{\tilde \varepsilon}$ to cases where we allow $=$ in the case distinctions to compensate for the arising potential log-factors. The integral in \eqref{eq:medium-u} equals
\begin{equation}
    \tilde I_1\mathbf 1_{\{\alpha\le \tau-1\}}=\begin{cases}
\begin{aligned}
     F_{1}&:= L^{-d(\alpha-1)+\zeta(\tau-1-\alpha-\tilde \varepsilon)/\mu} L^{-s(\tau-1-\alpha-\tilde \varepsilon)/\mu}+ \\
     &+
     L^{-d(\alpha-1)-\tilde \varepsilon \zeta/\mu} L^{\tilde \varepsilon s/\mu}  + L^s L^{-\zeta-d(\tau-2+\mu-\tilde \varepsilon)}
     \end{aligned} &\text{ if } \zeta(\tau-1-\alpha)/\mu\le d(\alpha-1)\\
\begin{aligned}
     F_{2}&:=(L^{1+\varepsilon})^{-d(\alpha-1)+\zeta(\tau-1-\alpha-\tilde \varepsilon)/\mu} L^{-s(\tau-1-\alpha-\tilde \varepsilon)/\mu}+\\
     &+
     L^{-d(\alpha-1)-\tilde \varepsilon \zeta/\mu} L^{\tilde \varepsilon s/\mu}  + L^s L^{-\zeta-d(\tau-2+\mu-\tilde \varepsilon)}
     \end{aligned} &\text{ if } \zeta(\tau-1-\alpha)/\mu> d(\alpha-1)\\
     \end{cases} 
\end{equation}

Collecting terms and error terms, 
we arrive to a unified upper bound and there is no need for the extra case distinction. Returning to \eqref{eq:medium-u}, we arrive at, with $c_3:=\max(\zeta(\tau-1-\alpha )/\mu - d(\alpha-1), 0)$
\begin{equation}\label{eq:medium-u-2}
\begin{aligned}
    \mathbb E[E(A, L, s)] \mathbf 1_{\{\alpha\le \tau-1\}} &\le  c_2 A^d \Big( L^{-d(\alpha-1)+\tilde \varepsilon (s-\zeta)/\mu}  + L^{s-\zeta-d(\tau-2+\mu-\tilde \varepsilon)} + L^{-d(\alpha -1)+ (\zeta-s)(\tau-1-\alpha) + \tilde \varepsilon (s-\zeta)/\mu + \varepsilon c_3    } \Big). \\
\end{aligned}
\end{equation}
Observe that $\tilde \varepsilon$ was arbitrary small bounding polylogarithmic correction terms, hence by re-labeling we can set its coefficient to $1$. However, $\varepsilon$ is set in the lemma statement so we keep its coefficient. 
This gives the first row in \eqref{eq:no_long_cheap_edge-2}.
We now move to the next case computing $\tilde I_1$ in \eqref{eq:medium-u} when $\tau-1<\alpha\le \tau-1+\mu$. For this we use $T_2$ in \eqref{eq:lambdac-medium-r}. Observe that both terms of $T_2$ also appeared already in $T_1$ in \eqref{eq:lambdac-medium-r}, only the coefficient of the $\tilde \varepsilon$ in the exponent is different. 
\begin{equation}
    \tilde I_1\mathbf 1_{\{\tau-1<\alpha\le \tau-1+\mu\}}=\begin{cases}
\begin{aligned}
     \tilde F_{1}&:= L^{-d(\alpha-1)+(\zeta-s)(\tau-1-\alpha)/\mu+d\tilde \varepsilon} \\
      &+ L^{s-\zeta-d(\tau-2+\mu-\tilde \varepsilon)}
     \end{aligned} &\text{ if } \zeta(\tau-1-\alpha)/\mu\le d(\alpha-1)\\
\begin{aligned}
     \tilde F_{2}&:=(L^{1+\varepsilon})^{-d(\alpha-1)+(\zeta-s)(\tau-1-\alpha)/\mu+d\tilde \varepsilon}\\
     &+ L^{s-\zeta-d(\tau-2+\mu-\tilde \varepsilon)}
     \end{aligned} &\text{ if } \zeta(\tau-1-\alpha)/\mu> d(\alpha-1)\\
     \end{cases} 
\end{equation}
The disappearance of the middle term compared to $\alpha<\tau-1$ is important, it means that when $\alpha>\tau-1$ then edges on low weight vertices do not play a relevant role.  We can collect the coefficients of $\varepsilon$, estimate it from above, then merge the two cases, and then the final formula becomes (with $c_3:=\max(-d(\alpha-1)+ \zeta(\alpha+1-\tau)/\mu,0)$):
\begin{equation}\label{eq:medium-u-3}
\begin{aligned}
    \mathbb E[E(A, L, s)] \mathbf 1_{\{\tau-1<\alpha\le \tau-1+\mu\}} &= c_2 A^d \Big( L^{s-\zeta-d(\tau-2+\mu-\tilde \varepsilon)}+
     L^{-d(\alpha-1)+(\zeta-s)(\tau-1-\alpha)/\mu+d\tilde \varepsilon + c_3 \varepsilon}\Big). 
\end{aligned}
\end{equation}

Note that the last term is coming from the first term of $\tilde F_1$: this corresponds to edges of length $\approx L$ (as the integral is dominated at $r=L$) and the endpoints having $W_1W_2$  roughly $\approx L^{s/\mu}r^{-\zeta/\mu}\approx L^{(s-\zeta)/\mu}$ as this term is coming from the integral of the first term of $T_2$ in \eqref{eq:lambdac-medium-r} and that corresponds to $z=L^{(s-\zeta)/\mu}$ in $I_{m2}$ corresponding to the second integral in \eqref{eq:medium-integral-cut}. 

Finally we evaluate the last case in the medium range $s$, when $\alpha> \tau-1+\mu$. In this case there is no need for a case distinction as the first term in $r^{d-1}T_3$ with $T_3$ in \eqref{eq:lambdac-medium-r} is always integrable, and we obtain the exact same formula as for the previous case, except with no $c_3\varepsilon$ error in the exponent. 
\begin{equation}\label{eq:medium-u-4}
\begin{aligned}
\tilde  I_1\mathbf 1_{\{\alpha>\tau-1+\mu\}}&= c_1\Big(L^{-d(\alpha-1)+(s-\zeta)(\alpha+1-\tau)/\mu+\tilde \varepsilon\zeta/\mu} +
     L^{s-\zeta-d(\tau-2+\mu-\tilde \varepsilon)}\Big), \\
    \mathbb E[E(A, L, u)] \mathbf 1_{\{\alpha>\tau-1+\mu\}} &= c_2 A^d \Big( L^{s-\zeta-d(\tau-2+\mu-\tilde \varepsilon)} 
    + L^{-d(\alpha-1)+(\zeta-s)(\tau-1-\alpha)/\mu+d\tilde \varepsilon}\Big)
\end{aligned}
\end{equation}
The three cases are summarized in the lemma statement in \eqref{eq:no_long_cheap_edge-2}. 

\emph{Case 3: relative large $s$: $s>  d\mu+\zeta$.} 
Here we show that for relative large $s$, the expected number of edges of length $\approx L$ and transmission cost below $L^{s}$ is increasing when $\tau<3$. This, -- together with a strategy using `short edges' -- implies that one can never achieve a lower bound on the transmission cost that has an exponent above $\min(d\mu+\zeta,1)$.  

For $s> d\mu+\zeta$, it holds that  $L\le L^{s/(d\mu+\zeta)}$, so we need to integrate $r^{d-1}\Lambda_c(r)$ using the `relative small' $r$ range.   
\begin{equation}\label{eq:large-u-start}
\begin{aligned}
    \mathbb E[E(A, L, s)] &\ge  A^d \mathrm{Surf}(d) \widehat c_1 \int_{r=L}^{L^{1+\varepsilon}} r^{d-1}\Lambda_c(r) \mathrm{d} r 
\end{aligned}
\end{equation}
We evaluate the integral using \eqref{eq:lambdac-smallr-lower-bound}:
\begin{equation}
\begin{aligned}
\mathbb E[E(A, L, s)]&\ge\widehat c_1A^d \mathrm{Surf}_d\int_{L}^{L^{1+\varepsilon}} r^{d-1}\Big( r^{-d\alpha} +  r^{-d(\tau-1)} +  L^{-s(\tau-1)/\mu} r^{\zeta(\tau-1)/\mu}\Big) \mathrm d r\\
&\ge A^d  \hat c_2\Big( L^{d(1-\alpha)}
+ L^{-d(\tau-2)} + L^{-s(\tau-1)/\mu}L^{(1+\varepsilon)\cdot (d+\zeta(\tau-1+\tilde \varepsilon)/\mu)}\Big)
\\&=  A^d \hat c_2 \Big( L^{d(1-\alpha)}
+ L^{-d(\tau-2)} + L^{d- (s-\zeta)(\tau-1)/\mu +  c_4\varepsilon}\Big),
\end{aligned}
\end{equation}
where we obtained the last exponent after elementary rearrangements of the last exponent in the middle row and introduced $c_4:=d+\zeta(\tau-1)/\mu$. Tracing back where each of these terms come from, we come to the following: 
the first term $L^{d(1-\alpha)}$ corresponds the edges between vertices with product of weights $z=\Theta(1)$ and typical cost $L^\zeta \ll L^{s}$. The second term $L^{-d(\tau-2)}$ corresponds to vertices carrying weight with product $z=\Theta(L^{d})$, and of typical cost $L^{d\mu +\zeta}\ll L^{s}$. Finally, the last term corresponds to edges between vertices  having product of weight of order $z=\Theta(L^{s/\mu}r^{-\zeta/\mu})=\Theta(L^{(s-\zeta)/\mu})$, and typical cost $z^\mu r^{\zeta}\approx L^{s-\zeta} L^{\zeta}\approx L^{s}$. Somewhat counterintuitively, the last term decreases as $s$ increases. This is due to the fact that here we count edges with end-vertices whose weight increases with $s$, and there are fewer and fewer of them as $s$ increases. 
Observe that the corresponding upper bound with $c\tilde \eps$ additional terms in each exponent is also valid using the upper bound \eqref{eq:lambdac-smallr-upper} instead of \eqref{eq:lambdac-smallr-lower-bound}.
This finishes the proof of the lemma.
\end{proof}

\subsection{Empirical studies}\label{sec:empirical-studies}

\subsubsection{Construction of the Gowalla Dataset}\label{sec:construction-Gowalla}
The Gowalla dataset is a publicly available network dataset \cite{gowalla} that consists of a total of $196,591$ users and $950,327$ links connecting users that represent friendships. For a subset $n = 107,092$ of these users  geographical log-in data is available in longitude/latitude format. Typically, users use multiple login locations to access the social network. We construct a graph on $n=107,092$ vertices where each vertex represents a single user with available location data. To determine a unique location for each vertex, we identify for each user the modal set of coordinates (rounded to the nearest $.25$), then choose the most common login location from within the corresponding $.25\times .25$ longitude-latitude box, breaking all ties uniformly at random. This method is consistent with the method in \cite{gowalla}, where an accuracy of 85\% for approximate home location is estimated. We then add the links representing the friendships between each vertex pair. 
With this procedure we obtain a graph on $n=107,092$ vertices with $456,830$ many edges, giving an average degree of $8.53$.

\subsubsection{Generating Synthetic Geometric Inhomogeneous Random Graphs}\label{sec:generating-GIRG}
Geometric Inhomogeneous Random Graphs are synthetic network models, see Definition \ref{def:GIRG} above. There is an available open source C++ package that provides the code for a fast generation of these graphs~\cite{girgs-c}, and open source bindings to allow its use from within Python~\cite{GIRG_sampling_github}. 
However, it uses a slightly different version of the connection probabilities compared to~\eqref{eq:girg-connection} in two key ways. First, it works on the unit torus $[0,1)^2$ rather than the scaled torus $[0, \sqrt{n})^2$. These two models would be equivalent up to a factor of $n$ in the $W_uW_v/||u-v||^d$ term in~\eqref{eq:girg-connection}, but unfortunately the code of~\cite{girgs-c} scales by a factor of the total weight instead (which is equivalent up to a random constant factor). Second, in place of $||u-v||^d$ it uses the $L^\infty$ metric on the torus --- that is, for two nodes $u = (x_1, y_1)$ and $v = (x_2, y_2)$ it computes 
$\|(u,v)\|_{\mathbb{T},\infty}:=\max\{\min\{|x_1-x_2|, 1- |x_1-x_2|\},\min\{ |y_1-y_2|, 1-|y_1-y_2|\}\}$. Thus the overall connection probability of an edge $\{u,v\}$ on vertex set $V$ is given by
\[
    p_{uv}' = \min\Big\{\frac{W_uW_v}{\|(u,v)\|_{\mathbb{T},\infty}\sum_{v \in V}W_v}\Big\}.
\]
Both of these were standard choices in the early development of GIRGs when this code was written, but they are now less standard and in particular do not align with~\cite{komjathy2023four, komjathy2024polynomial} on which our work is based. As such we have modified the code accordingly to sample the GIRG on $[0,\sqrt{n})^2$ with connection probability matching~\eqref{eq:girg-connection}. It is sufficiently optimised that this was quite subtle work, but we provide unit tests using the DKW inequality~\cite{dkw-inequality} which indicate small GIRGs are being sampled according to approximately the correct distribution. 

\subsubsection{Synthetic Simulation of Contact-Dependent Epidemics and the Epidemic Curves}\label{sec:simulation-epidemic-curves}
 We take as input a graph with geometric data, that is, each vertex is assigned a location. To simulate the infection process, for each link we draw independent and identically distributed exponential random variables $Y_{uv} \sim \text{Exp}(1)$ between nodes $u,v$.
Then we determine the transmission cost/time across each link by setting $T_{uv}:=Y_{uv} (\deg(u)\deg(v))^{\mu}(\|u-v\|_{\mathbb{T},2} \vee 1)^\zeta$, 
as in \cite[Eq. (1)]{benjert25degree} of the main text for a given $\mu, \zeta$. 

Fixing $T_{u,v}$ for each edge determines for each vertex $v$ the minimal transmission cost-weighted distance $d_{\mathcal C}(u,v)$ from the initially infected node $u_0$ that we determine using Dijkstra's algorithm.
We compute for each node $v$ its infection time $d_{\mathcal C}(u_0, v)$, and we the  order the nodes according to their infection times. We then set $I(t):=\{v: d_{\mathcal C}(u_0,v) \le t\}$ for the set of nodes that are reached by the infection before time $t$. 

For the Gowalla dataset, we use the Haversine distance between users when we calculate the transmission time on any edge, i.e., here we sample iid Exponential $Y_{u,v}$ variables with mean $1$ and then set the transmission cost to be across the (undirected) edge $u,v$: 
$$T_{uv}:=Y_{uv} (\deg(u)\deg(v))^{\mu}\|u-v\|_{\mathrm{H}}^\zeta$$ 
where $\|u-v\|_{\mathrm H}$ is the Haversine distance between the two nodes $u,v$ given their latitutes and longitudes. The initially infected node corresponds to a user located roughly at the centre of Europe near Nuremberg, (at coordinates $49.50, 11.44$). 

For the visualisation of the epidemic curves on a logarithmic scale, we sample the time moments $t_i$ at which $I(t_i)=n_i$ nodes are reached by the epidemic, where $n_i$ sweeps over the five most significant bits of the binary expansion Thus e.g.\ for each $n \ge 4$ we sample at points $\{2^n + \sum_{i=1}^4 a_i2^{n-i}\colon a_1,\dots,a_4 \in \{0,1\}$, with $I_{\text{max}}= 96,774$ being the total number of vertices ever infected; this is the same for all runs.  This creates a single epidemic curve for each run of the epidemic with fixed values of $\mu, \zeta$. Then we resample the whole collection of the exponential random variables $Y_{uv}$ a total of $55$ times. On Figure \cite[Figure 3]{benjert25degree} we show the median value of $t_i$  for each $n_i$. The shaded areas covers the $25$th and $75$th-percentile of the runs. Within the component with $96,774$ vertices, there are $37,024$ vertices in Europe and $46,381$ in the US. We show their relative proportion within $I(t)$ on the inlays of \cite[Figure 3]{benjert25degree}. 
 
\subsubsection{Visualisation of the Diffusion Process}\label{sec:visualisation}
For the visualisation of the infection process on synthetic GIRG networks on \cite[Figure 3]{benjert25degree} in the main text, we take the underlying graph to be a GIRG on the torus $[0, 1000)^2$, using the same sampled graph for each choice of $\mu$ and $\zeta$ with $\tau=2.78$ and $\alpha=1.2$ chosen to match the Gowalla dataset. For each $(\mu, \zeta)$ pair, we run the infection process on this graph, then project the node set from the torus onto $[-500, 500)^2$ with the initial infection site at the centre. We then discretise the projection into $2.5\times 2.5$ boxes, i.e.\ $400$ boxes per side; we chose this number to ensure that most boxes would contain at least one node. To form the heatmap, we colour each box according to the \emph{infection order} of the node-set under matplotlib's reversed plasma colour map. Thus the colour scheme reflects the order at which nodes are reached by the infection, not the actual infection time. As the vertex set is random, many boxes contain multiple nodes, and in this case we take the first infection. Some boxes contain no nodes, and these are rendered as white pixels.

For the visualisation of the infection process on the Gowalla dataset, we follow a similar procedure on the Gowalla graph, with some slight adjustments that we explain now. After running the infection, we project the vertex set onto the plane via an azimuthal equidistant projection centred at the initial infection site in order to preserve distances to this site. We then crop the projection to Europe for visibility (although the infection runs on the whole network). We partition this projection into $400$ boxes per side to match the synthetic network; these boxes are approximately $8\times 8$km. As we should expect given the inhomogeneity of population density, the number of nodes within a given non-empty box varies significantly with median $2$, mean $8.85$, standard deviation $71.6$ and maximum $3059$. As before, we then form the heatmap by colouring each box according to the infection order of the \emph{first infected node} inside it, but we restrict to infections within the cropped projection and normalise by the number of non-empty boxes. 
We use the same method for generating and visualising the shortest infection paths in Figure \ref{fig:geodesics}.

\subsubsection{Animations}\label{sec:animations}
We provide animations of the epidemic spread in a synthetic network and the Gowalla network of Europe. They are in separate SI files, their names starting with animations\_GIRG and animations\_Europe. For the synthetic network we used a synthetic GIRG network of $500\times 500$ nodes placed on a regular grid with parameters $\tau=2.7$ and $\alpha =1.2$. The epidemics have penalisation exponents $\mu =0, \zeta=0$ on the first animations (explosive growth), $\mu=\zeta=1 $ on the second animations (quasi-exponential growth),   $\mu =1, \zeta=2$ on the third animation (polynomial growth), and finally $\mu=1, \zeta=3$ on the fourth animation (pure geometric growth), the same as for the heatmaps in \cite[Figure 3]{benjert25degree}.
For the animated epidemic on Europe, we use the same epidemic parameters as above and also as in \cite[Figure 3]{benjert25degree}.

\subsection{Statistical methods}\label{sec:statistical-methods}
\subsubsection{Power law tail estimates for the degree distribution: recovering $\tau$} \label{sec:stat-method-tau}

We fit a power-law distribution to the empirical degree distribution $n_k/n\sim k^{-\tau_{\mathrm{Gow}}}$ of the Gowalla dataset. To estimate the tail exponent $\tau_{\mathrm{Gow}}$ of this distribution, we use the Hill's estimator \cite{hill1975simple}. This estimator is proven to be consistent on synthetic networks \cite{bhattacharjee2022large} and performs well on empirical network data \cite{voitalov2019scale}. The estimator for any i.i.d.\ sample $x_1, \ldots , x_n$ with corresponding order statistics $x_{(1)} \geq x_{(2)} \geq \ldots \geq x_{(n)}$ is defined by 
\[
\xi_{\kappa,n}^{\text{Hill}} = \frac{1}{\kappa}\sum_{i=1}^{\kappa}\log\left(\frac{x_{(i)}}{x_{(\kappa+1)}}\right).
\]
An open source repository is available accompanying \cite{voitalov2019scale} which estimates the power-law exponent using the Hill's estimator on any input degree sequence. The Hill's estimator is dependent on the choice of $\kappa$ and this code uses a double bootstrap method to determine the optimal choice of $\kappa$ that minimises the asymptotic mean squared error. While the Hill's estimator is proven to be consistent \cite{bhattacharjee2022large}, we do experience finite size effects on synthetic GIRG networks, as the estimator over-estimates the true $\tau$ for smaller networks $n\sim 40k$ with $\hat\tau=3.08$ versus true $\tau=2.78$, while its performance increases as $n$ gets larger and already for $n$ ranging from $40k$ to $4\cdot 10^6$ we only experience an error of $\hat \tau= 2.7 \pm 0.062$ when $\tau=2.7$.

\subsubsection{Tail estimates for the edge-length distribution: recovering $\alpha$} \label{sec:stat-method-alpha}
Estimating the parameter $\alpha$ governing the edge-length distribution is a more challenging task. While the asymptotic edge-length distribution follows a power-law by Corollary \ref{cor:edge-length-limit} with tail exponent $d(\alpha-1)$, finite-size effects are more prevalent for edge-lengths as for degrees. For example, in dimension $2$, having a network size $n\in[10^5, 10^6]$ leads to maximal edge-lengths of order $~100-~1000$, introducing visible finite size effects already for the empirical distribution around $L=50$. With fits to straight lines on log-log scale, this leaves only limited data for the estimation and causes large errors, especially for the relevant range of values of $\alpha < 2$, when the network contains many long edges. 
Indeed, the asymptotic prediction of power law tail gives a very bad fit, see Figure \ref{fig:alpha-estimators}.
Using the results of Theorem \ref{thm:edge-length} which we developed for finite-size GIRGs, we carry out a non-linear least squares regression on the tail of the empirical truncated edge length distribution that we explain now in detail.

\emph{Tail of the empirical truncated edge length distribution.} For a GIRG on $n$ vertices in a box of side length $\sqrt{n}$ in dimension $2$, we first construct the following function. With $L_{+}<\sqrt{n}/2$, we fix an  interval of edge lengths $[L_{-}, L_{+}]$  that we consider our observation window. Then we consider the restricted tail distribution of edge lengths given by
\begin{equation}\label{eq:truncated-tail}
    \bar F_{n,[L_{-}, L_{+}]}(L):=\frac{E_n[L, L_{+}]}{E_n[L_{-}, L_{+}]}:=\frac{\#\{e \in G: L<\|e\|_{\Pi,2} \le L_{+}]\}}{\#\{e: \in G: L_{-}\le \|e\|_{\Pi,2} \le L_{+}\}}.
\end{equation}
that we call the tail of the edge-length distribution in the observation window $[L_{-}, L_{+}]$. 
By our theoretical results in Theorem \ref{thm:edge-length}, $E_n[L, L_{+}]/n$ converges to $c_1(L^{-d(\alpha-1)} - L_{+}^{-d(\alpha-1)})$, assuming that $\alpha<\tau-1$, while the denominator converges to some normalisation constant as the total number of edges is $\Theta(n)$ \cite{bringmann2019geometric}. 
This convergence is strong in practice; see Figure~\ref{fig:alpha-estimators}. Assuming $\alpha$ and $c_1$ are unknown, we can fit the function $
f_{a, b}(L)=b(L^{-d(a-1)} - L_+^{-d(a-1)})$ to the empirical data $\log F_{n, [L_-, L_+]}(L)$ using a non-linear least square regression, which gives $\hat\alpha = \hat{a}$.

Our estimator robustly recovers the true $\alpha$ in synthetic  networks. The approximation of the truncated edge-length distribution by $f$ is weakest when $L$ is small, and so our error drops as $L_-$ increases; on a $10^6$-node synthetic GIRG on $[0, 1000)^2$ with $\tau = 2.78$ and $\alpha = 1.2$, we obtain $\hat\alpha = 1.183$ for $L_- = 10$, $\hat\alpha = 1.191$ for $L_- = 20$, $\hat\alpha = 1.196$ for $L_- = 40$ and $\hat\alpha = 1.198$ for $L_- = 80$. On a synthetic dataset there is little downside to taking $L_-$ large, but on the Gowalla dataset large edge lengths are more subject to distortion due to the presence of oceans so we take $L_- \le 20km$; this allows us to take $L_+$ relatively small and still analyse a reasonable number of edges. Exploratory analysis indicates that the recovered $\hat\alpha$ for the Gowalla network is not very sensitive to the values of $L_-$ and $L_+$ we choose as long as $L_+ - L_-$ is reasonably large and $L_+$ is not too large; all choices of $(L_-, L_+) \in \{(5, 100), (5, 175), (10, 100), (10, 175), (20, 100)\}$ yield $\hat\alpha = 1.2 \pm 0.02$.

\subsubsection{Long-range parameter estimation for the Gowalla dataset}
For the Gowalla dataset we measure the length of edges via the Haversine distance. 
With the estimator $\bar F_{n, [L_-, L_+](L)}$ in \eqref{eq:truncated-tail} at hand, we perform the same estimation on the Gowalla dataset, i.e., using an observation window $[L_-, L_+]$ of edge lengths and subsampling a random $100k$ random edges in the range $L_-, L_+$, then proceed with the nonlinear regression as above for GIRGs. For the Gowalla dataset we observe that the estimator for $\alpha$ is robust over the choice of $L_-$ but depends more sensitively on the choice of $L_+$: when $L_+$ reaches above $500$kms then the fact that for many vertices there are oceans within that range starts to be visible in the estimate, we obtain $\hat \alpha=1.2\pm 0.05$.

\subsubsection{Polynomial growth exponents on epidemic curves} 
We used linear regression to estimate the polynomial growth exponent $\psi$ of the epidemic curves $I(t)\sim t^{d\psi}$ in their polynomial and pure geometric phase. Here, we excluded the initial part of the curve $I(t)$, until the time that $I(t)$ reaches $~60$ individuals as this is the number of users in the same city as the initially infected user. This part of the curve is thus not representative of the later growth as it is a dense sub-community of the network. 

We mention that the estimated growth exponents are admittedly far from the theoretically proven values, and they experience slow convergence. The reason for this is that the theoretical values are obtained in the proof of Theorem \ref{thm:main-supporting} by asymptotical analysis: we set the expected number of long edges forming the fast transmission `bridges' from one area to the other to be roughly $n^{\varepsilon}$ for vanishingly small $\varepsilon>0$. As the number of these bridges tend to infinty in the large network limit, they provide the theoretical results, but in finite networks we see slower growth.  Bridges are present (as long edges are part of the graph) but on the finite network they carry longer transmission times, yielding slower growth on the polynomial epidemic curves on  \cite[Figure 4c]{benjert25degree} compared to the asymptotic result. For the linear growth, the direction is the other way round: our lower bound shows that fast edges are not present in the graph, but in fact we set their number to $n^{-\varepsilon}$, so on finite networks some faster edges may still exist. This means that when theory predicts linear distances and hence $\sim t^2$ epidemic curves, the actual observed exponents in finite graphs are found to be slightly higher \cite[Figure 4d]{benjert25degree}

\subsection{List of supporting technologies}\label{sec:supporting-technologies}

The following open-source libraries were used as part of the codebase.

\begin{itemize}
    \item The Python standard library.
    \item An efficient implementation of the GIRG sampling algorithm of~\cite{bringmann2017sampling} in C++ by Christopher Weyand~\cite{girgs-c}, and Python bindings for it by Tom\'{a}\v{s} Gaven\v{c}iak~\cite{GIRG_sampling_github}, forked as described above.
    \item The graph-tool library~\cite{peixoto_graph-tool_2014} and numpy~\cite{numpy} for other graph manipulation code.
    \item The tail index estimation library of~\cite{voitalov2019scale} for calculating the Hill's estimator of the Gowalla dataset's degree sequence.
    \item The scipy~\cite{SciPy} library for random variable generation and non-linear regression.
    \item The cartopy~\cite{Cartopy} and shapely~\cite{Shapely} libraries for manipulating geographic coordinates.
    \item The matplotlib~\cite{matplotlib-paper} and PILlow~\cite{PILlow} libraries for figure generation.
    \item The dill library from the pathos project~\cite{pathos-github, pathos-paper} for saving and loading data.
    \item The requests~\cite{requests} library for automatically downloading data.
    \item The haversine~\cite{Haversine} library for calculating geographic distances in code written before migrating to cartopy.
    \item The scikit-learn~\cite{scikit-learn} library for linear regression in code written before migrating to scipy.
\end{itemize}

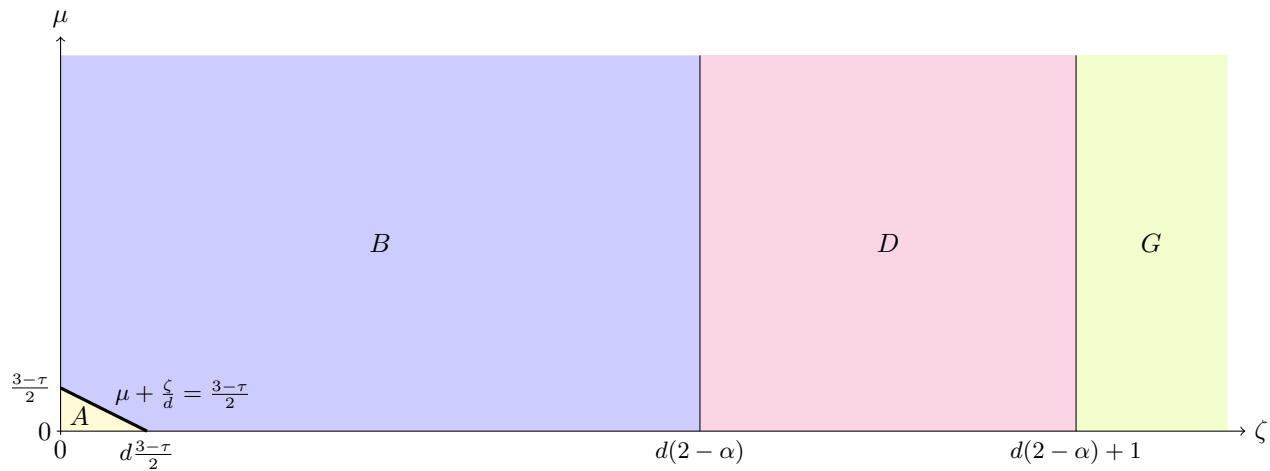
\begin{figure}
    \centering
    \begin{tikzpicture}[scale=5]
        
\fill[color=yellow!20] (0,0)--(0,0.115)--(0.23,0)--(0,0);
\fill[color=blue!20] (0,0.115)--(0.23,0)--(1.7,0)--(1.7,1)--(0,1)--(0,0.115);
\fill[color=magenta!20] (1.7,0)--(1.7,1)--(2.7,1)--(2.7,0)--(1.7,0);
\fill[color=lime!20] (2.7,0)--(2.7,1)--(3.1,1)--(3.1,0)--(2.7,0);

\fill[yellow!20] (0.05,0.04) circle (1pt) node {\color{black}{$A$}}; 
\fill[blue!20] (0.85,0.5) circle (1pt) node {\color{black}{$B$}};
\fill[magenta!20] (2.2,0.5) circle (1pt) node {\color{black}{$D$}};
\fill[lime!20] (2.9,0.5) circle (1pt) node {\color{black}{$G$}};

\draw[->] (-0.01,0)--(3.15,0);
\draw[->] (0,-0.01)--(0,1.05);

\draw[very thick] (0,0.115) -- (0.23,0);
\draw (1.7,0)--(1.7,1);
\draw (2.7,0)--(2.7,1);

\node[right] at (3.15,0) {$\zeta$};
\node[above] at (0,1.05) {$\mu$};
\node[left] at (0,0) {$0$};
\node[below] at (0,0) {$0$};
\node[left] at (0,0.115) {\small$\frac{3-\tau}{2}$};
\node[below] at (0.23,0) {\small$d\frac{3-\tau}{2}$};
\node[below] at (1.7,0) {\small$d(2-\alpha)$};
\node[below] at (2.7,0) {\small$d(2-\alpha)+1$};

\node[right] at (0.12,0.1) {\small$\mu + \frac{\zeta}{d} = \frac{3-\tau}{2}$};

\end{tikzpicture}
    \caption{Enlarged version of Figure 5a, with explicit formulas for the boundaries.}
    \label{fig:enlarged-gowalla-phase-diagram}
\end{figure}

\begin{figure}
    \centering


\begin{tikzpicture}[scale=6]

\fill[color=yellow!20] (0,0)--(0,0.3)--(0.6,0)--(0,0);
\fill[color=blue!20] (0,0.3)--(0.2,0.2)--(0.8,0.2)--(0.8,1)--(0,1)--(0,0.3);
\fill[color=magenta!20] (0.8,0.2)--(1.8,0.2)--(1.8,1)--(0.8,1)--(0.8,0.2);
\fill[color=cyan!20] (0.2,0.2)--(0.8,0.2)--(1.2,0)--(0.6,0)--(0.2,0.2);
\fill[color=red!20] (0.8,0.2)--(1.8,0.2)--(2.2,0)--(1.2,0)--(0.8,0.2);
\fill[color=lime!20] (1.8,0.2)--(2.2,0)--(2.4,0)--(2.4,1)--(1.8,1)--(1.8,0.2);

\fill[yellow!20] (0.2,0.1) circle (1pt) node {\color{black}{A}}; 
\fill[blue!20] (0.4,0.6) circle (1pt) node {\color{black}{B}};
\fill[cyan!20] (0.7,0.1) circle (1pt) node {\color{black}{C}};
\fill[magenta!20] (1.3,0.6) circle (1pt) node {\color{black}{D}};
\fill[red!20] (1.5,0.1) circle (1pt) node {\color{black}{F}};
\fill[lime!20] (2.1,0.6) circle (1pt) node {\color{black}{G}};

\draw[->] (-0.01,0)--(2.45,0);
\draw[->] (0,-0.01)--(0,1.05);

\draw[very thick] (0,0.3) -- (0.6,0);
\draw[dashed] (0.2,0.2)--(1.8,0.2);
\draw (0.8,0.2)--(0.8,1);
\draw (0.8,0.2) -- (1.2,0);
\draw (1.8,0.2)--(1.8,1);
\draw (1.8,0.2)--(2.2,0);

\draw[dotted] (0,0.2)--(0.2,0.2);
\draw[dotted] (0.8,0.2)--(0.8,0);
\draw[dotted] (1.8,0.2)--(1.8,0);

\node[right] at (2.45,0) {$\zeta$};
\node[above] at (0,1.05) {$\mu$};
\node[left] at (0,0) {$0$};
\node[below] at (0,0) {$0$};
\node[left] at (0,0.2) {\small$\alpha-\tau+1$};
\node[left] at (0,0.3) {\small$\frac{3-\tau}{2}$};
\node[below] at (0.6,0) {\small$d\frac{3-\tau}{2}$};
\node[below] at (0.8,0) {\small$d(2-\alpha)$};
\node[below] at (1.2,0) {\small$d(3-\tau)$};
\node[below] at (1.8,0) {\small$d(2-\alpha)+1$};
\node[below] at (2.2,0) {\small$d(3-\tau)+1$};

\node[rotate=-27, above] at (0.46,0.06) {\small$\mu + \frac{\zeta}{d} = \frac{3-\tau}{2}$};
\node[rotate=-27, above] at (1.06,0.06) {\small$\mu + \frac{\zeta}{d} = 3-\tau$};
\node[rotate=-27, above] at (2.02,0.08) {\small$\mu + \frac{\zeta}{d} = 3-\tau+\frac{1}{d}$};

\end{tikzpicture}

    \caption{Enlarged version of Figure 5b, with explicit formulas for the boundaries.}
    \label{fig:enlarged-mu-zeta-small-alpha}
\end{figure}
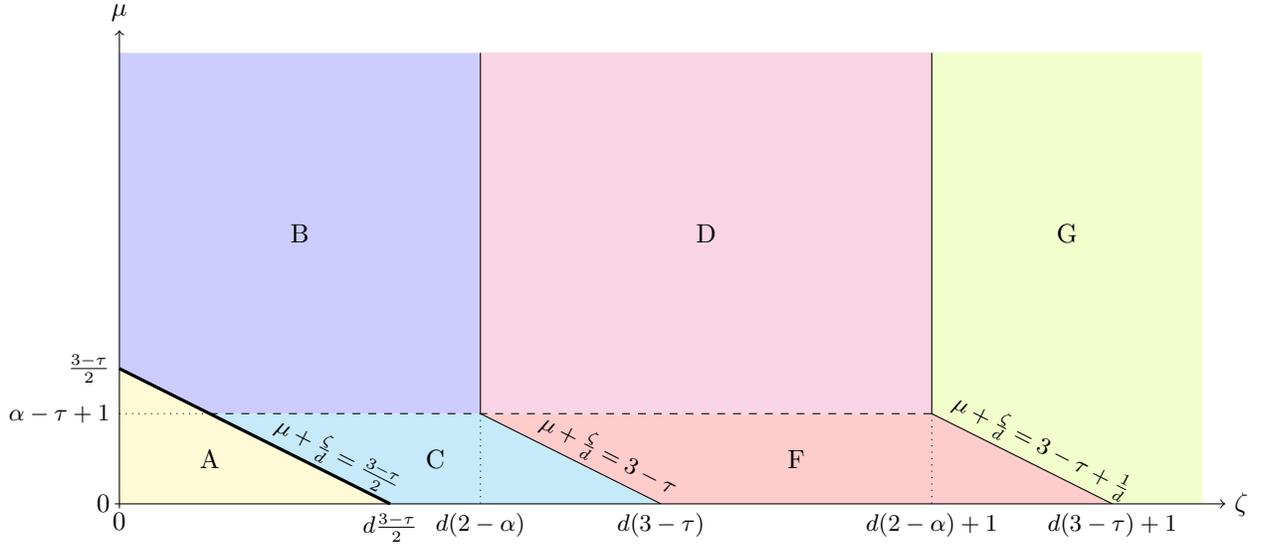

\begin{figure}
    \centering


\begin{tikzpicture}[scale=5]

\fill[color=yellow!20] (0,0)--(0,0.4)--(0.8,0)--(0,0);
\fill[color=cyan!20] (0,0.4)--(0.8,0)--(1.6,0)--(0,0.8)--(0,0.4);
\fill[color=red!20] (0,0.8)--(1.6,0)--(2.6,0)--(0.5,1.05)--(0,1.05)--(0,0.8);
\fill[color=purple!20] (0,1.05)--(0.5,1.05)--(0,2.1)--(0,1.05);
\fill[color=lime!20] (2.6,0)--(0.5,1.05)--(0,2.1)--(0,2.3)--(2.9,2.3)--(2.9,0)--(2.6,0);

\fill[yellow!20] (0.25,0.125) circle (1pt) node {\color{black}{A}}; 
\fill[cyan!20] (0.65,0.325) circle (1pt) node {\color{black}{C}};
\fill[purple!20] (0.18,1.3) circle (1pt) node {\color{black}{E}};
\fill[red!20] (1.1,0.55) circle (1pt) node {\color{black}{F}};
\fill[lime!20] (1.6,1.3) circle (1pt) node {\color{black}{G}};

\draw[->] (-0.01,0)--(2.95,0);
\draw[->] (0,-0.01)--(0,2.35);

\draw[very thick] (0,0.4) -- (0.8,0);
\draw (0,0.8) -- (1.6,0);
\draw[dashed] (0,1.05)--(0.5,1.05);
\draw (2.6,0)--(0.5,1.05);
\draw (0.5,1.05)--(0,2.1);

\node[right] at (2.95,0) {$\zeta$};
\node[above] at (0,2.35) {$\mu$};
\node[left] at (0,0) {$0$};
\node[below] at (0,0) {$0$};
\node[left] at (0,1.05) {$\alpha-\tau+1$};
\node[left] at (0,0.4) {$\frac{3-\tau}{2}$};
\node[below] at (0.8,0) {$d\frac{3-\tau}{2}$};
\node[left] at (0,0.8) {$3-\tau$};
\node[below] at (1.6,0) {$d(3-\tau)$};
\node[below] at (2.6,0) {$d(3-\tau)+1$};
\node[left] at (0,2.1) {$\frac{1}{d}+\frac{3-\tau}{d(\alpha-2)}$};

\node[rotate=-27, above] at (0.4,0.2) {$\mu + \frac{\zeta}{d} = \frac{3-\tau}{2}$};
\node[rotate=-27, above] at (0.8,0.4) {$\mu + \frac{\zeta}{d} = 3-\tau$};
\node[rotate=-27, above] at (1.55,0.525) {$\mu + \frac{\zeta}{d} = 3-\tau+\frac{1}{d}$};
\node[rotate=-63, above] at (0.25,1.575) {$\mu+(\frac{1}{d}+\frac{3-\tau}{d(\alpha-2)})\zeta=\frac{1}{d}+\frac{3-\tau}{d(\alpha-2)}$};

\end{tikzpicture}

    \caption{Enlarged version of Figure 5c, with explicit formulas for the boundaries.}
    \label{fig:enlarged-mu-zeta-large-alpha}
\end{figure}
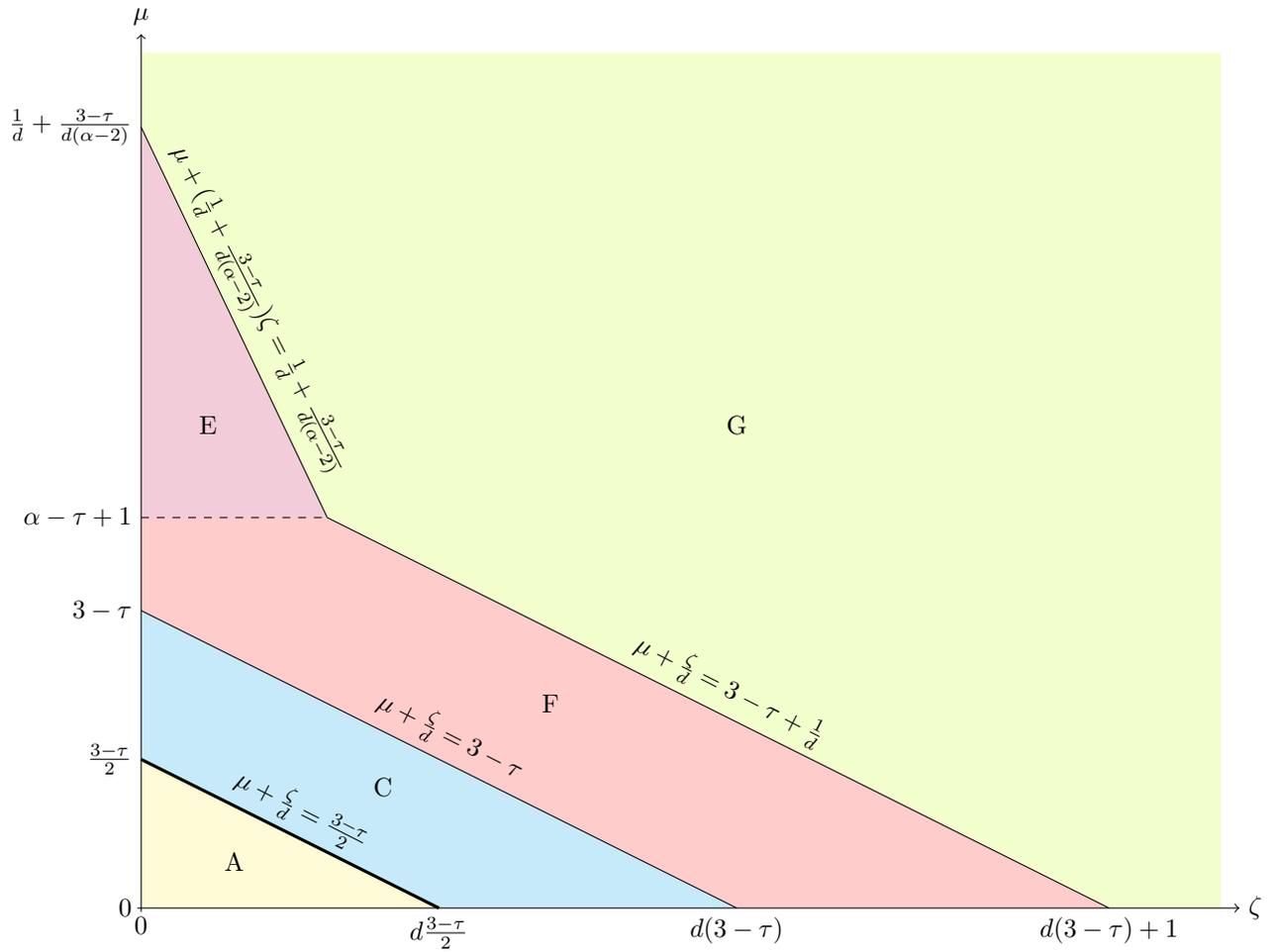

\begin{figure}
    \centering


\begin{tikzpicture}[scale=8]

\fill[color=yellow!20] (1,2)--(2.9,2)--(2.9,2.2)--(1,2.2)--(1,2);
\fill[color=blue!20] (1,2.2)--(1.5,2.2)--(1.9,2.6)--(1.9,3.2)--(1,3.2)--(1,2.2);
\fill[color=magenta!20] (1.9,2.6)--(2,2.7)--(2,3.2)--(1.9,3.2)--(1.9,2.6);
\fill[color=purple!20] (2,3)--(2.15,2.85)--(2,2.7)--(2,3);
\fill[color=cyan!20] (1.9,2.6)--(2.9,2.6)--(2.9,2.2)--(1.5,2.2)--(1.9,2.6);
\fill[color=red!20] (1.9,2.6)--(2.9,2.6)--(2.9,2.85)--(2.15,2.85)--(1.9,2.6);
\fill[color=lime!20] (2.15,2.85)--(2,3)--(2,3.2)--(2.9,3.2)--(2.9,2.85)--(2.15,2.85);

\fill[yellow!20] (1.7,2.1) circle (1pt) node {\color{black}{A}}; 
\fill[blue!20] (1.4,2.4) circle (1pt) node {\color{black}{B}};
\fill[cyan!20] (2.4,2.4) circle (1pt) node {\color{black}{C}};
\fill[magenta!20] (1.95,2.92) circle (1pt) node {\color{black}{D}};
\fill[purple!20] (2.05,2.82) circle (1pt) node {\color{black}{E}};
\fill[red!20] (2.4,2.72) circle (1pt) node {\color{black}{F}};
\fill[lime!20] (2.4,3.1) circle (1pt) node {\color{black}{G}};

\draw[->] (0.99,2)--(2.95,2);
\draw[->] (1,1.99)--(1,3.25);
\draw[very thick] (1,2.2)--(2.9,2.2);
\draw[dashed] (1.5,2.2) -- (2.15,2.85);
\draw (1.9,2.6)--(1.9,3.2);
\draw[dashed] (2,2.7)--(2,3.2);
\draw (2,3)--(2.15,2.85);
\draw (1.9,2.6)--(2.9,2.6);
\draw (2.15,2.85)--(2.9,2.85);
\draw[very thick] (2,3)--(2,3.2);

\draw[dashed] (1,3)--(2.9,3);
\draw[dashed] (2,2)--(2,2.7);
\draw[dotted] (1,2.6)--(1.9,2.6);
\draw[dotted] (1,2.85)--(2.15,2.85);
\draw[dotted] (1.9,2)--(1.9,2.6);

\node[right] at (2.95,2) {$\alpha$};
\node[above] at (1,3.25) {$\tau$};
\node[left] at (1,3) {$3$};
\node[left] at (1,2) {$2$};
\node[below] at (1,2) {$1$};
\node[left] at (1,2.2) {$3-2(\mu+\frac{\zeta}{d})$};
\node[left] at (1,2.6) {$3-(\mu+\frac{\zeta}{d})$};
\node[left] at (1,2.85) {$3+\frac{1}{d}-(\mu+\frac{\zeta}{d})$};
\node[below] at (2,2) {$2$};
\node[below] at (1.9,2) {$2-\frac{\zeta}{d}$};
\node[rotate=45, below] at (1.75,2.45) {$\tau = \alpha+1-\mu$};
\node[right] at (2.1,2.92) {$\tau+(\frac{d\mu}{1-\zeta}-1)\alpha=3+2(\frac{d\mu}{1-\zeta}-1)$};

\end{tikzpicture}

    \caption{Enlarged version of Figure 5d, with explicit formulas for the boundaries.}
    \label{fig:enlarged-alpha-tau}
\end{figure}
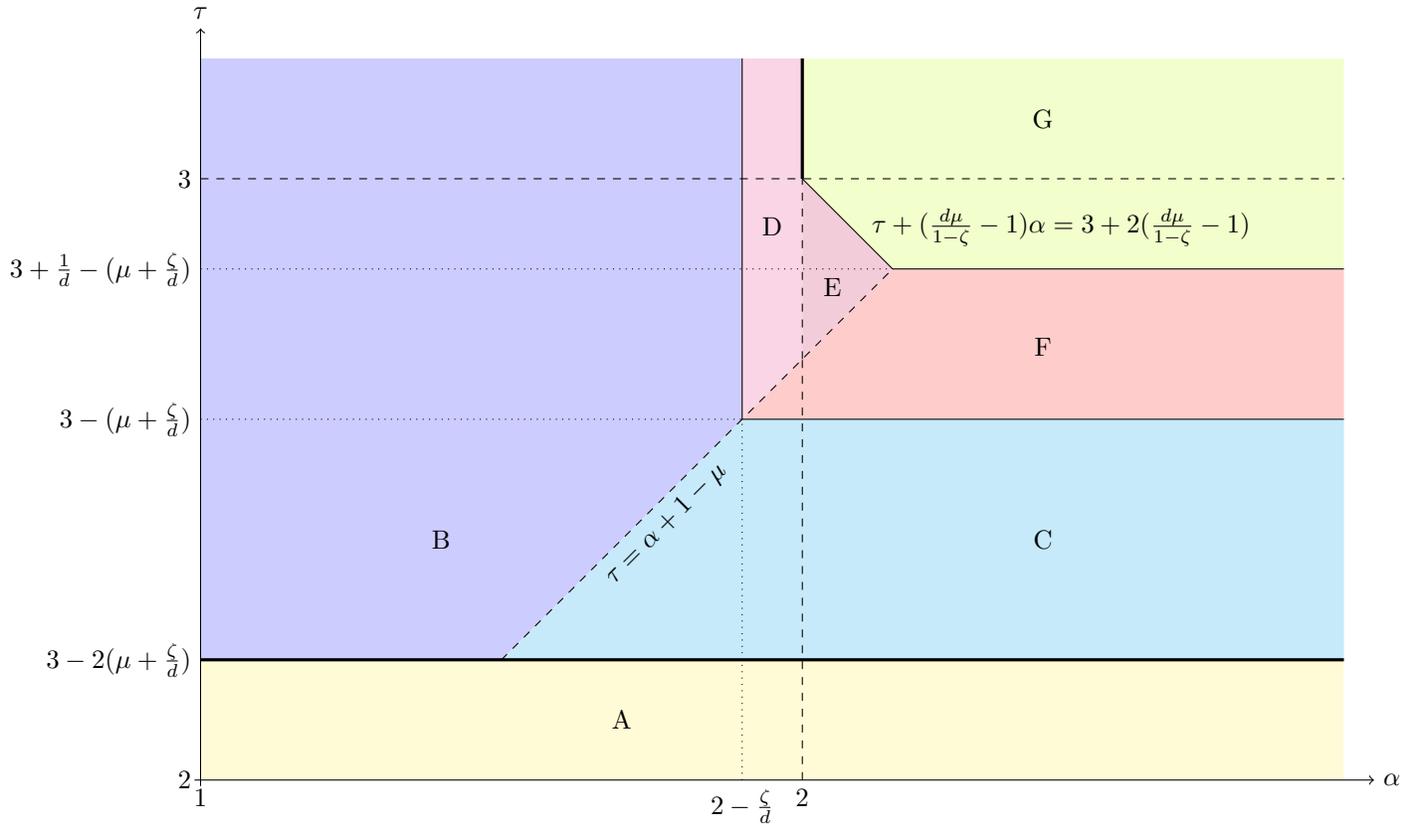

\begin{figure}
    \centering
    \includegraphics[width=
    \textwidth]{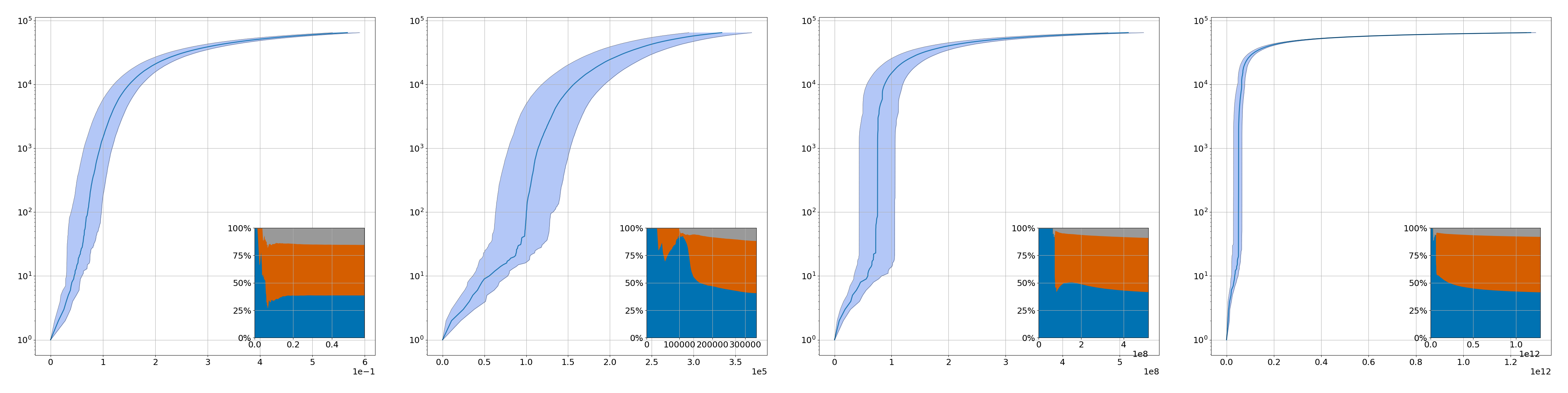}
    \caption{Four epidemic curves $I(t)$ as a function of $t$ on the random reference network obtained from the Gowalla dataset, with the same parameters as on Figure 4 of \cite{benjert25degree}. We rewired edges randomly while maintaining the same degree counts.  We visualise $I(t)$ up to $10^{4.77}$ infected nodes, after which network saturation occurs. The middle curve is the median value of $55$ epidemic runs while the shaded area shows the $25\%-75\%$ quantiles of the runs. In each run only the random link factors $E_{uv}$ are re-sampled. The $y$-axis is on a log-scale with base $10$. The $x$-axis is linear scale on figures (a)-(b), and it is on a log-scale with base $10$ on figures (c)-(d). On a random reference network theory predicts only two quantitative phases: explosive and exponential. Here only fig (a) is predicted to grow explosively with $\mu=\zeta=0$, while figures (b)-(c)-(d) all grow exponentially with $\mu=1$ and $\zeta=1,2,3$ respectively. The increase of $\zeta$ causes a slowdown while maintains the general shape of the curves. 
 The inlays show the proportions of infected European (blue), US (orange), and other (grey) nodes; with $t$ on a log-scale. As predicted by the theory, fast mixing occurs in both phases.}
     \label{fig:epicurves-random-ref}
 \end{figure}

 \begin{figure}
    \centering
     \includegraphics[width=0.45\textwidth]{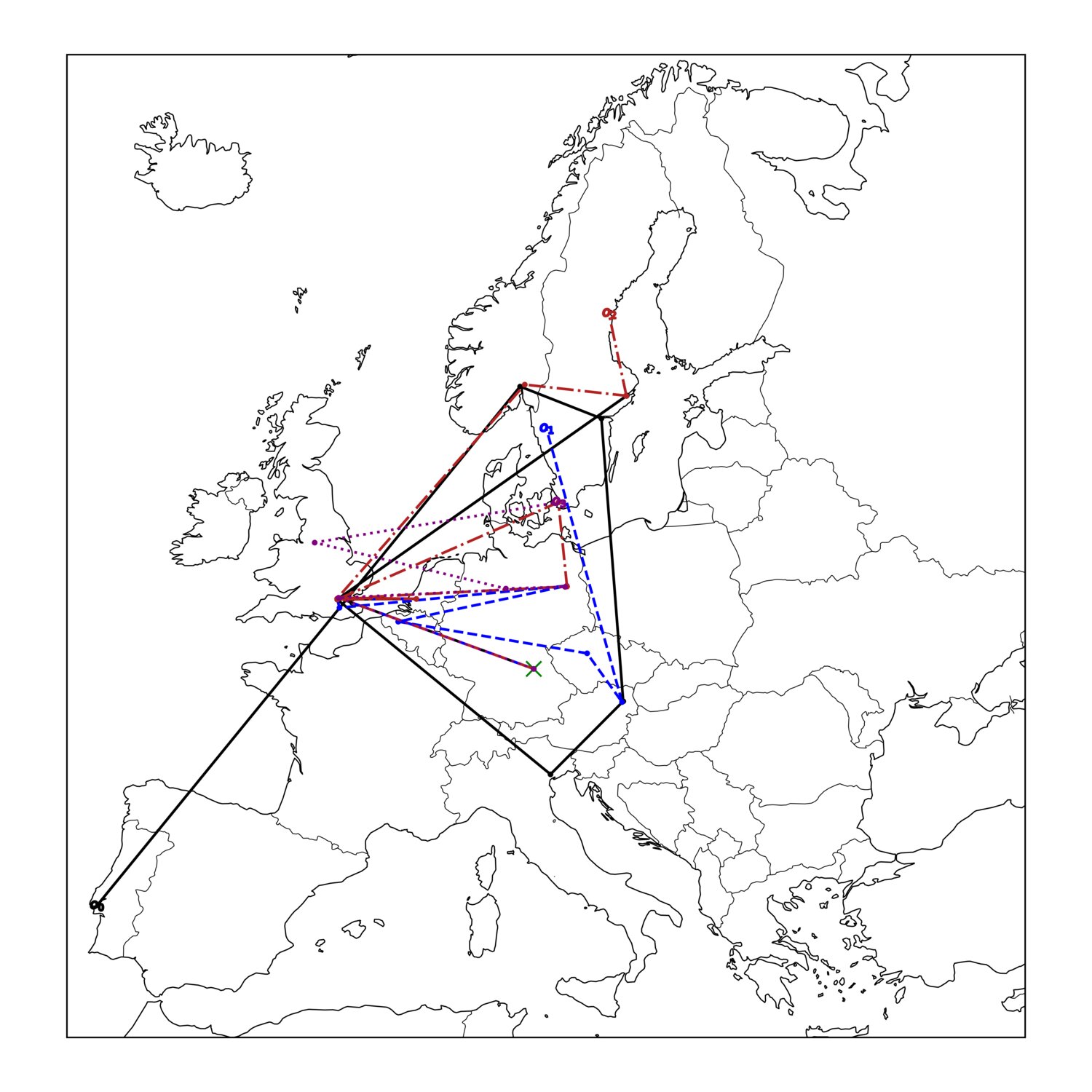}
    \includegraphics[width=0.45\textwidth]{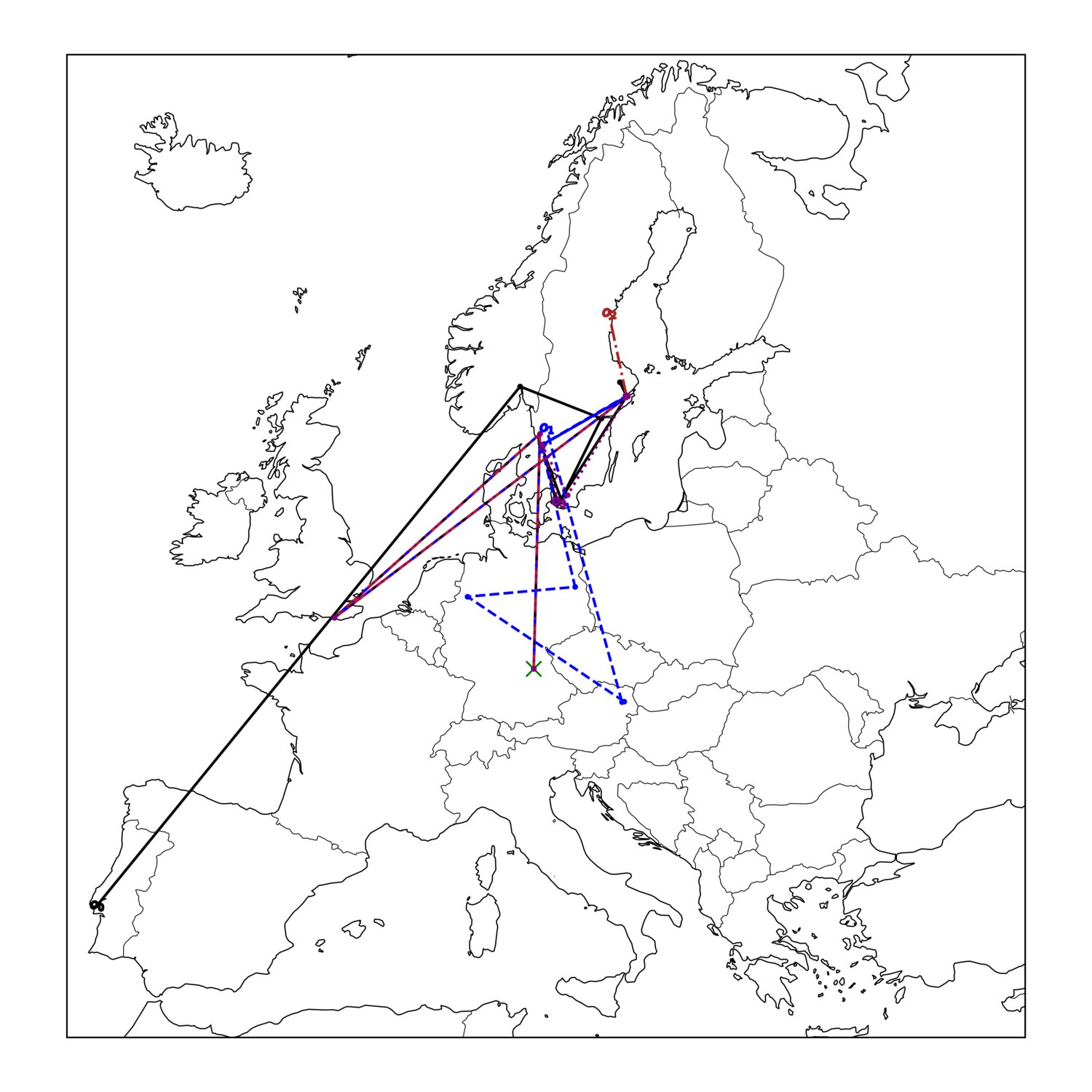}
    \includegraphics[width=0.45\textwidth]{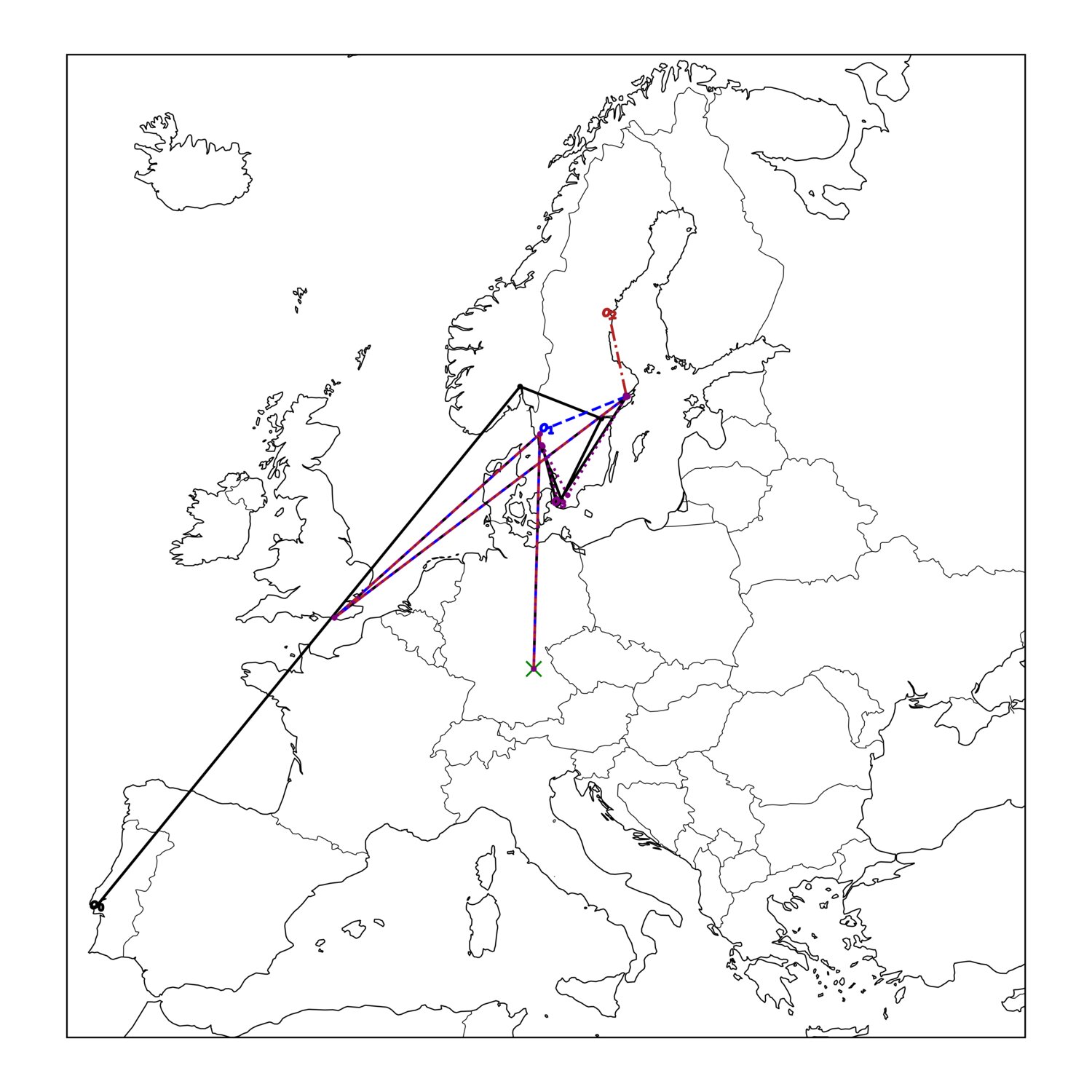}
    \includegraphics[width=0.45\textwidth]{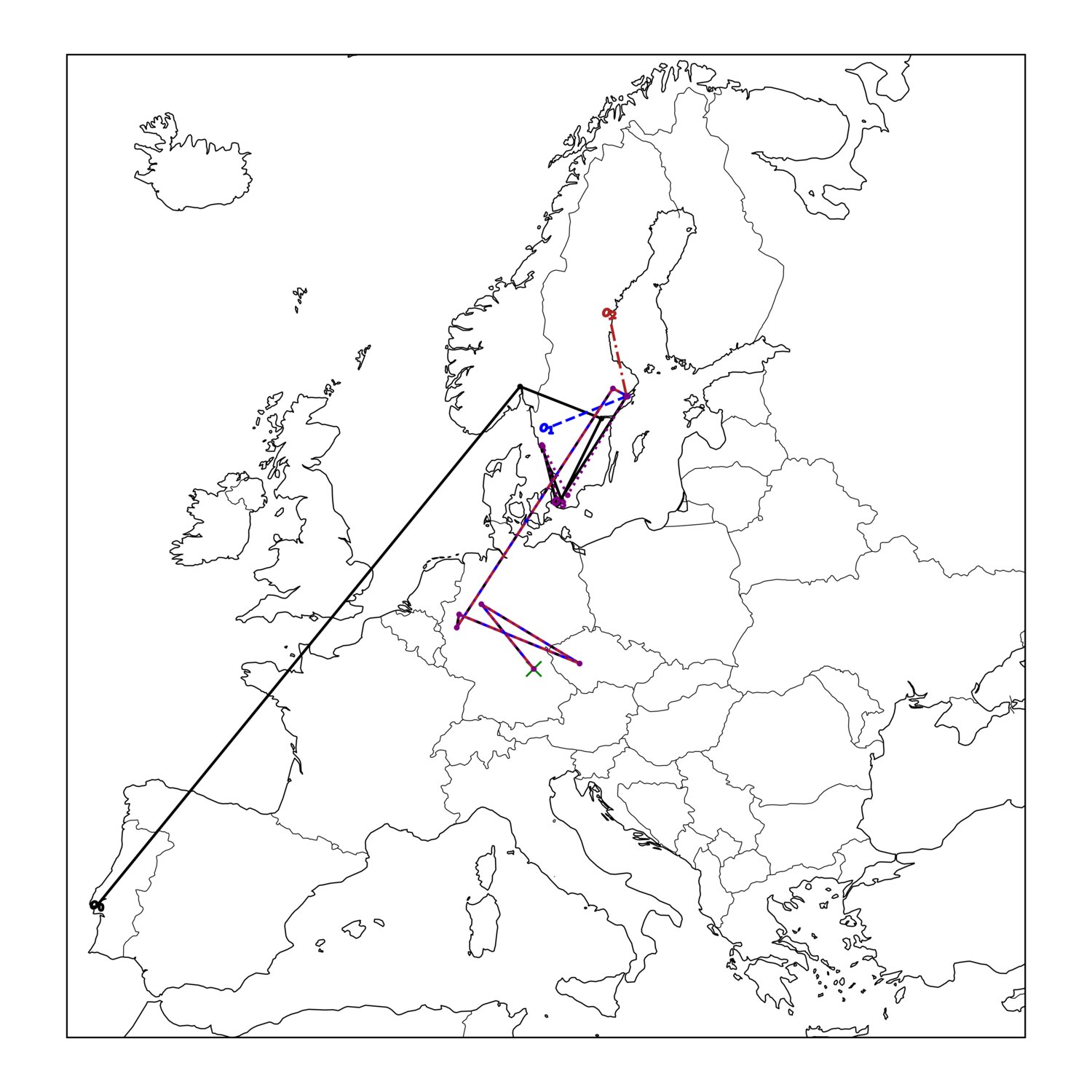}
    \caption{Typical infection paths in the four different phases for the epidemic run on Europe part of the Gowalla network and on synthetic GIRG networks. The epidemics have the same parameters as on Figure 4 \cite{benjert25degree}. The `$x$' indicates the source node while the $o_1, o_2, o_3, o_4$ the location of the four randomly selected users.}
     \label{fig:geodesics}
 \end{figure}

 \begin{figure}
    \centering
     \includegraphics[width=0.45\textwidth]{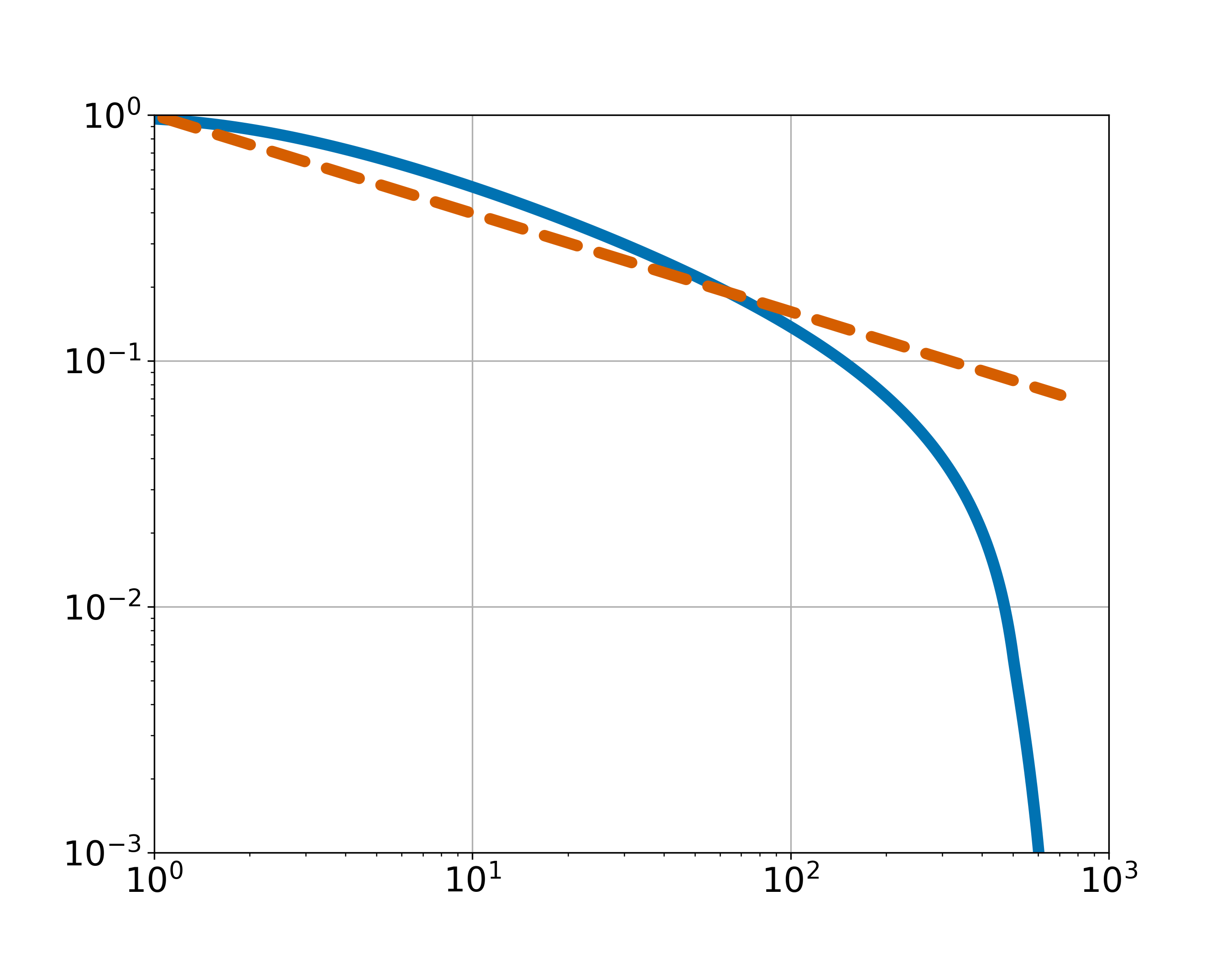}
    \includegraphics[width=0.45\textwidth]{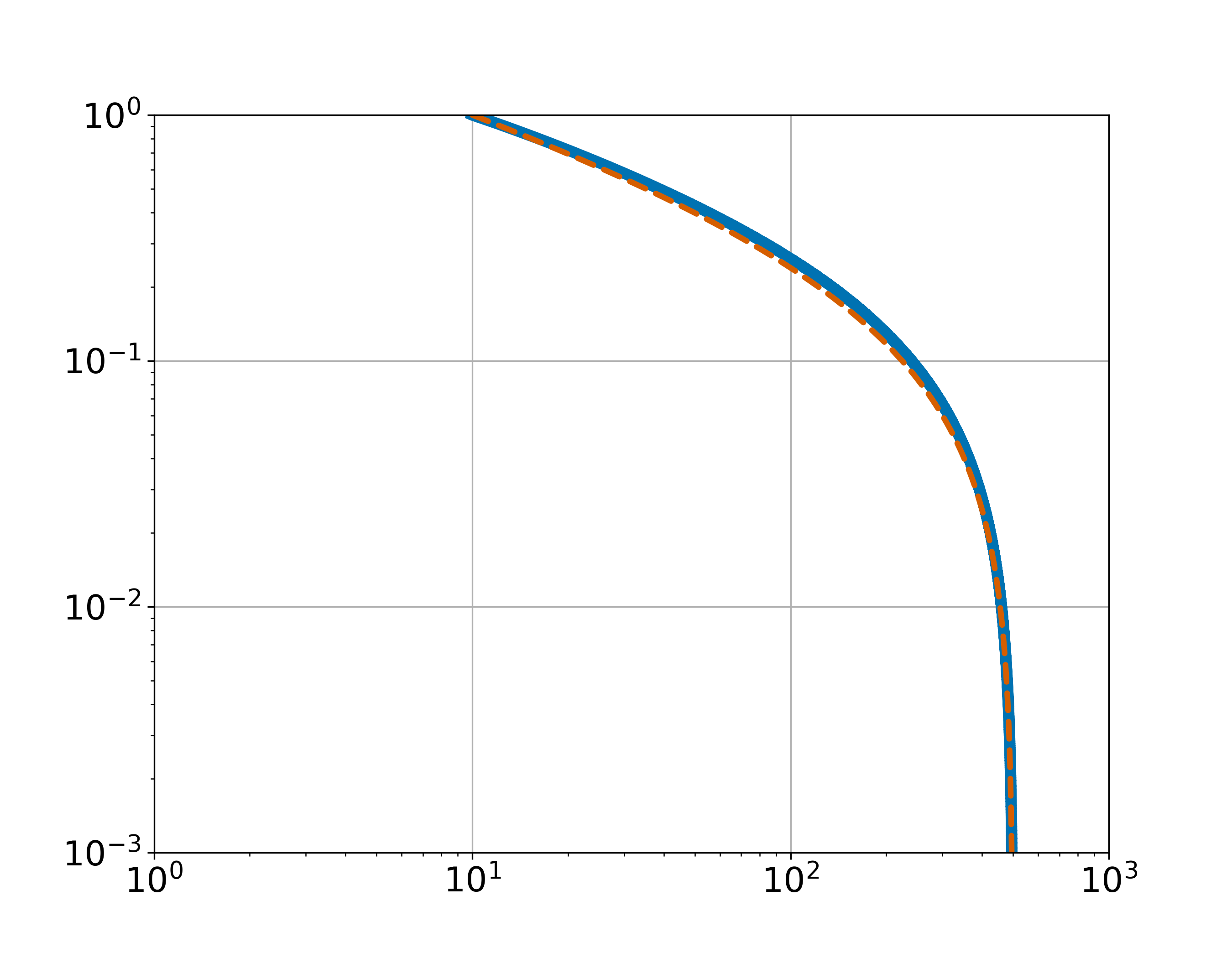}
    \caption{This plot shows the edge-length distributions for the synthetic Gowalla network. Left: The blue solid curve is the empirical proportion of edges with length at least $L$ plotted against $L$; note that the diameter of the torus is roughly $707$. The red dashed curve is the asymptotic proportion predicted by the large network limit. Right: The blue solid curve is the the empirical truncated edge-length distribution with $L_- = 10$ and $L_+ = 500$. The red dashed curve is the distribution predicted by Theorem~\ref{thm:edge-length} (up to a normalising factor). Note that the two curves do not precisely coincide due to the low value of $L_-$.}
     \label{fig:alpha-estimators}
 \end{figure}

\end{document}